\documentclass[12pt, reqno]{article}
\usepackage{etoolbox}
\usepackage{fullpage}
\usepackage[normalem]{ulem}
\usepackage[utf8]{inputenc}
\usepackage[portrait,margin=1in]{geometry} 
\usepackage{mathrsfs,xfrac} 
\usepackage[colorlinks=true,linkcolor=blue,citecolor=blue,urlcolor=blue]{hyperref} 
\usepackage[mathlines]{lineno} 
\usepackage{amsmath}
\usepackage{amssymb,amsthm,amsfonts,amsbsy,latexsym,dsfont,color,booktabs} 
\usepackage[numeric,initials,nobysame]{amsrefs} 
\usepackage{upref,halloweenmath}
\renewcommand{\skull}{\vcenter{\hbox{\scalebox{.5}{$\bigskull$}}}}

\usepackage[utf8]{inputenc}
\usepackage[english]{babel}

\usepackage[textsize=tiny]{todonotes}

\usepackage{enumerate}
\newenvironment{enumeratei}{\begin{enumerate}[\upshape (i)]}{\end{enumerate}}

\usepackage{cleveref}
  \crefname{theorem}{Theorem}{Theorems}
  \crefname{thm}{Theorem}{Theorems}
  \crefname{lemma}{Lemma}{Lemmas}
  \crefname{lem}{Lemma}{Lemmas}
  \crefname{rem}{Remark}{Remarks}
  \crefname{prop}{Proposition}{Propositions}
  \crefname{proposition}{Proposition}{Propositions}
  \crefname{problem}{Problem}{Problems}
\crefname{notation}{Notation}{Notations}
\crefname{claim}{Claim}{Claims}
  \crefname{defn}{Definition}{Definitions}
  \crefname{cor}{Corollary}{Corollaries}
  \crefname{section}{Section}{Sections}
  \crefname{figure}{Figure}{Figures}
  \crefname{exercise}{Exercise}{Exercises}
    \crefname{assumption}{Assumption}{Assumptions}

\newtheorem{thm}{Theorem}[section]

\newtheorem{claim}[thm]{Claim}
\newtheorem{lemma}[thm]{Lemma}

\newtheorem{cor}[thm]{Corollary}

\newtheorem{proposition}[thm]{Proposition}
\newtheorem{defn}[thm]{Definition}

\numberwithin{equation}{section}

\theoremstyle{definition}
\newtheorem{rem}[thm]{Remark}
 
\renewcommand{\leq}{\leqslant} 
\renewcommand{\geq}{\geqslant} 
\renewcommand{\le}{\leqslant} 
\renewcommand{\ge}{\geqslant} 

\newcommand{\ra}{\rangle}
\newcommand{\la}{\langle}

\newcommand{\eps}{\varepsilon}

\newcommand{\norm}[1]{\left\Vert#1\right\Vert}

\newcommand{\ie}{\emph{i.e.,}}

\newcommand{\equald}{\stackrel{\mathrm{d}}{=}}

\let\ga=\alpha \let\gb=\beta \let\gc=\gamma \let\gd=\delta 
    \let\gk=\kappa \let\gl=\lambda           \let\gs=\sigma  
  
 \let\gD=\Delta   
\let\gO=\Omega           
                                
\newcommand{\cA}{\mathcal{A}}\newcommand{\cB}{\mathcal{B}}\newcommand{\cC}{\mathcal{C}}
\newcommand{\cD}{\mathcal{D}}
\newcommand{\cG}{\mathcal{G}}\newcommand{\cH}{\mathcal{H}}
\newcommand{\cJ}{\mathcal{J}}

\newcommand{\cP}{\mathcal{P}}
\newcommand{\cS}{\mathcal{S}}\newcommand{\cU}{\mathcal{U}}
\newcommand{\cV}{\mathcal{V}}
\newcommand{\cZ}{\mathcal{Z}}  

\newcommand{\vd}{\mathbf{d}}

\newcommand{\mv}[1]{\boldsymbol{#1}}

\newcommand{\mvk}{\boldsymbol{k}}
\newcommand{\mvo}{\boldsymbol{o}}

\newcommand{\mvu}{\boldsymbol{u}}
\newcommand{\mvx}{\boldsymbol{x}}\newcommand{\mvy}{\boldsymbol{y}}

\newcommand{\f}[1]{\mathfrak{#1}}

\newcommand{\fI}{\mathfrak{I}}

\newcommand{\fR}{\mathfrak{R}}

\newcommand{\ff}{\mathfrak{f}}

\newcommand{\fr}{\mathfrak{r}}\newcommand{\fs}{\mathfrak{s}}

\newcommand{\fg}{\mathfrak{g}}

\newcommand{\sbeta}{\text{\sout{$\beta$}}}

\newcommand{\bC}{\mathbb{C}}
\newcommand{\bE}{\mathbb{E}}

\newcommand{\bN}{\mathbb{N}}
\newcommand{\bP}{\mathbb{P}}\newcommand{\bR}{\mathbb{R}}
\newcommand{\bT}{\mathbb{T}}

\newcommand{\bZ}{\mathbb{Z}}        

\newcommand{\dC}{\mathds{C}}

\newcommand{\dP}{\mathds{P}}

\newcommand{\dT}{\mathds{T}}

\newcommand{\dZ}{\mathds{Z}} 

\DeclareMathOperator{\E}{\mathds{E}}

\DeclareMathOperator{\var}{Var}

\DeclareMathOperator{\argmin}{argmin}

\renewcommand{\i}{{i\mkern1mu}}

\newcommand{\ttwo}{\ensuremath{\ell^{2}(\mathbb{T}^{d}_{n})}}
\newcommand{\zp}{\ensuremath{\ell^{p}(\mathbb{Z}^{d})}}
\newcommand{\ztwo}{\ensuremath{\ell^{2}(\mathbb{Z}^{d})}}

\newcommand{\Psisph}{\Psi_{\mathsf{sph},N}}
\newcommand{\Psispheps}{\Psi_{\mathsf{sph},N,\varepsilon_N}}
\newcommand{\Phizero}{\Phi_{0{\text{-}}\mathsf{avg}}}
\newcommand{\Gzero}{\mathsf{G}_{0{\text{-}}\mathsf{avg}}}

\begin{document}

\title{Local limit of the focusing discrete NLS}
\author{Kesav Krishnan \thanks{Research supported in part by PIMS pdf fellowship, University of Victoria} \and Gourab Ray \thanks{Research supported by NSERC 50711-57400, University of Victoria }}

\maketitle

\begin{abstract}
    We examine the behavior of a function sampled from the invariant measure associated to the focusing discrete Non Linear Schr\"odinger equation, defined on a discrete torus of dimension $d \geq 3$, and nonlinearity parameter $p>4$, in the infinite volume limit. The Gibbs measure has two parameters, the inverse temperature and the strength of the non linearity, and the scaling is such that the non linear and linear parts of the Hamiltonian contribute on the same scale. It was shown by Dey, Kirkpatrick and the first author that this measure undergoes a phase transition, concerning concentration of mass of a typical function {at the level of partition functions}. We prove that the three regions of the phase diagram yield three distinct local limits, a massive Gaussian free field, a massless Gaussian free field plus a random constant, and finally a (possibly trivial) mixture of massive Gaussian free fields, where some mass is ``lost'' to the region of concentration. Our proof relies on the analysis of the spherical model of a ferromagnet. The measure under consideration is an exponential tilt of the spherical model, and the additional tilt can be seen to induce an additional phase transition. Although our motivations come from the NLS equation and soliton resolution conjecture, our proofs are completely probabilistic, and can be read without any knowledge of PDE theory.
\end{abstract}
\section{Introduction and main results}\label{sec:intro}
The nonlinear Schr\"odinger equation (or \textbf{NLS} in short) is a prototypical example of a nonlinear dispersive PDE and has attracted significant attention over the years (see \cite{TAO} and references therein for a broad overview on the study of nonlinear dispersive equations).  Constructing invariant Gibbs measures for this PDE and studying the typical functions sampled from these measures is an important problem in this field with several influential works that address this \cites{LRS88,B94,B96,TZE,R02,TW,BS96}. 
Our main concern in this article is to study the behavior of the invariant Gibbs measure associated with this PDE in the discrete setting of the $d$-dimensional torus.

Specifically, we work on the torus $\bT_n^d:=(\bZ/n\bZ)^{d}$, with $N:= n^d.$
Fix $p>4$ and $\nu, \theta>0$. 
Let us introduce the notation
\begin{equation}
    \nu_N := \frac2p\left(\frac{\nu}{N}\right)^{^{\frac{p-2}{2}}}.\label{eq:nuN}
\end{equation}
Our main protagonist is the random variable $\Psi_N^{\nu, \theta}: \bT_n^d \to \bC$ distributed according to the following law. For any bounded continuous function $\ff$ define
\begin{align}
\bE(\ff(\Psi^{\nu, \theta}_N)):=\frac{1}{\mathcal Z_{N}(\theta, \nu)}\int_{\bC^N} \ff(\psi)\exp\left[\theta \left(\nu_N\norm{\psi}_{p}^{p}-\norm{\nabla \psi}_{2}^{2} \right)\right]\mv1_{\norm{\psi}_{2}^{2}\leq  N} \vd \psi.\label{eq:psi}
\end{align}
where $\|\psi\|_p^p = \sum_{\mvx \in \bT^{d}_{n}} |\psi(\mvx)|^p $ and $\|\nabla \psi \|_2^2 = \sum_{\mvx\sim\mvy} |\psi(\mvx) - \psi(\mvy)|^2$, where $\mvx\sim \mvy$ denotes the vertices being adjacent.
We call the random vector $\Psi_N^{\nu, \theta}$ as defined in \Cref{eq:psi} a \textbf{focusing NLS field} with parameters $\nu, \theta $ and the measure induced by it the \textbf{focusing NLS invariant measure}, or simply the \textbf{focusing NLS measure}, denoted by $\mu_{N, \nu, \theta}$. The reasoning behind this nomenclature is that this measure is an invariant measure for the discrete focusing nonlinear Schr\"odinger equation, we shall explain this connection later in \Cref{sec:soliton_NLS}.
For now, readers familiar with the Gaussian free field (GFF)  (see \Cref{sec:gff} for an overview) will recognize \eqref{eq:psi} as an exponential tilt the complex GFF at temperature $\theta$  by the term $\exp(\theta \nu_N \|\psi\|_p^p)$. Furthermore, since $\nu_N$ is positive, the restriction of the $\ell^2$-norm is introduced to ensure the integral is finite (so we have a probability measure).

The reasoning for the form of the tilt coefficient $\nu_N $ is to match the order of the Gaussian tail with the tilt, the details will be clearer later.
Observe for now that we could have introduced an extra parameter $\gamma$ in the indicator above and restricted to $\norm{\psi}_{2}^{2}\leq  \gamma N$ instead of $\norm{\psi}_{2}^{2}\leq  N$. However, calling this random field $\Psi^{\nu, \theta, \gamma}$, and using the invariance

$$\Psi^{k \nu, k\theta,  k^{-1}\gamma}_N = \Psi^{\nu,  \theta, \gamma}_N \text{ in law,}$$ it is sufficient to concentrate on $\Psi_N^{\nu, \theta}$ as defined in \Cref{eq:psi}.  Recall that the local weak limit of a random field  $X_N : \bT_n^d \to \bC$ is $X:\bZ^d \to \bC$ if for every finite set $A$, $(X_N)_{x \in A} $ converges to $(X)_{x \in A}$.

Fix $d \ge 3$ and define 
\begin{equation}
    C_d = {\sf G}^{\bZ^d}(0,0) < \infty,\label{eq:C_d}
\end{equation}
where ${\sf G}^{\bZ^{d}}$ denotes the Green's function in $\bZ^d$. 
Our main result, censoring some of the details, is summarized as follows. We postpone the formal definition of the massive Gaussian free field to \Cref{def:GFF}, for now it suffices to recall that it is a Gaussian field on $\bZ^d$ whose correlations are given by the massive Green's function, which decay exponentially in the distance if the mass is strictly positive. If the mass is zero, this decay is polynomial.
\begin{thm}\label{thm:main_summary}
Fix $p>4$ and $d \ge 3$.
For every $\theta>0$ there exists a $\nu_c=\nu_c(\theta) \in [0,\infty)$ such that the following holds. All the GFFs mentioned below are complex GFFs with variance $\theta^{-1}$.
    \begin{enumerate}[a.]
        \item If $\theta<C_d$, $\nu<\nu_c$, then the local weak limit of $\Psi_N^{\nu, \theta}$ is $\theta^{-1/2}$ times a massive GFF with mass $m(\theta)>0$ given by \eqref{eq:mass}.
        \item If $\theta \ge  C_d$ and $\nu<\nu_c$, then the local weak limit if $\Psi_N^{\nu, \theta}$ is  $\theta^{-1/2}\cdot $GFF(with no mass) $+ U\mv1$ where $U \in \bC$  is uniform over a disc of radius $\sqrt{1-\frac{C_d}{\theta}}$ centered around $0$ and  and $\mv1$ is the function which is identically 1 on $\bZ^d$. In particular, if $\theta = C_d$ then $U=0$ almost surely.
        \item If $\nu>\nu_c$, then subsequential weak limit exists and each such limit is $\theta^{-1/2}$ times a massive GFF with a potentially random mass $M$.  Furthermore, $M>0$ almost surely, and if $\theta<C_d$ then $M<m(\theta)$ almost surely, where $m(\theta)$ is as in item a.
    \end{enumerate}
\end{thm}

\begin{figure}
    \centering
    \includegraphics[width=0.5\linewidth]{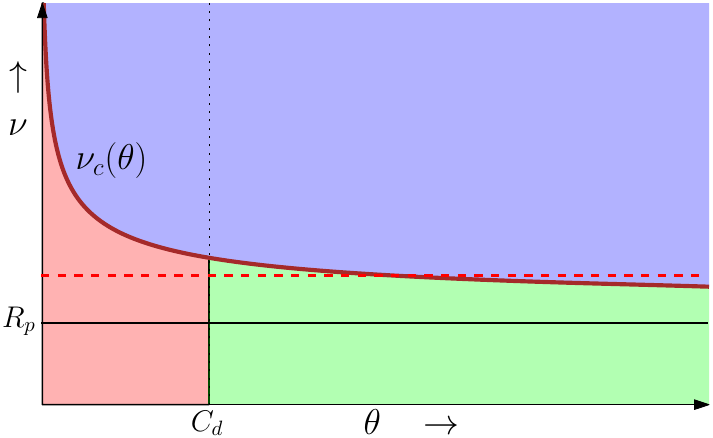}
    \caption{The phase diagram for NLS field. The purple region corresponds the massive GFF limit for supercritical $\nu$ (\Cref{thm:main_summary}, part c.), the red region corresponds to the massive GFF limit for subcritical $\nu$ (\Cref{thm:main_summary}, part a.), and the green region corresponds to the (non-massive) GFF with the random global shift limit (\Cref{thm:main_summary}, part b.) We do not know of the behaviour of the field at the critical curve (drawn in solid brown), see \Cref{sec:future} for further discussions.  Something curious happens if we fix $\nu$ to correspond to the horizontal dashed red line. As we vary $\theta$, the correlations are decaying exponentially at first (red part), then polynomially (green part), and then exponentially again (purple part).}
    \label{fig:phasediag}
\end{figure}
A more precise version of item c. above can be found in \Cref{thm:supercritical}.
We believe that the random variable $M$ above is actually deterministic and is a subject of future work; more details follow later.
To explain \Cref{thm:main_summary} in greater detail, we introduce an important supporting character in this article, the \textbf{spherical model}. We define the \textbf{spherical law} to be the law of the random field $\Psisph:\bT_n^d \to \mathbb C^N$ defined as follows. For every bounded continuous function $\ff:\bC^N \to \bR$ we have  
\begin{align}
\bE(\ff(\Psisph^{\theta})):=\frac{1}{\cZ_{\text{sph},N}(\theta)}\int_{\bC^N}\ff(\psi) \exp\left(-\theta \norm{\nabla \psi}_{2}^{2}\right)\mv1_{\norm{\psi}_{2}^{2}\leq N}\vd \psi.\label{eq:spherical}
\end{align}
Clearly, the NLS model is a tilt of the spherical model.
More precisely,
\begin{equation}
    \bE(\ff(\Psi^{\nu, \theta}_N))  = \frac{\cZ_{\text{sph},N}(\theta)}{\mathcal Z_{N}(\theta, \nu)}\bE(\exp(\theta \nu_N \|\Psisph^{\theta}\|_p^p )\ff(\Psisph^{\theta})),\label{eq:NLS_sph_relation}
\end{equation}
and the spherical model, in turn, is a GFF conditioned to have its total $\ell^2$-mass restricted. To compare the behavior of the spherical model with that of the NLS, we mention here that the spherical model also undergoes a phase transition.
\begin{thm}\label{thm:main_sph}
    The following holds.
    \begin{enumerate}[a.]
        \item If $\theta<C_d$, $\Psisph^{\theta}$ converges locally to $\theta^{-1/2}$ times a massive GFF with mass $m(\theta)$, which is the same as item a. in \Cref{thm:main_summary}.
        \item If $\theta \ge  C_d$, then the local weak limit if $\Psisph^{\theta}$ is  $\theta^{-1/2}$GFF(with no mass) $+ U\mv1$ where $U $  is uniform over a disc of radius $\sqrt{1-\frac{C_d}{\theta}}$ centered at $\mv0$ and  and $\mv1$ is the function which is identically 1 on $\bZ^d$. In particular, if $\theta = C_d$ then $U=0$ almost surely.
    \end{enumerate}
\end{thm}
The spherical model defined with the $\ell^2$-norm being equal to $N$, as opposed  $\ell^2$-norm less than or equal to $N$ considered in this article has a rich history. It was originally introduced by Berlin and Kac \cite{KAC} as a model of ferromagnetism more amenable to analysis by the method of steepest descent. In particular, the limiting free energy was shown to match with a that of a corresponding Gaussian Free Field.  A detailed picture of the low temperature regime for the case of complex valued spins was established by Lukkarinen in \cite{LUK}, with a specific emphasis on the nature of the ``condensate'' (The Fourier modes corresponding to the small eigenvalues) as the dispersion relation varies, and the convergence of the ``normal fluid'' (the Fourier modes corresponding to the orthogonal complement of the condensate) to a critical Gaussian field. During the preparation of our article, Aru and Korzhenkova put forth an article \cite{ARU} which characterized the local limit of the spherical model for real valued spins, as well as the spin $O(N)$ model with simultaneous limit of large spin dimensionality, the major advancement being the rigorous proof of the local limit in the critical case $\theta=C_{d}$. 

{Our main focus is the local limit of the NLS field, of which \Cref{thm:main_sph} is an important ingredient. There is a significant overlap between Theorem \ref{thm:main_sph} and the main theorem of \cite{ARU} as well as a subcase of the main result of \cite{LUK}. In particular, the $\theta>C_{d}$ case can be deduced from the methods of \cite{ARU}, and our technique here is an alternative. However, there are notable differences which arise from the differences in the constraint. The ``$\leq$" constraint allows for significantly simplified proofs in the $\theta= C_{d}$ case. Additionally, in the $\theta<C_{d}$ case, we also require the local limit when boundary conditions are imposed, which is not covered by \cite{ARU} or \cite{LUK}.  Our proof techniques follow somewhat diverging paths, proceeding via only analyzing the Fourier modes, more details on this in \Cref{sec:spherical_proofs}. We remark however that we do use the sharpened Green's function estimates verified in \cite{ARU}. However, we only need a weak version of this estimate for $d \ge 4$. Furthermore, we do not need the precise sign of the difference of the finite and infinite volume Green's functions in $d=3$.    

}
With the aid of \Cref{thm:main_sph}, the effect of the tilt is apparent. If $\nu>\nu_c$, then item c. of \Cref{thm:main_summary} tells us that some mass from the local limit of the spherical model is `lost'. In  contrast if $\nu<\nu_c$, then no mass is `lost' and the behavior is similar to that of the spherical model, and in fact identical in the thermodynamic limit. In fact, the tilt induces the following ``double'' phase transition on the spherical model. 
\begin{cor}\label{cor:doublephase}
There exists a $\nu>0$ such that $\Psi_{N}^{\nu, \theta}$ undergoes two phase transitions as $\theta$ is varied, the first at $\theta=C_{d}$ and the second at $\nu_{c}^{-1}(\theta)>C_{d}$, where the local limit is first massive, then non massive, then finally massive again (see \Cref{fig:phasediag}) 
\end{cor}
In particular, what this means is that the decay of correlations for $\Psi^{\theta,\nu}_{N}$ as $\theta$ is increased is first exponential, then polynomial, and finally exponential again. 
\subsection{Heuristics behind the phase transition}\label{sec:heuristics}
Let us now explain the heuristics behind \Cref{thm:main_summary}, particularly, the reason behind the appearance of mass in the GFF, and its subsequent `loss' for large $\nu$. This program was initiated by the first author along with Dey and Kirkpatrick in \cite{DKK} which forms the basis of this work. The key idea is that large values of $\nu$ forces the focusing NLS field to concentrate $O(N)$ amount of  $\ell^2$ mass in a small region (the exact connection with the NLS equation and soliton formation is explained later in \Cref{sec:soliton_NLS}). To that end, let us consider the (random) set $U$ where  absolute value of the NLS field is at least $ \sqrt{\eps N}$. If we assume the total $\ell^2$-mass concentrated on $U$ is $aN$, then the NLS field on $U$ should concentrate around the minimal energy field on $U$ constrained to have mass $aN$. This minimal energy is captured by the function $I$ which we define below in \eqref{def:Ifn}. 

Note that because of the $\ell^2$-constraint, $U$ is not too large, only of order roughly $\eps^{-1}$ in size. It is not hard to see  (\Cref{lem:lp_largest}) that the largest $\ell^p$-norm of a function with constraints in both $\ell^2$ and $\ell^\infty$ norm is obtained by assigning the largest possible absolute values to as many vertices as possible. Obviously, the largest number of vertices where we can assign absolute value $ \sqrt{\eps N}$, if the $\ell^2$ norm is bounded by $N$, is given by $\eps^{-1}$. Thus the maximum $\ell^p$-norm with the given constraints is $\eps^{-1}(\eps N)^{p/2}$. Overall the tilt term is
$$
\exp\left(\frac2p (\nu \eps)^{\frac{p-2}{2}} N\right)
$$
Thus if $\eps>0$ is small enough, the tilt term is insignificant in exponential scale, and the dominant term is the partition function of the spherical model with mass at most $(1-a)N$. This contribution (which we call the contribution of the entropy) is dictated by the log partition function of a GFF whose $\ell^2$-mass is constrained. This is captured by the function $W$ which we define below. The main result of \cite{DKK} is to show that the log partition function of the NLS field concentrates on a value of $a$ which minimizes the sum of the contribution  from $I$ and $W$.

Now let us define the terms $I$ and $W$ more concretely, the above heuristics should demystify the motivation behind these definitions.
Let us define energy of a function $\psi_N : \bT_n^d \to \bC^N$, denoted $\cH_N(\psi_N)$ as
\begin{equation}
  \cH_N(\psi_N):=\norm{\nabla \psi_N}_{2}^{2} - \nu_N\norm{\psi_N}_{p}^{p}. \label{eq:soliton_energy}
 \end{equation}
 The reasoning behind calling it `energy' should be clear by recalling the Hamiltonian in the NLS measure \eqref{eq:psi}.
The minimum energy at (rescaled) $\ell^{2}$ mass $a$ is denoted
\begin{align}
    I_N(aN):=\inf\left\{ \cH_N(\psi_N):\psi_N\in \ell^{2}(\bT_n^d), \norm{\psi_N}^{2}_{\ztwo}=aN\right\}.
\end{align}
It is known (see \cite{DKK} and \cite{C14}) that $$\lim_{N \to \infty} \frac1N I_N(aN) = I(a),$$ where $I(a)$ can be characterized as the minimum energy with  $\ell^{2}$ mass $a$ in infinite volume as follows:
	\begin{align}\label{def:Ifn}
		I(a):=\inf\left\{ \cH(\psi):\psi\in \ell^{2}(\dZ^{d}), \norm{\psi}^{2}_{\ztwo}=a\right\}.
	\end{align}
   where $$\cH(\psi) = \norm{\nabla \psi}_{2}^{2}-\frac{2}{p}\nu^{\frac{p-2}{2}} \norm{\psi}_{p}^{p}.$$
Note the (appriximate) scaling relation $$``\cH(\psi_N/\sqrt{N}) = \frac1N\cH_N(\psi_N)''$$
where the approximation comes from the fact that the periodic boundary conditions are somewhat incompatible near the boundary for the energy computation. We don't need the exact scaling relation in what follows.

As for the function $W$, let $\gD$ denote the graph Laplacian on the discrete torus, defined as 
\begin{equation}
    (\Delta \psi)(\mvx) = \sum_{\mvy \sim \mvx} \psi(\mvy) - \deg(\mvx)\psi(\mvx), \qquad \text{ for all }\mvx \in \bT_n^d.\label{eq:laplacian}
\end{equation}
where $\deg(\mvx)$ is the degree of $\mvx$, and $\mvx\sim \mvy$ denotes the $\mvx$ and $\mvy$ being adjacent.
We point out to the astute reader that sometimes the graph Laplacian is defined as the above quantity divided by $\deg(x)$, which might be a source of confusion with normalization in other references.
Now define 
\begin{equation}
    W_N(b) := \frac1N \log \int_{\bC^N} \exp(-\|\nabla \psi_N \|^2_2 )1_{\|\psi_N\|^2_2 \le bN} d \psi_N \label{eq:Wn}
\end{equation}
and
\begin{equation}
    W(b) := \lim_{N \to \infty} W_N \label{eq:W}
\end{equation}
The fact that the above limit exists is proved in \cite{DKK}. Let
\begin{equation}
    F_{\nu, \theta}(a) = F(a):= W(\theta(1-a)) + \frac\theta{\nu} I(\nu a); \qquad {a \in [0,1]}.\label{eq:F}
\end{equation}
and define $F_{\min} = \min\{F(a):a \in [0,1]\}$. It is proved in \cite{DKK}*{Theorem 2.3} that
\[
\lim_{N\to \infty}\frac{1}{N}\log \cZ_{N}(\theta,\nu)=\log(\frac{\pi}{\theta})-F_{\min}.
\]
We will later need a more quantitative version  of the above convergence in \Cref{sec:tempering}.
It was proved in \cite{DKK}*{Lemmas 4.7, 5.9} that $W$ and $I$ are continuous functions in $[0,\infty)$. Define
\begin{align}\label{def:minset}
	\mathscr{M}(\theta,\nu):= \argmin_{0\le a\le 1} \left(F(a)\right)
\end{align}
as the set of minimizers for the variational formula. This set is compact by the continuity of $W$ and $I$. 

Going back to the heuristics with the set $U$, the splitting of mass on $U$ and $U^c$ is dictated by \Cref{def:minset}. It turns out that $\mathscr{M}(\theta,\nu) \ni 0$ if $\nu<\nu_c$ and $\mathscr{M}(\theta,\nu)$ is bounded away from 0 if $\nu>\nu_c$. The former is called the \emph{dispersive phase} and the latter the \emph{solitonic phase} (more on this in \Cref{sec:soliton_NLS}). Furthermore, if $\nu>\nu_c$, $\theta(1-a)<C_d$ for every $a \in \mathscr{M}(\theta, \nu)$ which explains the positive mass of the local limit GFF in \Cref{thm:main_summary}, item c.

Note however that the analysis in \cite{DKK} is mostly at the level of partition functions, and not at the level of the distribution of the NLS field itself. Therefore, several technical hurdles need to be crossed to make the above heuristic precise, which is what we accomplish in this manuscript.

\subsection{Motivations: Solitions  and NLS equation}\label{sec:soliton_NLS}

Although we do not directly deal with the associated PDE in this article, it is worthwhile explaining the PDE perspective. To draw analogues with the continuum, we need to introduce a parameter $h$ which is equal to the \emph{lattice spacing}, or the distance between nearest neighbours of $\bT_n^d$ in $\bR^d$. The focusing discrete NLS now with nonlinearity parameter $p$ is given by
\begin{equation}\label{eq:DNLS}
    {\mathsf i}\frac{\partial}{\partial t}\psi(\mvx,t) = -\frac1{h^2}\Delta \psi(\mvx,t) -|\psi(\mvx,t)|^{p-2}\psi(\mvx,t), \qquad \mvx \in \bT_n^d, \quad t>0
\end{equation}
where $\mathsf i = \sqrt{-1}$.
Without the nonlinear term $|\psi(\mvx,t)|^{p-2}\psi(\mvx,t)$, \eqref{eq:DNLS} is known as the free Schr\"odinger equation which can be solved explicitly, by diagonalizing the Laplacian using the Fourier basis. The free Schr\"odinger equation is \emph{dispersive}, and one of the consequences is that the $\ell^q$ norm of the solution locally tends to 0 as $t \to \infty$, for any $q\in (2,\infty]$. 

The term \emph{focusing} refers to the negative sign in front of the nonlinear term, which in a very precise sense combats the dispersive effect of the linear evolution. A unique feature of the focusing DNLS is the fact that it admits \textbf{soliton} solutions, which arise explicitly as a consequence of the competition between the dispersion and non linearity. These are stable, time periodic, spatially preserved solutions to \eqref{eq:DNLS} Changing this sign leads to the \emph{defocusing} variant of this equation, where the non linearity in a sense aids the dispersion. 

The DNLS equation \eqref{eq:DNLS} admits a Hamiltonian of the form
\begin{equation}\label{eq:DNLS_hamiltonian}
\cH(\psi)  = \frac{1}{Nh^{2}} \|\nabla \psi\|_2^2 - \frac{2}{p\cdot N}  \|\psi\|_p^p,
\end{equation}
where $\psi$ and $\bar{\psi}$ are the canonical conjugate variables. By N\"oether's theorem (see \cite{TAO}*{Section 1.5}), both the Hamiltonian $\cH$ and the  mass $\|\psi\|_2^2$ are conserved in time under the DNLS flow. By Liouville's theorem, \cite{TAO}*{Exercise 1.33}, the Lebesgue measure on $\ell^{2}(\bT^{d}_{n})$ is also invariant with respect to the flow. Consequently, the following Gibbs measure
\begin{equation}
   \mu_{\beta, B}(\psi)  \propto \exp(-\beta \cH(\psi))1_{\|\psi\|^2_2 \le B N}\label{gibbs_DNLS}
\end{equation}
is invariant with respect to the DNLS flow (more precisely, if we consider $(\Psi(\mvx,0))_{\mvx \in \bT_n^d} \sim \mu_{\beta, B}$ and let it evolve using \eqref{gibbs_DNLS}, then $(\Psi(\mvx,t))_{\mvx \in \bT_n^d} \sim \mu_{\beta,B}$ as well for any $t>0$.)

The Gibbs measure in \eqref{gibbs_DNLS} was first analyzed by Chatterjee and Kirpatrick in \cite{CK12} for $p=4$ and $d\geq 3$. Chatterjee extended this result to a parameter regime where the gradient term is non-negligible but only for $p<2+4/d$ which yields a non-trivial scaling limit \cite{C14}. It is also worthwhile to note that the result in \cite{C14} is not for a Gibbs measure of the form \cref{eq:DNLS}, but from a `uniformly chosen' random field with given mass and energy properly chosen. The scaling regime considered in this article (and first introduced in \cite{DKK}) is distinct from either of these two cases.  As observed in \cite{CK12}, in the limit $N\to \infty$ and $h\to 0$ slowly such that the ``side length'' $N^{\frac{1}{d}}h$ remains bounded away from $0$ or diverges, the non linear term dominates the dispersion. This can be seen by evaluating the Hamiltonian at a function $\phi$ which takes value $\sqrt{BN}$ on a single lattice site. Observe that 
\[
\frac{2}{pN}\norm{\phi}_{p}^{p}=\frac{2N^{\frac{p-2}{2}}B^{\frac{p}{2}}}{p}\text{ and }\frac{1}{Nh^{2}}\norm{\nabla \phi}_{2}^{2}=\frac{2dB}{h^{2}}.
\]
When $p>4$, the non linearity not only beats the dispersion, but in fact dominates the entropy associated to the uniform measure on the $\ell^{2}$ ball, and there is no phase transition. The scaling of $N^{-\frac{p-2}{2}}$ introduced in front of the non linearity in this article and \cite{DKK} is meant to combat this, such that the linear and non linear parts of the Hamiltonian always contribute on the same scale. The scaling in \cite{C14} is also such that the two parts contribute on the same scale, however the condition that $p<2+\frac{4}{d}$ is crucial. When $d\geq 3$ and $p>4$, we are in a regime of scaling where the continuum limits of discrete solitons diverge in the appropriate norms, which removes the possibility of considering the analogous continuum scaling in \cite{C14}. 

We can relate the parameters $\nu, \theta,N$ of the NLS field \eqref{eq:psi} to the parameters $\beta, B,h,N$ that appear in \eqref{gibbs_DNLS}. This is simply given by 
\begin{equation}
    \theta  = \frac{\gb B}{Nh^{2}}\text{ and }\nu_N = h^{2}B^{\frac{p-2}{2}}.
\end{equation}


Our result is thus a natural addition to the results in \cites{DKK,C14,CK12}, we work with a parameter setup which yields a non-trivial discrete limit.

The study of the discrete NLS stems in part to understand qualitative aspects of the continuum equation. The continuum focusing NLS displays subtle behavior in terms of well-posedness and depends crucially on $p$ and the regularity of the initial data, see Chapter 14 of \cite{TES} and \cite{TAO} for a broader overview, and there remain many regimes where the long term behavior of solutions is poorly understood. In particular, there is a drastic difference in the requirements for well posedness when $p<2+\frac{4}{d}$ and $p\geq 2+\frac{4}{d}$, and the finite time blow up phenomenon for large $p$ has an analogue in discrete regime, in the form of a threshold for soliton formation \cite{WEI99}. 

There is a rich collection of literature on the construction of continuum invariant measures for the NLS, as well as their uses in addressing questions such as well-posedness. The typical approach here is to construct exponential tilts of an appropriate Gaussian Free field, and showing that the tilt is integrable. For the defocusing NLS, invariant measures were used by Bourgain to prove global well-posedness for the cubic NLS in $d=1$ \cites{B94} and $d=2$ \cite{B96} on tori, and later by Tzvetkov for the NLS defined on the disk in $\bR^{2}$ \cite{TZE}. Recently, Bourgain's work in $d=2$ was significantly extended beyond cubic nonlinearity by Deng, Nahmod and Yue \cite{DNY}. We also point out a remarkable phenomenon, the use of randomly sampled initial data to construct solutions in regularity regimes that are deterministically ill-posed, identified by Tzvetkov and Burq \cites{TZEB1,TZEB2}. It is worth noting that the non linearity terms are Wick ordered in $d>1$, since the underlying Gaussian free field is supported on distributions which cannot be evaluated at points. 

The focusing case has also been examined, particularly in the regime $p<2+\frac{4}{d}$, where $H^{1}$ initial data is adequate to ensure a global solution. This was initiated by Lebowitz, Rose and Speer on the 1-dimensional torus \cite{LRS88}. This measure was proved to be invariant for the focusing NLS by McKean and Vaninsky \cite{MV94}, as well as Bourgain \cite{B94}. The cases of $p=2+\frac{4}{d}$ are more subtle, and normalizability (finiteness of the partition function) holds in $d=1$ upto and including a critical mass threshold as proved by Oh, Sosoe and Tolomeo \cite{OST}, but not in $d=2$ \cite{BS96} for any mass cut off as proved by Brydges and Slade. The approach of tilting the Gaussian Free Field fails even with Wick ordering in $d\geq 3$, as the field is too rough, which is why the approach of discretization has been adopted in this article, as well as in \cite{CK12} and \cite{DKK}. 

Understanding the behavior of a typical function from the continuum focusing NLS measure as the volume tends to infinity is a problem that is quite involved, and has attracted a great deal of interest. Early results on the triviality of the limit when $p=4$ and $d=1$ are due to Rider \cite{R02}, particularly that there is a strong concentration of the measure about a continuum soliton of appropriate mass, which can be freely translated. This was significantly expanded upon by Tolomeo and Weber in \cite{TW}, again for $p=4$ and $d=1$ by tuning the strength of the non linearity and examining the corresponding regimes. In particular, there is a critical regime where a phase transition similar to the one considered here and in \cite{DKK} emerges, and the subcritical limiting measure is Gaussian. Corresponding to the strongly non linear case for the same parameters, Seong and Sosoe recently verified a central limit theorem for fluctuations of a the typical function around the soliton, namely that they converge to white noise \cite{SS24}. Our result here in a precise sense is the higher dimensional analogue of \cite{TW} and \cite{SS24} for the regime of critical strength of non linearity and $p>4$. It would be interesting to see whether our results can be extended to study the fluctuations in the vicinity of the discrete soliton like \cite{SS24}, but difficulties such as verifying local convexity of the Hamiltonian remain. From a rather different perspective, understanding the infinite volume limit of an invariant measure for the focusing NLS led Chatterjee to a proof of a version of the \textbf{soliton resolution conjecture} \cite{C14} for general $p<2+4/d$, which roughly states that for generic initial data, after sufficient time has elapsed a solution should be composed of a soliton and a portion that disperses away, and invariant measures are a natural means of characterizing the term ``generic''. Although still in finite volume, we also direct attention towards the work of Greco, Oh, Tolomeo and L. Tao \cite{GOTT} for the case of ``weak'' non-linearity in $d=2$, where the strength of the non linearity term is sent to zero, which also results in a phase transition.

\subsection{Outline of the proof}\label{sec:outline}
The proof of Theorem \ref{thm:main_summary} is split into two parts, corresponding to subcritical and supercritical $\nu$ for a given $\theta$. Common to both is the characterization of the phases of the measure, as put forth in \cite{DKK}. Namely, we deterministically define a (possibly empty) set $$U(\psi,\eps):= \{\mvx: |\psi(\mvx)| > \sqrt{\eps N}\}$$ for some threshold $\eps$ (to be carefully chosen later). Because of the $\ell^2$-bound on $\psi$, this set is of size $O(\eps^{-1})$, again deterministically. As mentioned before, the mass of the NLS field on $U$ concentrates on a minimizer of the energy functional $F$ \eqref{eq:F}, and on the rest the $\ell^{\infty}$ norm is $o(\sqrt{N})$.
We need $U$ to satisfy some more technical properties, which is recapped again in Section \ref{sec:tempering}, but let us ignore those for now. This statement is fleshed out in \Cref{prop:U}, which we call \emph{tempering the NLS field.} 
One of the main results of \cite{DKK} is that $\frac{1}{N}\norm{\psi|_{U}}_{2}^{2}$ jumps discontinuously from $0$ to $>0$ as $\nu$ is changed from subcritical to supercritical, {however this was only done through the lens of partition functions, which is not enough for the results in this article.} The key to the actual behaviour of $U$ is an analysis of the energy functional $F$, which is done in \Cref{sec:landscape}. {There are certain gaps in our knowledge of the behaviour of $F$, and indeed once those are filled, the techniques of our paper can be pushed through to give finer versions of \Cref{thm:main_summary}, particularly in the case $\nu \ge \nu_c$.}

In the subcritical case ($\nu<\nu_c$), the goal is to show that the local limit of $\Psi^{\theta,\nu}_{N}$ is the same as $\Psisph$. Broadly, this is accomplished in two steps:
\begin{enumerate}
\item We modify the proof of convergence of free energy in \cite{DKK} to obtain large deviation estimates for  $\norm{\Psi|_{U}}_{2}^{2}$, namely Proposition \ref{prop:U}. These estimates enable us to deduce an effective $\ell^{\infty}$ bound for $\Psi^{\nu, \theta}_{N}$, namely that with high probability $\norm{\Psi}_{\infty}$ is $o(\sqrt{N})$.  The verification for the $\ell^{\infty}$ bounds are distinct for $\theta\leq C_{d}$ and $\theta>C_{d}$, the latter requires an additional application of a discrete version of the GNS inequality. 
\item Verify a gaussian tail estimate for $\Psisph$, and use it to upgrade the previous $\ell^{\infty}$ bound to $O(N^{\gb})$, where $0<\gb<1/2$, and is small enough such that any function with this $\ell^{\infty}$ norm has $\ell^{p}$ which is $o(N^{\frac{p-2}{2}})$. This allows us to disregard the exponential tilt. 
\end{enumerate}
All of this is accomplished in \Cref{sec:subcritical}.
\bigskip

The supercritical case $(\nu>\nu_c)$ is more involved, and essentially involves the removal of the region of concentration, followed by an application of a version of the subcritical argument. The order of events are as follows.  
\begin{enumerate}
\item Separate out the set of concentration, yielding an $\ell^{\infty}$ bound outside $U$. This also removes a (possibly random) but strictly non zero fraction of the mass. This involves the large deviation estimates for $\norm{\Psi|_{U}}_{2}^{2}$ mentioned earlier, namely Proposition \ref{prop:U}. 
\item Verify a tail bound for the spherical model with reduced mass, and boundary condition from the gradient estimates of the separation. The tail bound is for the spherical law with a natural harmonic extension (with correctly chosen mass) removed. The subtlety here is that the boundary condition can be large, but not $O(N)$ and we want the tail bound even near the boundary. This is done in \Cref{lem:spherical_tail} and is of independent interest.
\item Use the tail bound to upgrade the $\ell^{\infty}$ bound to be $o(N^{\gb})$ on the outside of an expanded set $\tilde{U}$, which none the less has poly-logarithmic size, and moreover, the mass fraction on $\tilde{U}$ is still close to the mass fraction on $U$. Note here that even a polylog size set outside $U$ can have potentially a lot of mass as the $\ell^{\infty}$-bound is only $\sqrt{\eps N}$.  We bypass this by finding set containing $U$ using a delicate choice of parameters from the tempering \Cref{prop:U}. This is done in \Cref{prop:tilde_U}.
\item Use the subcritical argument, namely that the non linear exponential tilt is now $o(1)$ on $\tilde{U}^{c}$, to show that the local limit on $\tilde{U}^{c}$ is the same as the spherical model with small boundary condition, which is separately proven to locally converge to a massive GFF.  Technicalities one needs to overcome involve preserving some aspects of domain Markov property, despite the global $\ell^{\infty}$-bound imposed by conditioning on $U$.
\end{enumerate}
All of the above ideas  are executed in \Cref{sec:supercritical}.

\medskip

We now turn to Theorem \ref{thm:main_sph}. 
\begin{enumerate}
\item Re-express the spherical law in the Fourier Basis. This makes it clear that we can regard the spherical law as the sum of a zero average GFF and a random constant that is uniform on the disk of radius $\sqrt N$ in $\bC$, conditioned to have $\ell^{2}$ norm appropriately bounded above.
\item The Fourier basis is also useful, as we may represent  $\norm{\Phi_{m}}_{2}^{2}$ and $\norm{\Phizero}_{2}^{2}$ as the sum of independent exponential random variables. 
\item For $\fg:\bT^{d}_{n}\to \bC$, an arbitrary function supported on a fixed and finite set $A\subseteq \bT^{d}_{n}$, we express $\E|\la\Psisph,\fg\ra|^{2r}$ for some $r>0$ in terms of expectations of the random variables in the previous point, obtaining a sum of several ratios of expectations, involving monomials of fixed finite degrees. 
\item In the case $\theta\geq C_{d}$, the conditioning is not too atypical, and is actually typical when $\theta>C_{d}$. We use the convergence in probability of $N^{-1}\cdot \norm{\Phizero}_{2}^{2}$, and the fact that the limit is the same when a fixed finite number of exponential random variables constituting the sum $\norm{\Phizero}_{2}^{2}$ are removed. For the case $\theta=C_{d}$, an anti-concentration estimate is additionally used to bound the denominator term from below. The subtlety here is that some of the exponential random variables constituting $\norm{\Phizero}_{2}^{2}$ have diverging mean as $N\to \infty$, this requires some care, especially when $d=3$. 
\item For the case $\theta<C_{d}$ introduce an exponential tilt by $\norm{\Psisph}_{2}^{2}$ to make the conditioning typical. Controlling the ratio of expectations requires a refined version of the Lindeberg-Feller central limit theorem. The subtlety here arises in considering the case with boundary conditions, and the contribution of a massive harmonic function needs to be carefully canceled out. 

\end{enumerate}
\subsection{Summary of notations}

The following notation is standardized and appears throughout the article.
\begin{enumerate}
\item We use standard big $O$ and little $o$ notation. To recall, let $a_{N}$ and $b_{N}$ be sequences, such that $a_{N}\to \infty$ as $N\to \infty$. Then
\begin{enumerate}
    \item $b_{N}$ is $O(a_{N})$ if there exists a constant $C>0$ such that $b_{N}\leq Ca_{N}$. 
    \item $b_{N}$ is $o(a_{N})$ if $b_{N}/a_{N}\to 0$ as $N\to \infty$.
    \item $b_{N}\asymp a_{N}$ if there exist $0<c<C$ such that $cb_{N}\leq a_{N}\leq Cb_{N}$. We also denote this as $b_{N}=\gO(a_{N})$. 
\end{enumerate}
\item We will denote by $d(\mvx, \mvy)$ the graph distance between $\mvx,\mvy \in \bT^{d}_{n}$. This makes its first appearance in Section \ref{sec:mgff}. Also, for $B\subset \bT^{d}_{n}$, 
\[
d(\mvx,B):=\min \{d(\mvx, \mvy):\mvy\in B\}.
\]
\item Let $B\subset \bT^{d}_{n}$. Then $\partial B$ denotes the inner boundary of $B$, that is 
\[
\partial B:=\{\mvx:\mvx\in B\text{ and } d(\mvx,B^{c})=1\}.
\] 
and $\bar B = B \cup \partial B^c$. (defined in \Cref{sec:mgff})
\item Let $f:\bT^{d}_{n}\to \bC$ be a function, and $B\subset\bT^{d}_{n}$. Then we denote the restriction of $f$ to $B$ as $f|_{B}$, that is 
\[
f|_{B}(\mvx):=\begin{cases}
f(\mvx) \text{ when }x\in B, \\ 0 \text{ when } \mvx \in B^{c}. 
\end{cases}
\]
This makes its first appearance in Section \ref{sec:mgff}. 
\item All norms, unless explicitly mentioned are for functions $f:\bT^{d}_{n}\to \bC$, and the relevant ones are 
\[
\norm{f}_{p}^{p}:=\sum_{\mvx\in \bT^{d}_{n}}|f(\mvx)|^{p},\text{ }\norm{\nabla f}_{2}^{2}:=\sum_{\mvx\sim \mvy}|f(\mvx)-f(\mvy)|^{2}. 
\]
\item The following random fields are always denoted as indicated:
\begin{enumerate}
\item $\Psi^{\theta, \nu}_{N}$ denotes the focusing NLS field, at temperature $\theta^{-1}$ and strength of non linearity $\nu$. (defined in \Cref{sec:intro})
\item $\Phi_{m}$, $\Phi^{U}_{m}$ and $\Phizero$ denote the various avatars of the Gaussian Free Field; the massive, massive with Dirichlet boundary conditions and zero average respectively.   These are defined in \Cref{sec:gff}
\item $\Psisph^{\theta}$  denotes the usual spherical law at inverse temperature $\theta$ defined in \Cref{sec:intro}, and $\Psisph^{f, \theta, U^{c}}$ denotes the spherical law on $U^c$ with boundary condition $f$ on $U$, defined in \cref{sec:spherical_details}.

\end{enumerate}
\item The following constants are important and are always denoted as such:
\begin{enumerate}
\item $C_{d}={\sf G}^{\bZ^{d}}(\mv0,\mv0)$, the diagonal of the Green's function on $\bZ^{d}$, closely related to the expected number of returns of the simple symmetric random walk. 
 \end{enumerate}
\end{enumerate}
\section{GFF preliminaries}\label{sec:gff}

The overarching reference for this section is Chapter 8 of \cite{FL}, and specific proofs are provided for results that are not entirely standard but are required for this article. Let $G = (V,E)$ be a finite, connected graph and recall the Laplacian operator as defined in \eqref{eq:laplacian} acting on $\bC^V$. It is a negative semidefinite matrix acting on $\dC^{V}$, and has kernel of dimension 1 which is spanned by the constant functions on $G$. Let $B\subset V$. We also need to work with the Dirichlet Laplacian \emph{restricted on $B^{c}$} as follows. For any function $f:V \to \bC$, recall that $f|_{B^{c}} \equiv f$ on $B^{c}$ and 0 on $B$. Then $\Delta_{B^{c}} f=g$ where $g(\mvx) = \Delta_{B^{c}}f|_{B^{c}}(\mvx)$  for all $\mvx \in B^{c}$ and $g(\mvx)=0$ for all $\mvx \in B$. This definition may seem backward, but it is important from our perspective, as we want to think of the set $B$ as relatively small, and where we may impose boundary conditions.  
The inverse and resolvents of $\gD_{B^{c}}$ play a prominent role in this work, and we provide a probabilistic description of them now. Let $\sf D_{B^{c}}$ denote the diagonal matrix with entries given by the degrees of the vertices of $B^{c}$. Note that for a given $B\subset V$ and $m\geq 0$, $(-\gD_{B^{c}}+m)$ is positive definite when either $B^c\neq \emptyset$ or $m>0$. It is also useful to recall that
\begin{equation}
    ({\sf D}_{B^{c}})^{-1}(-\Delta_{B^{c}}+m) = -P_{B^{c}} + I+m{\sf D}_{B^{c}}^{-1},
\end{equation}
where $P_{B^{c}}$ is the sub-probability kernel for the simple random walk on $B^{c}$, killed when it enters $B$ (sub-probability as the walk has positive probability of entering $B$). This gives us a relation for the inverse of $(-\gD _{B^{c}}+m)$ in terms of the kernel of a killed random walk. Specifically, when we rearrange, 
\[
(-\gD_{B^{c}}+m)={\sf D}_{B^{c}}(I+m{\sf D}_{B^{c}}^{-1})(I-(1+m{\sf D}_{B^{c}}^{-1})^{-1}P_{B^{c}}). 
\]
Consider a random walk on $G$, denoted $X_{t}^{m,\mvx}$, started at $\mvx\in B^{c}$ which is killed with probability $\frac{m}{m+\deg(\mvy)}$ when it is at $\mvy\in B^{c}$, and is also killed when it enters $B$. It is not hard to verify that 
\[
(I-(1+m{\sf D}_{B^{c}}^{-1})^{-1}P_{B^{c}})^{-1}_{\mvx,\mvy}=\sum_{t=0}^{\infty}\bP(X_{t}^{m,\mvx}=\mvy),
\]
which is the expected number of times that $X_{t}$ visits $y$ before being killed. All assembled, we obtain
\begin{align}\label{eq:gfrw}
{\sf G}^{m,B^{c}}(\mvx,\mvy):=(-\gD_{B^{c}}+m)^{-1}_{\mvx,\mvy} =   \sum_{t=0}^{\infty}\bP(X_{t}^{m,\mvx}=\mvy)\cdot ({\sf D}_{B^{c}}+m)^{-1},\end{align}
where ${\sf G}^{m,B^{c}}(\mvx,\mvy) $ is called the \textbf{massive Green's function} with mass $m$ and Dirichlet boundary condition at $B$. In our case, where $B,B^{c}\subset \bT^{d}_{n}$, ${\sf D}_{B^{c}}$ is replaced by $2d I$.

\subsection{Massive GFF with Dirichlet boundary condition}\label{sec:mgff}

A \textbf{complex Gaussian} with variance $\sigma^2$ is a random variable $X+iY$ where $X, Y \sim $i.i.d.\ N$(0,\sigma^2/2)$. A standard complex Gaussian is a complex Gaussian with variance 1. Complex Gaussians have a useful property, if $X$ is complex gaussian with variance $\gs^{2}$, then $|X|^{2}$ is exponential with mean $\gs^{2}$.  We assume $G = (V,E)$ is a finite, connected, simple graph in this subsection.
\begin{defn}\label{def:GFF}
Let $B\subseteq V$ and $m \ge 0$. Assume $m=0$ only if $B \neq \emptyset$.
The complex valued \textbf{massive Gaussian Free Field on $G$} = $(V,E)$ with \textbf{Dirichlet boundary condition on $B$ and mass $m$} is a centered Gaussian process $\{\Phi^{B^{c}}_{m}(\mv{x})\}_{\mvx\in B^{c}}$ on the vertices of $V$ which is identically 0 on $B$ and 
\[
\bE(\overline{\Phi^{B^{c}}_{m}(\mvx)}\Phi^{B^{c}}_{m}(\mvy))={\sf G}^{m,B^{c}}(\mvx,\mvy), \qquad \mvx,\mvy \in B. 
\]
where ${\sf G}^{m,B^{c}}(\mvx,\mvy)$ is as defined in \eqref{eq:gfrw}.

Equivalently, if $\{\lambda_{B^{c},k}\}_{k=1}^{|B^{c}|}$ and $\{\phi_{B^{c},k}\}_{k=1}^{|B^{c}|}$ respectively denote the eigenvalues and orthonormal basis of  eigenvectors of $\gD_{B^{c}}$, then we may define the massive Gaussian free field as 
\[
\Phi^{B^{c}}_{m}(\mv{x})=\sum_{k=1}^{|B^{c}|} \frac{X_{k}}{\sqrt{\lambda_{B,k}+m}} \phi_{B,k}(\mv{x})
\]
where $\{X_{k}\}_{k=1}^{|V|}$ is a family of i.i.d. standard complex Gaussian random variables. 
\end{defn}
The relation to this work stems from the form of the density of $\Phi^{B^{c}}_{m}$ as a random vector taking values in $\bC^{B^{c}}$, it can easily be verified that for any $\ff:\bC^{B^{c}}\to \bC$,
\[
\bE(\ff(\Phi^{B^{c}}_{m}))=\frac{\det(-\gD_{B^{c}}+m)}{\pi^{|B^{c}|}}\int\ff(\psi)\cdot\exp(-\norm{\nabla\psi}_{2}^{2}-m\norm{\psi}_{2}^{2} )\vd \psi_{B^{c}}\mv1_{\psi|_{\partial B}=0}. 
\]
where $\partial B$ denotes the inner boundary, that is vertices in $B$ that neighbour vertices outside $B$.
In the specific context when $B=\emptyset$ and $m>0$, we replace $\Phi^{\bT_n^d}_m$  by $\Phi^{N}_m$ for brevity. Moreover, we have explicit expressions for the eigenvalues and eigenvectors. We  label vertices in $\dT^{d}_{n}$ as $\mvx=(x_{1},x_{2},\ldots ,x_{d})$ , where $x_{i}\in\{0,1,2,\ldots,n-1\}$. The eigenvalues and eigenvectors are indexed by a copy of the same torus (corresponding to the Pontryagin dual). We will index them by $\mvk=(k_{1},k_{2},\ldots k_{d})$. We have that 
\begin{align}\label{eq:evectors}
\phi_{\mvk}(\mvx)=\frac{1}{\sqrt{n^{d}}} \exp\left(\i \frac{2\pi}{n}\mvk\cdot \mvx\right),
\end{align}
and 
\begin{align}\label{eq:evalues}
\lambda_{\mvk}=4\sum_{i=1}^{d} \sin^{2}\left(\frac{\pi k_{i}}{n}\right). 
\end{align}

We now record a formula for $\bE(\|\Phi^N_m\|_2^2)$ which is a simple consequence of \Cref{def:GFF}, Parseval's identity, the explicit formulas \eqref{eq:evectors}, \eqref{eq:evalues} and the Riemann sum formula:
\begin{multline}\label{lem:mass_torus}
    \frac1N \bE(\|\Phi^N_m\|^2_2) = \frac1N\sum_{\mvx \in \bT_n^d} {\sf G}^m(\mvx, \mvx) = {\sf G}^m(\mv0,\mv0)=\frac1N \sum_{\mvk \in \bT_n^d} \frac1{\lambda_{\mvk}+m}\\=\int_{[0,1]^{d}}\frac{d\mv\gk}{4\sum_{i=1}^{d}\sin^{2} (\pi \gk_{i})+m} +o_N(1).
\end{multline}
where $o_N(1)$ is a function which tends to 0 as $N \to \infty$.
We now move towards describing the GFF with prescribed boundary conditions, as well as the extremely useful domain Markov property. An important ingredient here is the notion of harmonic functions. Given $B \subset V$, let recall that $\partial B$ denotes the inner boundary, that is vertices in $B$ that neighbour vertices outside $B$.
\begin{defn}
In particular, given $B \subset V$ and $f:\partial B\to \bC$, we say a complex valued function $h^{B^{c},f}$ is massive harmonic on $B^{c}$ with boundary condition $f$ if 
$$
((-\gD_{B^{c}}+m)h^{B^{c},f})(\mvx)=0 \text{ for all }\mvx \in B^{c}\text{ and }h^{B^{c},f}(\mvx)=f(\mvx)\text{ for all }\mvx\in \partial B. 
$$
\end{defn}
Observe that if $h^{B^{c},f}$ is complex valued and massive harmonic then both its real and imaginary parts are separately massive harmonic with the same mass. It is a standard fact that the massive harmonic function exists and is unique. We now provide the extremely convenient and standard description in terms of the killed random walk, which also makes an appearance in the expression for the Green's function. Let $\{X_{t}^{m,\mvx}\}_{t\in \bN\cup \{0\}}$ denote a simple symmetric random walk started at a point $\mvx\in B^{c}$ that is killed with probability $\frac{m}{m+2d}$ at each step, and is also killed when it enters $B$, which will clearly be at a boundary vertex in $\partial B$. Let $\skull$ denote the cemetery state and we imagine that when it is killed without hitting the boundary, it's state becomes $\skull$. With this convention,
\[
T_{m,B}:=\inf\{t\geq 0, X_{t}^{m,\mvx}\in B\cup \{\skull\}\}. 
\]
The massive harmonic function with boundary values $f$ may be represented as
\begin{align}\label{eq:massharmdef}
h^{B,f}_{m}(\mvx)=\bE(f({X^{m,\mvx}_{T_{m,B}}})). 
\end{align}
The convention here is that the boundary value when killed is taken to be $0$, that is $f({\skull})=0$. Expanded, 
\begin{align}\label{eq:massharmdef2}
h^{B^{c},f}_{m}({\mvx})=\sum_{\mvy \in \partial B}f(\mvy)\cdot \dP (X^{m,\mvx}_{T_{m,B}}=\mvy ).
\end{align}

\begin{defn}
We define the massive GFF on $B^{c}$ with prescribed boundary values given by $f:\partial B\to \dC$, denoted $\Phi_{m}^{B^{c},f}$ as
\begin{align}
\Phi^{B^{c},f}_{m}=\Phi^{B^{c}}_{m}+h^{B^{c},f}_{m}
\end{align}
\end{defn}
In terms of density, it can easily be verified using summation by parts that 
\[
\bE\ff(\Phi_{m}^{B^{c},f})=\frac{\det(-\gD_{B^{c}}+m)}{\pi^{|B^{c}|}}\int \ff(\psi)\cdot \exp(-\norm{\nabla\psi}_{2}^{2}-m\norm{\psi}_{2}^{2} )\vd \psi_{B^{c}}\mv1_{\psi|_{\partial B}=f}. 
\]
We now state the well known \textbf{Domain Markov Property} of the massive Gaussian free field. For any graph $G = (V,E)$ and $B \subset V$, we can decompose
\begin{align}\label{eq:DMP}
\Phi_m = \Phi^{B^{c}}_m + h^{B^{c}}
\end{align}
where  $\Phi^{B}_m$ is a massive GFF on with Dirichlet boundary condition on $B$ and $h^{B^{c}}$ is massive harmonic on $B^{c}$, and equal to $\Phi_{m}$ on $B$, in particular with boundary values $\Phi_{m}|_{\partial B}$.  Furthermore, $\Phi^{B^{c}}_m$ and $h^{B^{c}}$ are independent random fields (where we think of $\Phi^{B^{c}}_m$ as a field on $G$ with values 0 on $B^c$).

The killing rate  of the random walk in terms of which the massive harmonic function is expressed leads to some important consequences. 
\begin{lemma}\label{lem:l_p_h}

Let $B \subset \bT_n^d$ be fixed, and $f:\partial B\to \dC$. Let $h$ be the massive harmonic extension of $f$ onto $B^{c}$ with mass $m>0$. Then there exists a constant $c$ depending only on $m$ such that for every $\mvx$,
$$
|h(\mvx)| \le \|f|_{\partial B}\|_1 e^{-cd(\mvx, \partial B)}.
$$
There also exists a constant $C>0$ depending only on $m$ such that 
\[
\norm{h|_{B}}^p_{p}\leq C \norm{f|_{\partial B}}^p_{p}
\]
for all $p>1$
\end{lemma}
\begin{proof}
We will use the killed random walk description of massive harmonic function \eqref{eq:massharmdef}.
Since the killing probability in every step is $\frac{m}{m+2d}$, we must have 
$$
\dP(X_{T_{m,B}}^{m,\mvx}=\mvy) \le \left(\frac{m}{m+2d}\right)^{d(\mvx,\mvy)}.
$$
The first item is now immediate.
For the second item, note that
\begin{align*}
\norm{h}_{p}^{p} &\leq \sum_{\mvx}\left(\sum_{\mvy \in \partial B}\bP(X^{m,\mvx}_{T_{m,B}}=\mvy)f(\mvy)\right)^{p}
\\
&\leq \sum_{\mvx}\sum_{\mvy \in \partial B}\bP(X^{m,\mvx}_{T_{m.B}}=\mvy)|f(\mvy)|^{p}. 
\end{align*}
where the last inequality follows from an application of Jensen's inequality. Next, we may swap the order of summation to obtain 
\begin{align*}
\norm{h}_{p}^{p}\leq \sum_{\mvy \in \partial B} \sum_{\mvx \in B}\bP(X^{m,\mvx}_{T_{m.B}}=\mvy) |f(\mvy)|^p\leq \sum_{\mvy \in \partial B}\sum_{\mvx \in B^c}  e^{-cd(\mvx,\mvy)} |f(\mvy)|^{p} \le C\sum_{\mvy \in \partial B} |f(\mvy)|^p.
\end{align*}
where $c = c(m)$ and $C =\sum_{\mvx \in \bZ^d  }  e^{-cd(\mvx,\mvy)} $. 
\end{proof}    
Now we prove a lemma stating that slightly perturbing the torus does not change the formula \eqref{lem:mass_torus}. Essentially this is a consequence of the fact that far enough away from the perturbation, the Dirichlet and full torus massive GFFs are close in distribution, thanks to the domain Markov property and decay of harmonic functions.

    \begin{lemma}\label{mass_torus_holed}
        Let $m>0$ and $U\subset \dT^{d}_{n}$ be such that $|U|= O((\log N)^{10d})$. Let $\{\lambda_{U^{c},k}\}_{k=1}^{|U^{c}|}$ be as in \Cref{def:GFF}. Then
        \begin{equation*}
            \frac1N \bE(\|\Phi^{U^{c}}_m\|^2_2) = \frac1N\sum_{\mvx \in \bT_n^d} {\sf G}^{m,U^{c}}(\mvx, \mvx) = \frac1N \sum_{k=1}^{|U^{c}|} \frac1{\lambda_{U, k}+m}=\int_{[0,1]^{d}}\frac{d\mv\gk}{4\sum_{i=1}^{d}\sin^{2} (\pi \gk_{i})+m} +o_N(1).
        \end{equation*}
    \end{lemma}
\begin{proof}

    There are many methods of proof for this statement, we find the proof using the Domain Markov Property of the massive Gaussian Free field to be the easiest. Starting with $\Phi_{m}$, Let $\Phi_{m}^{U^{c}}$ be an independent Dirichlet MGFF which is identically $0$ on $U$, and let $h_{m}^{U^{c}}$ be equal to $\Phi_{m}$ on $U$, and massive harmonic on $U^{c}$, and they are all related as per \eqref{eq:DMP}.  Now, note that by independence,
    \begin{align}\label{eq:GF1}
    \sum_{\mvx\in U^{c}} {\sf G}^{m,U^{c}}(\mvx,\mvx)=\E\norm{\Phi_{m}^{U^{c}}}_{2}^{2}=\E \norm{\Phi_{m}|_{U^{c}}}_{2}^{2}+\E\norm{h|_{U^{c}}}_{2}^{2}.
    \end{align}
    By definition,
    \begin{align}\label{eq:GF2}
    \sum_{\mvx\in \bT^{d}_{n}} {\sf G}^{m}(\mvx,\mvx)=\E \norm{\Phi_{m}}_{2}^{2}=\E\norm{\Phi_{m}|_{U^{c}}}^{2}_{2}+\E\norm{\Phi_{m}|_{U}}_{2}^{2}. 
    \end{align}
    By translation invariance, for any $A\subset \bT^{d}_{n}$
    \begin{align}\label{eq:GF3}
    \bE \norm{\Phi_{m}|_{A}}_{2}^{2}=\sum_{\mvx \in A}|\Phi_{m}(\mvx)|^{2}=C(m)|A|.
    \end{align}
    By combining Lemma \ref{lem:l_p_h} and \eqref{eq:GF3}, 
    \begin{align}\label{eq:GF4}
    \norm{h|_{U^{c}}}_{2}^{2}\leq C|U|.
    \end{align}
    Finally, subtracting \eqref{eq:GF1} from \eqref{eq:GF2}, and bringing in the bounds from \eqref{eq:GF3} and \eqref{eq:GF4} yields
    \begin{equation*}
        |\frac1N\sum_{\mvx \in U^c} {\sf G}^{m,U^c}(\mvx, \mvx) -{\sf G}^m(\mv0,\mv0)| = O(\log^{11d}(N)/N) = o_N(1).
    \end{equation*}
    
\end{proof}
\begin{lemma}\label{lem:torus_holed_fourth}

    Let $U\subset \dT^{d}_{n}$ be such that $|U|= O((\log N)^{10d})$. $\{\lambda_{U^c,k}\}_{k=1}^{|U^c|}$ be as in \Cref{def:GFF}. Then
    \begin{equation*}
        \frac1N \sum_{k=1}^{|U^c|} \frac1{(\lambda_{U^c, k}+m)^2} =\Omega(1)
    \end{equation*}
\end{lemma}
\begin{proof}
    By Lemma \ref{mass_torus_holed}, we know that 
    \[
    \frac{1}{N}\sum_{\mvx\in U^{c}} {\sf G}^{m, U^{c}}(\mvx,\mvx)=\frac{1}{N}\sum_{k=1}^{|U^{c}|}\frac{1}{\gl_{U^{c},k}+m}\to \int_{[0,1]^{d}}\frac{d\mv\gk}{4\sum_{i=1}^{d}\sin^{2} (\pi \gk_{i})+m}.
    \]
    We may interpret the sum as the integral of the function $(x+m)^{-1}$ against the empirical spectral measure of the Dirichlet Laplacian, given by 
    \[
    \rho_{\gD,U^{c}}=\frac{1}{|U^{c}|}\sum_{k=1}^{|U^{c}|}\gd_{\gl_{U^{c},k}}. 
    \]
    With $\rho_{\gD,U^{c}}$ defined as above, 
    \[
    \frac{1}{N}{\sf G}^{m,U^{c}}(\mvx,\mvx)=\frac{|U^{c}|}{N}\int \frac{1}{x+m}d\rho_{\gD, m}(x). 
    \]
    The intergal above is the Stieltjes transform of $\rho_{\gD,m}$, and the convergence of the Stieltjes transform to a limit guarantees the weak convergence of $\rho_{\gD,m}$ to a limiting probability measure $\rho_{*}$, since our measures are compactly supported which guarantees tightness (see Chapter 2 of \cite{Bern}). The convergence of the Stieltjes transforms holds since $|U^{c}|/N\to 1$. Next, note that the function $(x+m)^{-2}$ is bounded and continuous on the interval $[0,4d]$. Thus, 
    \[
    \frac{1}{|U^{c}|}\sum_{k=1}^{|U|^{c}}\frac{1}{(\gl_{U^{c},k}+m)^{2}}=\int\frac{d\rho_{\gD,U^{c}}}{(x+m)^{2}}\to \int\frac{d\rho_{*}}{(x+m)^{2}}\text{ as }N\to \infty. 
    \]
    which completes the proof.
\end{proof}

\subsection{Zero average GFF on the torus}\label{sec:zero_avg_gff}
In this subsection, we consider the case $G = \bT_n^d$ and $m=0$. Note that in this case $(-\Delta)$ is not positive definite and has an eigenvalue 0 corresponding the eigenvector which is constant on the torus. Thus to define the Green's function, we need to project onto the orthogonal complement of the constant vector.

\begin{defn}\label{def:zero_avg_gff}
    Let $(\phi_{\mvk})_{\mvk \in \bT_n^d}$ and $(\lambda_{\mvk})_{\mvk \in \bT_n^d}$ be the orthonormal basis of eigenvectors and their eigenvalues as specified in \eqref{eq:evectors} and \eqref{eq:evalues}. Then the \textbf{zero-average} GFF on the torus is defined as 
    \begin{equation*}
        \Phizero^{N}:=\sum_{\mvk \in \bT_n^d \setminus \{\mv0\}} \frac{X_{\mvk}}{\sqrt{\lambda_{\mvk}}}\phi_{\mvk}
    \end{equation*}
    The zero-average Green's function is defined to be the covariance of the zero-average GFF:
    $$
    \Gzero^{N}(\mvx, \mvy) = \bE(\Phizero^{N}(\mvx)\overline{\Phizero^{N}(\mvy)}).
    $$
    
\end{defn}
It is not too hard to see that the density of a zero-average GFF is simply a GFF conditioned to have sum 0. Or in other words, for any bounded continuous function $\ff:\bC^N \to \bR$,
\begin{equation}
    \bE(\ff(\Phizero^N)) = \frac1{\cZ_{0\text{-}\mathsf{avg}}}\int_{\bC^N}\ff(\psi)\exp(-\|\nabla \psi\|_2^2 1_{\la \psi, \mv1 \ra = 0})
\end{equation}
where $\mv1$ is the function which is identically equal to 1 on the torus. One of the first investigations into the properties of the zero average free field was carried out by Ab\"acherli in \cite{A19}, who established an extremely useful coupling between $\Phizero^{N}$ and its infinite volume counterpart to be introduced in the next subsection. 
\subsection{Infinite volume GFF}
We assume $d \ge 2, m \ge 0$ and furthermore, we need to assume $m>0$ only if $d = 2$. The massive Green's function, as defined in \eqref{eq:gfrw}, is still a well defined object when the killed random walk on $\bZ^{d}$ is used instead. In fact, when $d\geq 3$, we may use the simple random walk with no killing, since it is transient and thus the sum in \eqref{eq:gfrw} is finite almost surely. Thus, ${\sf G}^{\bZ^{d}}_{m}$ is well defined as a positive definite kernel, as is ${\sf G}^{\bZ^{d}}$. The centered Gaussian processes with covariance given by these kernels, denoted $\Phi_{m}^{\bZ^{d}}$ and $\Phi^{\bZ^{d}}$ are called the infinite volume massive GFF and infinite volume GFF respectively. To summarize, they are defined by  
$$
\bE(\Phi^{\bZ^d}_m(\mvx) \overline{\Phi^{\bZ^d}_m(\mvy)}) = {\sf G}^{\bZ^d}_m(\mvx, \mvy).
$$ 
Moreover, $\Phi_{m}^{\bZ^{d}}$ is the local limit of its torus counterpart. It is well known that the limit as $N \to \infty$ of ${\sf G}^{N}_{m}(\mvx, \mvy)$ is ${\sf G}_{m}^{\bZ^{d}}$, which by Wick's theorem and the method of moments is sufficient to verify local convergence.  The massless GFF also arises as a local limit, but as discussed earlier, we cannot immediately impose the condition that $m=0$, we need to work with the zero average Green's function. 
\begin{lemma}[Aru, Korzhenkova Proposition 2.13 \cite{ARU}]\label{lem:gfconv}
 As $N\to \infty$, there exists $C>0$ such that for any chosen lattice point $\mvx\in \bT^{d}_{n}$,
   \[
   |\Gzero^{N}(\mv0,\mvx)-{\sf G}^{\bZ^{d}}(\mv0,\mvx)|\leq CN^{\frac{2}{d}-1}.
   \]
\end{lemma}
We emphasize the crucial point in \Cref{lem:gfconv}, the constant $C$ is independent of $\mvx$.
The convergence of the finite volume Green's function to its infinite volume counterpart has been known for some time, but the rate of convergence in Lemma \ref{lem:gfconv} was put forth by \cite{ARU}. Let $\fg:A\to \bC$. Note that for all $N$, $\la \Phizero^{N} ,\fg\ra$ is a complex gaussian random variable with variance given by $\la \fg,\Gzero^{N}\fg\ra$. Since all moments of a gaussian are determined by the variance, proving that $\Gzero^{N}$ converges to ${\sf G}^{\bZ^{d}}$ verifies the convergence in distribution, and in particular that 
\begin{align}
\left|\bE \bigl(\cP(\la \fg , \Phizero^{N} \ra)\bigr)-\bE \bigl(\cP(\la \fg , \Phi^{\bZ^{d}} \ra)\bigr) \right|\to 0\text{ as }N\to \infty,     
\end{align}
for all $\fg:A\to \bC$ and $\cP(z)$ of the form $(\fR z)^{q}\cdot (\fI z)^{q'}$, with $q,q'\in \bN$. With the exception of the sharpened Green's function estimates of Lemma \ref{lem:gfconv}, the other results mentioned in this section are well known in the literature and we again direct the reader to Chapter 8 of \cite{FL} for reference.

\section{Further estimates on the spherical model}\label{sec:spherical_details}
In this section, we gather some results on the spherical model which are needed to prove \Cref{thm:main_summary}. These results are proven in \Cref{sec:spherical_proofs}.  The reader should recall the definition of $\Psisph^\theta$ from \eqref{eq:spherical}.
First, the $\theta \ge C_d$ case. Along with the second item of \Cref{thm:main_sph} we also need the following uniform tail bound.

\begin{lemma}\label{lem:sphericaltail2}
Let $\theta\geq C_{d}$ and $b>\log N$. Then for $N$ sufficiently large, there exists a constant $c$ such that 
\[
\bP(\norm{\Psi_{{\sf sph},N}^{\theta}}_{\infty}>b)\leq e^{-cb^{2}}.
\]
\end{lemma}
The key point in \Cref{lem:sphericaltail2} is that $b$ might diverge with $N$, and hence not captured by the local limit identification of \Cref{thm:main_sph}. The proof of \Cref{lem:sphericaltail2} is postponed to \Cref{sec:spherical_proofs}.

Next, the $\theta<C_d$ case. In the proof of \Cref{thm:main_summary}, we need the following generalization of the first item of \Cref{thm:main_sph} for the spherical law with boundary conditions.  Let $U$ be a subset of the set of vertices of $\bT_n^d$, $f: U \to \bC$ and $\gamma \in (0,1)$.
The spherical law with general boundary condition and general mass cutoff is the field $\Psisph^{f, \theta,\gamma, U^{c}}$  whose law is dictated by the following equality for all bounded continous functions $\ff: \bC^{U^{c}} \to \bR$.
\begin{equation}
\bE(\ff(\Psisph^{f, \theta,\gamma, U^{c}})):=\frac{1}{\cZ_{\text{sph},N} (f, \theta,\gamma, U^{c})}\int_{\bC^{U^{c}}} \ff(\psi) \exp\left(-\theta \norm{\nabla \psi|_{U^{c}}}_{2}^{2}\right)\mv1_{\norm{\psi|_{U^{c}}}_{2}^{2}\leq \gamma N}1_{\psi|_{ \partial U}\equiv f_{|\partial U}}\vd \psi.\label{eq:density_spherical_boundary_general_mass}
\end{equation}
The gradient above is calculated in the graph $G_U$ and $\cZ_{\text{sph},N} (f, \theta,\gamma, U^{c})$ is the appropriate partition function. If $\gamma=1$, we write $\Psisph^{f, \theta,1, U^{c}} = \Psisph^{f, \theta, U^{c}}$. Also note the scaling relation
\begin{equation}
    \Psisph^{f,\theta, \gamma,U^c} = \Psisph^{f,\theta\gamma,U^{c}}.\label{eq:scaling_spherical}
\end{equation}

Recall the definition of local weak convergence from \Cref{sec:intro} and the definition of massive GFF in $\bZ^d$ from \Cref{sec:gff}

\begin{thm}\label{lem:llimit1}
Suppose $\theta<C_d$. Let $U\subset \dT^{d}_{n}$ be such that $d(\mv0,U)\ge N^{0.1}$, $|U|= O((\log N)^{10d})$, and $\|f|_{\partial U}\|_2^2 =o(N)$. Then  $$\Psisph^{f, \theta, U^{c}} \xrightarrow[N \to \infty]{(d)}\theta^{-\frac{1}{2}}\cdot \Phi^{\bZ^d}_{m}$$ where the convergence above is in the local weak sense and $m$ solves 
\begin{equation}
\int_{[0,1]^{d}}\frac{d\mv\gk}{4\sum_{i=1}^{d}\sin^{2}(\pi \gk_{i})+m} = \theta.\label{eq:mass}
\end{equation}
\end{thm}
As mentioned in the introduction, this result generalizes the results in Aru and Korzhenkova \cite{ARU} for the massive phase. In terms of the boundary condition, we believe this result is sharp as we can expect a loss of mass in the local limit if, for example, $U$ is a singleton with boundary value $\sqrt{N}$.

We also need the following uniform bound analogous to the $\theta \ge C_d$ case.  For the application, we need a bound with more general boundary conditions, which might diverge with $N$. Recall that $\Delta_U$ is the restriction of the Laplacian to $U$ as defined in \Cref{sec:gff}.

\begin{lemma}\label{lem:spherical_tail}
Fix $\theta<C_d$.  Let $ U\subset \bT_n^d$ and let $f:U \to \bC$ be such that $\|f|_{\partial U}\|_2^2 =o(N)$ and $|U| = O((\log N)^{100d})$. Let $\lambda_1, \ldots, \lambda_{|U^{c}|}$ denote the eigenvalues corresponding to an orthonormal basis of eigenvectors of $-\Delta_{U^{c}}$. Let $h$ denote the massive harmonic extension of $f$ onto $U^{c}$ with mass $m$ given by \eqref{eq:mass}. Let $m$ be defined as 
\begin{equation}
     \sum_{i=1}^k \frac1{\lambda_i+m} + \theta\|h|_{U^{c}}\|_2^2 = \theta N\label{eq:mass_finite_def}
\end{equation}
Let $h^{m_N,f}$ denote the massive harmonic extension of $f$ with mass given by \eqref{eq:mass_finite_def}.
Then there exists a constant $c,C>0$ such that for all $\mvy \in U^{c}$ and all $b\ge 2|h(\mvy)|$,  
$$
\bP(|\Psisph^{f, \theta, U^{c}}(\mvy) -h^{m,N,f}(\mvy)|  > b) \le \exp(-cb^2).
$$
\end{lemma}
The proof if postponed of there results is postponed to Section \ref{sec:spherical_proofs}, along with that of \Cref{thm:main_sph}.
\section{Local limit of focusing NLS}

\subsection{Tempering the NLS field}\label{sec:tempering}
The goal of this section is to prove \Cref{prop:U} which roughly states that the NLS field cannot take $O(\sqrt{N})$ values outside a small set of polylogarithmic order. This result is general and works for all $p>2$ and all $\nu, \theta>0$. The amount of mass that is concentrated in this small set depends on whether $\nu<\nu_c(\theta)$ or $\nu>\nu_c(\theta)$.

We start two quick lemmas which we shall use multiple times in this article.
\begin{lemma}\label{lem:lp_largest}
Fix $p\ge2$.
    Let $f:A \to \bC$ be such that $\|f\|_2^2 \le a$ and $\|f\|_\infty <b$. Then $$\|f\|_p^p \le b^{p-2}a.$$
\end{lemma}
\begin{proof}
The largest possible $\ell^{p}$ norm is achieved when the function $f$ takes its largest value on as many lattice sites as possible. In other words, $|f|=b$ on $a/b^{2}$ many lattice sites and $0$ everwhere else. This is a consequence of the following inequality, for all $s>t>0$ and $p>2$ we have that 
\[
s^{p/2}+t^{p/2} \le (s+\alpha)^{p/2}+(t-\alpha)^{p/2}
\]
for all $\alpha\in [0,t]$ (which can be checked by taking the derivative with respect to $\alpha$.)
To see how to use this inequality, start from a configuration with equal mass at every site (satisfying the $\ell^2$-constraint.) Then try to pile up as much mass onto one site from other sites given the $\ell^\infty$-constraint. The $\ell^p$-norm only increases for this operation. The details are left to the reader.
\end{proof}
\begin{lemma}\label{lem:lp_smallest}
Fix $p \ge 2$.
Let $f:A\to \bC$ such that $\norm{f}^{2}_{2}=a$ $\norm{f}_{\infty}=b$, and $b\leq \sqrt{a}$. Then 
\[
\norm{f}_{p}^{p}\geq b^{p}+(a-b^{2})^{p/2}\cdot (|A|-1)^{-(p-2)/2}. 
\]
\end{lemma}
\begin{proof}
There has to be atleast one $\mvx\in A$ such that $|f_{\mvx}|=b$, which contributes $b^{p}$ to $\norm{f}_{p}^{p
}$. We have mass $a-b^{2}$ which remains to be distributed across the remainder of the lattice sites. The configuration that yields the lowest $\ell^{p}$ norm is thus realised by evenly distributing the remainder of the mass across the other lattice sites. This fact has a proof analogous to  \Cref{lem:lp_largest}.
\end{proof}
The following corollary will be useful later in controlling the exponential tilt by the $\ell^{p}$ norm: 
\begin{cor}\label{cor:betachoice}
Let $p>4$ and $0<\gb<\frac{(p-2)-2}{2(p-2)}$. Then for all $\psi\in \ttwo$ satisfying $\norm{\psi}_{2}^{2}\leq N$ and $\norm{\psi}_{\infty}<N^{\gb}$,
\[
\frac{1}{N^\frac{p-2}{2}}\norm{\psi}_{p}^{p}\to 0,
\]
as $N\to \infty$. 
\end{cor}
\begin{proof}
By Lemma \ref{lem:lp_largest}, we know that for such $\psi$, 
\[
\frac{1}{N^{\frac{p-2}{2}}}\norm{\psi}_{p}^{p}\leq \frac{1}{N^{\frac{p-2}{2}}}\cdot N^{\gb(p-2)+1}=N^{(p-2)\gb+(1-\frac{p-2}{2})}.
\]
For our choice of $\beta,p$, the exponent is negative. This completes the proof. 
\end{proof}
Controlling the tilt requires the following notion introduced in \cite{DKK}. Given a function $\psi\in \ttwo$ with $\norm{\psi}_{2}^2< N$ and $\eps>0$, we define the \textbf{separating set} $U(\psi,\eps)\subset \bT^{d}_{n}$ to be the set obtained via the following procedure: Let $U_0=\{\mvx: |\psi_{\mvx}|^2\ge \eps N\}$. Clearly $|U_0|\le 1/\eps$. We inductively define $U_{i}$ by the successive addition of the 2-step outer boundary, \ie\ $U_{i}=\{  \mvx \in V: d(\mvx,U_{i-1})\leq 2 \}$ and $B_i=U_i\setminus U_{i-1}$ for all $i\ge 1$. Take $B_0=U_0$. Note for any $k \ge 1$,
	\[
		\sum_{i=1}^k\sum_{\substack{ \mvx \in B_{i}}} |\psi_{\mvx}|^2
		\le \sum_{\mvx} |\psi_{\mvx}|^2\le N.
	\]
    By pigeonhole principle, there must be an $i \in \{1,\ldots, \lceil\frac{10}{\eps} \rceil\}$ such that $\|\psi|_{B_i}\|^2_2\le \frac{\eps}{10} N$. Let $i_0$ be the smallest such $i$ and define  $U=U(\psi,\eps)=U_{i_0-1} \cup \{\mvx: d(\mvx, U_{i_0-1}) = 1\}$.

	Note that $$|U|\le (2\lceil\frac{10}{\eps} \rceil+1)^d |U_0|\le(21)^d\eps^{-d-1}\leq \eps^{-d-2}$$ 
    if $\eps<1/(21)^d$ (which we assume from now on). 
    We compile some important properties of the separating set below, the proofs are almost immediate consequences of the definitions.
    \begin{lemma}\label{def:separating_set}
       A  separating set  set $U(\psi, \eps) \subset \bT_n^d$ satisfies the following properties:
\begin{enumerate}[i.]
    \item  $|U|\le \eps^{-d-2},$
    \item  $|\psi_{\mvx}|^2< \eps N \text{ for all } \mvx\notin U,$ 
    \item  $\sum_{\mvx\in B_{i_{0}}}|\psi_{\mvx}|^2 = \sum_{\mvx \in U_{\text{int}} \cup U_{\text{out}}}|\psi_{\mvx}|^2 < \frac{\eps N}{10},$ where $U_{\text{int}}$ and $U_{\text{out}}$ are the inner and outer shells of $B_{i_0}$. To be precise, $U_{\text{int}} = \{\mvx \not \in U_{i_0-1}: d(\mvx, U_{i_0-1}) = 1\}$ and $U_{\text{out}} = B_{i_0} \setminus U_{\text{int}}$. 
\end{enumerate} 
    \end{lemma}
The discussion before the definition guarantees the existence of the separating set. 
The last point in the \Cref{def:separating_set} is important because it implies that the gradient of $\psi$ on the boundary of $U$ can be controlled.

Now, recall our protagonist $\Psi^{\nu,\theta}_N$, the focusing NLS field with parameters $\nu$ and $\theta$ from \eqref{eq:psi}. It is worthwhile to spend some time describing the conditional law of the NLS field given the separating set and the field values on it.
For some $A \subset \bT_n^d$ and $f:A \to \bC$, let $\cU = \cU(A,f, \eps)$ be the following event:
\begin{itemize}
    \item $U(\Psi_N^{\nu, \theta}, \eps_N)=A$
    \item $\Psi_N^{\nu, \theta}|_A \equiv f$.
\end{itemize}
Let $\alpha = \frac1N\|f|_A\|_2^2$.  Let $\Psi_{N, \eps}^{\nu,\theta, 1-\alpha_N}$ denote the random field distributed as $\Psi^{\nu, \theta}_N$ in $A^c$ conditioned on the event $\cU(A,f, \eps)$.

\begin{lemma}\label{lem:conditional_density_separating_set}
    The random field $\Psi_{N, \eps}^{\nu,\theta, 1-\alpha,f,A}$, has a density given by
\begin{equation}
    \frac1{\cZ_N'}\exp\left[\theta \left(\nu_N\norm{\psi}_{p}^{p}-\norm{\nabla \psi}_{2}^{2} \right)\right]\mv1_{\norm{\psi|_{A^c}}_{2}^{2}\leq (1-\alpha) N}\mv1_{\|\psi|_{A^c}\|_{\infty} \le \sqrt{\eps N}}\mv1_{\psi|_{A} \equiv f|_{A}  }, \qquad \psi \in \bC^{\bar A^c} \label{eq:cond_density_linfty1}
\end{equation}
for an appropriate partition function $\cZ_N'= \cZ_N'(\nu,\theta, 1-\alpha,f,A, \eps)  $. 
\end{lemma} 
\begin{proof}
    This simply follows from the fact that we chose the first $i_0$ while defining the separating set, which preserves the Markov property. The remaining details are straightforward consequence of the density of the NLS measure.
\end{proof}
\begin{proposition}\label{prop:U}
Fix $\nu, \theta>0$ and $p>2$.
 Let $\eps_N \in (e^{-\sqrt{\log N}},1)$  with $\eps_N \to 0$ as $N \to \infty$, and let $U(\Psi^{\nu,\theta}_N, \eps_N)$ be the separating set as described above.   
     Then there exists $\delta_N \to 0$ as $N \to \infty$ such that  
    $$
    \bP\left(\frac1N\|\Psi^{\nu, \theta}_N |_U \|^2_2 \in (a - \delta_N, a+\delta_N) \text{ for some }a \in \mathscr{M}(\theta, \nu)\right) \ge 1-e^{-N\sqrt{\eps_{N}}/2}
    $$
    In particular, if $\mathscr{M}(\theta, \nu) = \{0\}$ then 
    $$
    \bP\left(\|\Psi^{\nu, \theta}_N  \|_\infty < \sqrt{\max\{\eps_N, \delta_N\} N}\right) \ge 1-e^{-N\sqrt{\eps_{N}}/2}.
    $$
    \end{proposition}  

    \begin{proof}
    The second assertion immediately follows from the first and the definition of the separating set, so we concentrate on the first assertion.
In this proof we drop superscript $\nu, \theta$ from $\Psi^{\nu,\theta}$ for the sake of brevity. Recall the notation  of the energy functional $F$ from \eqref{eq:F}. 
    Let $$\mathscr{M}(\eps_N)  := \{x \in [0,1), F(x)-F_{\min} < \sqrt{\eps_N} \}.$$  
    We claim that
    \begin{equation}
        \bP\left(\frac{1}{N}\norm{\Psi_{N}|_{U}}_{2}^{2}\notin \mathscr{M}(\eps_{N})\right)\leq e^{-N\sqrt{\eps_{N}}/2}\label{eq:l2_atypical}
    \end{equation}
    for $N$ sufficiently large. This bound follows from the upper and lower bounds of the partition function from \cite{DKK}*{Section 6}. We now provide a sketch of the proof of these bounds \eqref{eq:l2_atypical} for the sake of completeness, before completing the proof. Let us define
    \[
    \cG(\eps_{N}):=\{\psi:\frac{1}{N}\norm{\psi|_{U(\psi,\eps_{N})}}_{2}^{2}\notin \mathscr{M}(\eps_{N}) \}. 
    \]
    Note, 
    \[
    \bP(\Psi_{N}\in \cG(\eps_{N})):=\frac{1}{\cZ_{N}}\int_{\cG(\eps_{N})\cap \{\norm{\psi}_{2}^{2}\leq N\}}e^{-\theta \cH_{N}(\psi)}\vd\psi.
    \]
    For a given $A\subset \bT^{d}_{n}$, whose size is bounded above by $\eps^{-d-2}_{N}$, let $\cC(A,\eps_{N})$ denote the collection of $\psi$ such that $A = U(\psi, \eps_N)$. We can now split $\cH_N(\psi)$ as  
    \begin{equation*}
        \cH_N(\psi) = \cH_N(\psi|_A) + \cH_{N}(\psi|_{A^{c}})+ \sum_{\mvx\sim\mvy, \mvx\in A, \mvy\in A^{c}}|\psi_{\mvx}-\psi_{\mvy}|^{2}
    \end{equation*}
    Using property  iii. of $U$ above, the inequality $|a-b|^2 \le 2(|a|^2+|b|^2)$ and employing an union bound over all possible separating sets of size at most $\eps_N^{-d-2}$, we obtain the following upper bound
    \[
    \bP(\Psi_{N}\in \cG(\eps_{N}))\leq \frac{e^{\theta 3N\eps_{N}}}{\cZ_{N}}\sum_{A:|A|\leq \eps_{N}^{-d-2}}\int_{\cG(\eps_{N})\cap \{\norm{\psi}_{2}^{2}\leq N\}\cap \cC(A,\eps_{N})}e^{-\theta\cH_{N}(\psi|_{A})-\theta\cH_{N}(\psi|_{A^{c}})}\vd\psi
    \]
    With $\xi:=1-\frac{1}{2d}$ (this choice is governed by the error bounds from the entropy terms and follows \cite[Section 6]{DKK}), define 
    \[
    \cB_{i,A} =\left\{\zeta \in  \bC^A: \norm{\zeta}_{2}^{2}\in \bigl[iN^{\xi}, (i+1)N^{\xi} \bigr)\right\},
    \]
    and 
    \[
    	\tilde{\cB}_{j,A^{c}}:=\left\{\zeta \in \bC^{A^c}  : \norm{\zeta}_{2}^{2}\in (jN^{\xi},(j+1)N^{\xi}),\text{ }\norm{\zeta}_{\infty} \leq \sqrt{\eps_{N} N} \text{ and } \norm{\zeta|_{\partial (A^c)}}_{2}^{2}\leq 3\eps_{N} N\right\}. 
    \]
    By definition of $\cC(A,\eps_{N})$ and $\cG(\eps_{N})$, the following containment holds:
    $$
    \{\|\psi\|_2^2\le N\} \cap \cC(A,\eps_{N})\cap\cG(\eps_{N}) \subseteq \bigcup_{\stackrel{i+j\leq N^{1-\xi}}{ iN^{\xi-1} \not \in \mathscr{M}(\eps_N)}} \cB_{i,A}\times \tilde{\cB}_{j,A^{c}}.
    $$
    This yields
    \begin{align}\label{eq:5.2UB1}
    \bP(\Psi_{N}\in \cG(\eps_{N}))\leq\frac{e^{3N\eps_{N}\theta}}{\cZ_{N}}\sum_{A:|A|\leq \eps^{-d-2}}\sum_{\stackrel{i+j\leq N^{1-\xi}}{ iN^{\xi-1} \not \in \mathscr{M}(\eps_N)}}\int_{\cB_{i,A}}e^{-\theta \cH_{N}(\psi|_{A})}\vd \psi|_{A}\cdot \int_{\tilde{\cB}_{j,A^{c}}}e^{-\theta\cH_{N}(\psi|_{A^{c}})}\vd \psi|_{A^{c}}
    \end{align}
    It is a consequence of Lemmas 6.4, 6.5 and the proof of Lemma 6.2 of \cite{DKK} that we can further bound \eqref{eq:5.2UB1} above by 
    \begin{align} \label{eq:5.2UB2}
  \frac{1}{\cZ_{N}}\cdot e^{10N\eps_{N}+o(N\eps_{N})}\cdot \left(\frac{\pi}{\theta}\right)^{N}\sum_{\stackrel{i=0}{ iN^{\xi-1} \not \in \mathscr{M}(\eps_N)}}^{i^*}e^{-NF(i\cdot N^{\xi})}
  \end{align}
  where $i^{*}=\lceil  N^{1-\xi}\rceil$. Specifically, Lemma 6.4 addresses the intergral of $\psi|_{A}$, Lemma 6.5 addresses the integral over $\psi|_{A^{c}}$, and the proof of Lemma 6.2 demonstrates how to combine the respective bounds so as to make the saddle point method applicable. For ease of reference, we mention the differences in notation now, $\eps_{N}$ here is $s_{N}$ in \cite{DKK}, and $N^{\xi}$ here is replaced by a general sequence $\gk_{N}$ in \cite{DKK}, where an exact rate is specified later. 
 By definition, we have that
  \[
  F(a)>F_{\min}+\sqrt{\eps_{N}} \text{  for all  }a\notin \mathscr{M}(\eps_{N}).  
  \]
  Thus, a further upper bound of \eqref{eq:5.2UB2} is obtained as 
  \[
  \frac{1}{\cZ_{N}} e^{10N\eps_{N}+o(N\eps_{N})}\cdot N^{1-\xi}\left(\frac{\pi}{\theta}\right)^{N}\cdot e^{-NF_{\min}}\cdot e^{-N\sqrt{\eps_{N}}}
  \]
   To extract the required upper bound on $\bP(\Psi_{N}\in \cG(\eps_{N}))$, the last part required is the lower bound on $\cZ_{N}$, which by Lemma 6.6 of \cite{DKK}, for $N$ sufficiently large is given by
   \[
   \cZ_{N}\geq e^{o(N\eps_{N})}\cdot \left(\frac{\pi}{\theta}\right)^{N}\cdot e^{-NF_{\min}}
   \]
   Note that the maximal order of the errors is $e^{11N\eps_{N}}$ for all large enough $N$, which is beaten by $e^{-N\sqrt{\eps_{N}}}$ as $N\to \infty$ (since $\eps_N \to 0$ by our assumption). Combining our bounds we note that for $N$ sufficiently large,
   \[
   \bP(\Psi_{N}\in \cG(\eps_{N}))\leq e^{-N\sqrt{\eps_{N}}/2}.
   \]
   Next, we need to verify the existence of a $\delta_N \to 0$ such that 
    \[
    \mathscr{M}(\eps_{N})\subseteq \{x:\exists a\in \mathscr{M}(\theta,\nu)\text{ s.t. }|x-a|<\gd_{N}\}.
    \]
    Suppose there exists a sequence $N_{k}\in \bN$, an $x_{N_{k}}$ and a fixed $\gd>0$ such that $x_{{N_{k}}}\in \mathscr{M}(\eps_{N_{k}})$ and $d(x_{N_{k}},\mathscr{M(\theta,\nu)})>\gd$. The sequence $x_{N_{k}}$ is bounded, and every limit point $x_{*}$ is at least distance $\gd$ away from $\mathscr{M}(\theta,\nu)$. By definition, $x_{*}\notin \mathscr{M}(\theta,\nu)$ since $d(x_{*},\mathscr{M}(\theta,\nu))>\gd$. On the other hand $F(x_{k_{N}})-F_{\min}=|F(x_{k_{N}})-F_{\min}|<\eps_{k_{N}}$, and by continuity $G(x_{*})=F_{\min}$, which then implies that $x_{*}\in \mathscr{M}(\theta,\nu)$, a contradiction.
\end{proof}
    \subsection{Free energy landscape}\label{sec:landscape}
    In this section we provide required details about the behavior of the functions $I, W$ and $F$ from \Cref{sec:intro}. Many of these results are lifted from \cites{DKK,WEI99} and forms the core input in the phase transition behavior of the NLS field. We emphasize that results that are due to others are explicitly attributed to them, and the those that are not attributed are proved by us. Recall that 
    \[
    \lim _{N\to \infty}\frac{1}{N}\log \cZ_{N}(\theta, \nu)=\frac{\pi}{\theta}-\inf_{a\in (0,1)}\{W(\theta(1-a))-\nu^{-1}I(\nu a)\}
    \]
    \[
    I(a):=\inf\{\norm{\nabla\psi}_{2}^{2}-\frac{2}{p}\norm{\psi}_{p}^{p}:\psi\in \ell^{2}(\bZ^{d}),\text{ }\norm{\psi}_{2}^{2}=a\},
    \]
and 
\[
W(b):=\frac{1}{N}\log \int_{\norm{\psi}_{2}^{2}\leq Nb}\exp(-\norm{\nabla \psi}_{2}^{2})\vd\psi.
\]
The free energy functional, which defines the free energy landscape is 
\[
F_{\theta, \nu}(a)=W(\theta(1-a))+\theta \nu^{-1}I(\nu a).
\]
We start with the function $W(b)$ because the required properties are relatively short. The fact that the defining limit exists and is continuous in $b$ is proved in Section 4 of \cite{DKK}. As a broad summary of Section 4, $W$ is connected to the discrete Gaussian Free Field and is identified as  free energy of the field conditioned to have given mass. As such, it can be expressed as a Legendre transform. For the purposes of this article, the required results are as follows. 
\begin{lemma}[Dey, Kirkpatrick, K. Lemma 4.7, \cite{DKK}]
The function $W:(0,\infty)\to (0,\infty)$ satisfies the following properties:
\begin{enumerate}
    \item Convex, continuously differentiable and non increasing. 
    \item Constant for all $b \in [C_{d},\infty)$, where $C_{d}={\sf G}^{\bZ^{d}}(\mv0,\mv0)$. 
    \item Diverges as $b\to 0$, that is $\lim _{b\to 0}W(b)=\infty$. 
\end{enumerate}
\end{lemma}
For the second ingredient of the free energy landscape, we summarize several properties of the function $I$, which is the energy of a discrete soliton with given mass. The main result of \cite{WEI99} was to verify that there is a threshold for soliton formation for certain values of $p$, that is minimizers of $\norm{\nabla \psi}_{2}^{2}-2p^{-1}\norm{\psi}_{p}^{p}$ in $\ztwo$ exist only when the mass is larger than a threshold. Moreover, that this threshold is defined in terms of a functional inequality on the lattice. It is not hard to show that $I(a)\leq 0$ for all $a>0$, all one needs to do is consider functions that are increasingly spread out. Define

\begin{align}\label{eq:Rp1}
			R_{p}:=\inf\{a>0:I(a)<0\}. 		\end{align}

   \begin{thm}[Weinstein~\cite{WEI99}]\label{thm:wei}
	The following hold.
	\begin{enumeratei}
		\item An explicit minimizer of \eqref{def:Ifn}  in $\ell^{2}(\dZ^{d})$ exists when $I(a)\in (-\infty,0)$.
		\item The threshold $R_{p}$ defined above can be expressed as
  		\begin{align}\label{eq:Rp}
			\frac2{p} R_{p}^{(p-2)/2} = \inf_{\psi } \left\{ \frac{ \norm{\psi}_{\ztwo}^{p-2}\cdot \norm{\nabla \psi}_{\ztwo}^2 }{ \norm{\psi}_{\zp}^{p} } \right\}.
		\end{align}
        \item For $2<p<2+4/d$, we have $R_{p}=0$, \ie\ $I(a)<0$ for all $a>0$.
		\item For $p > 2+4/d$,  $R_{p}>0$, \ie\  $I(a)<0$ if $a>R_{p}$ and $I(a)=0$ if $a<R_{p}$. 
	\end{enumeratei}
\end{thm}
\begin{rem}
Combined with the next Lemma, we can in fact conclude that $I(R_{p})=0$.
\end{rem}

Equation~\eqref{eq:Rp} interprets $R_{p}$ as reciprocal to the best possible constant for a discrete Gagliardo-Nirenberg-Sobolev type inequality to hold~\cite{WEI99}. Basic analytic properties of $I$ were verified in \cite{DKK}. The only part required here is the following. 
\begin{lemma}[Dey, Kirkpatrick, K. Lemma 5.9 \cite{DKK}]\label{lem:Ider}
The function $I:[0,\infty)\to (-\infty, 0]$ is non increasing and differentiable. Further, $I'(a)\leq a^{-1}\cdot I(a)$ 
\end{lemma}
As discussed in Section \ref{sec:heuristics}, the free energy landscape describes the likelihood, on an exponential scale, of a certain fraction of the mass being concentrated. As such the minimizers of $F$ represent the most likely values of this fraction. The set of minimizers $\mathscr{M}(\theta, \nu)\subset [0,1)$ characterizes the phases. As $\theta$ and $\nu$ are varied, we have two distinct phases. 
\begin{thm}[Dey, Kirkptrick, K. Theorem 2.4 \cite{DKK}]\label{thm:phasecurve}
There exists a function $\nu_{c}:(0,\infty)\to (R_{p},\infty)$ which is continuous and strictly decreasing, such that 
\begin{enumerate}
    \item For $\nu<\nu_{c}$, ${\mathscr M}(\theta,\nu)$ is a closed interval containing $0$, and is $\{0\}$ when additionally $\theta\leq C_{d}$. This phase is referred to as \textbf{dispersive} or subcritical.
    \item For $\nu>\nu_{c}$, $\mathscr{M}(\theta, \nu)$ is bounded away from $0$. This phase is referred to as \textbf{solitonic} or supercritical.  
    \item As $\theta\to 0$, $\nu_{c}\to \infty$, and as $\theta\to \infty$, $\nu_{c}\to R_{p}$.
    \end{enumerate}
\end{thm}    

The characterization of the dispersive phase required for this article is slightly different and a little more refined as compared to \cite{DKK}. The following two lemmas are to verify this characterization.  
\begin{lemma}\label{lem:sub_theta_nu}
    If $\nu<\nu_c$ and $\theta\leq C_d$, then $\mathscr{M}(\nu, \theta) = \{0\}$. 
\end{lemma}
\begin{proof}
    Since $p>4$, \Cref{thm:wei}, we know that $I(\nu a)=0$ on the interval $[0,\nu^{-1}R_{p}]$. Further, since $\theta\leq C_{d}$, $W(\theta(1-a)),(W(\theta(1-a)))'>0$ for all $a\in (0,1)$. Thus, $F$ is strictly increasing in the interval $(0,\nu^{-1}R_{p})$ and admits no minima on this interval. Suppose there exists an $a_{*}>\nu^{-1}R_{p}$ such that $F_{\theta,\nu}(a_{*})=F_{\theta,\nu}(0)$. Observe that for any $\nu'>\nu$, $F_{\theta,\nu'}(0)=F_{\theta,\nu}(0)$, on the other hand $F_{\theta,\nu'}(a_{*})<F_{\theta,\nu'}(0)$ since $\nu^{-1}I(\nu a)$ is strictly decerasing in $\nu$ when $\nu a>R_{p}$ (see Theorem \ref{thm:wei}). Thus, for any $\nu'>\nu$, $(\theta,\nu')$ are supercritical, and thus $\nu=\nu_{c}$, which is a contradiction. 
\end{proof}
\begin{lemma}\label{lem:theta>cd}
Let $\theta>C_{d}$, $\nu<\nu_{c}(\theta)$. Then $\mathscr{M}(\theta,\nu)$ is the closed interval $[0,1-C_{d}/\theta]$. Moreover, $1-C_{d}/\theta<R_{p}/\nu$. 
\end{lemma}
\begin{proof}
Since $\theta>C_{d}$, we know that $W(\theta(1-a))=W(\theta)$ for all $a\in [0,1-C_{d}/\theta]$. Next, suppose $R_{p}/\nu<1-C_{d}/\theta$, then there exists $a\in (R_{p}/\nu,1-C_{d}/\theta)$ such that $F(a)<W(\theta)$, which contradicts $\nu<\nu_{c}$. If $1-C_{d}/\theta=R_{p}/\nu$, then for any $\nu'>\nu$, $(\theta, \nu')$ is a supercritical pair, implying that $\nu=\nu_{c}$. Thus, we know that $1-C_{d}/\theta< R_{p}/\nu$. For all $a\in (1-C_{d}/\theta,R_{p}/\nu)$, $F(a)=W(\theta(1-a))>W(\theta)$. Finally, suppose there exists an $a>R_{p}/\nu$ such that $F(a)=W(\theta)$. Again, just as in Lemma \ref{lem:sub_theta_nu}, for any $\nu'>\nu$, $(\theta,\nu')$ is a supercritical pair, which contradicts $\nu<\nu_{c}$. This shows that $\mathscr{M}(\theta,\nu)=[0,1-C_{d}/\theta]$.     
\end{proof}
Proving that the local limit in the supercritical phase is massive, requires a brief bridging step, we verify this now.  
\begin{lemma}\label{lem:super_to_sub}
    Let $\nu > \nu_c(\theta)$ and $a \in \mathscr{M}(\nu, \theta)$. Then $\theta(1-a) <C_d$.
\end{lemma}
\begin{proof}
 The proof calls upon the differentiability of $W$ and $I$. Since $0\notin \mathscr{M}(\theta,\nu)$, we know that $\nu^{-1}I(\nu a)<0$ for all $a\in \mathscr{M}(\theta, \nu)$, since $(W(\theta(1-a)))'\geq 0$ for all $a\in [0,1)$. Further, since $\mathscr{M}(\theta,\nu)\subset (0,1)$, all $a\in \mathscr{M}(\theta,\nu)$ are critical points of $F(a)$, which then implies that  $(W(\theta(1-a)))'+I'(\nu a)=0$. By Lemma \ref{lem:Ider} $\nu a\cdot I'(\nu a)\leq I(\nu a)<0$, and thus $(W(\theta(1-a)))'>0$. Finally, this tells us that $\theta(1-a)<C_{d}$. 
\end{proof}
As the final result of this subsection, we  prove Corollary \ref{cor:doublephase} modulo the proof of Theorem \ref{thm:main_summary}. Namely, we show that there is a range of $\nu$ such that as $\theta$ is varied with $\nu$ fixed, we see a double phase transition. This requires only the results on the transition curve and the structure of $\mathscr{M}(\theta, \nu)$ assembled here. 
\begin{proof}
We point out first that by Theorem \ref{thm:phasecurve}, since the transition curve is strictly decreasing, $\nu_{c}(C_{d})>R_{p}$. Choose $\nu\in (R_{p},\nu_{c}(C_{d}))$. Let $\theta_{c}(\nu)$ denote the inverse function of $\nu_{c}$, which clearly exists and is continuous for the same reason. By our choice of $\nu$, $C_{d}<\theta_{c}(\nu)$. When $\theta<C_{d}$, the parameters are subcritical and item a. of Theorem \ref{thm:main_summary} applies. When $\theta\in (C_{d}, \theta_{c}(\nu))$, the parameters are still subcritical, but the limit is massless and item b. applies. Finally, when $\theta>\theta_{c}(\nu)$, the parameters are supercritical and item c. applies.  
\end{proof}

\section{Proof of \Cref{thm:main_summary}}\label{sec:Mainproof}
\subsection{Local limit of the subcritical  phase $\nu<\nu_c$}\label{sec:subcritical}
Recall our strategy from \Cref{sec:outline}. If $\nu<\nu_c$ we want to first argue that the $\ell^{\infty}$-norm of $\Psi_N^{\nu, \theta}$ is $o(\sqrt{N})$, and then we want to bootstrap and push this bound down to $N^\beta$ for some small $0<\beta<\frac{1}{2}$.
For $\theta \le C_d$, we already know from \Cref{prop:U,lem:sub_theta_nu} that the $\ell^\infty$ norm of $\Psi_N^{\nu, \theta}$ is $o(\sqrt{N})$ with high probability. The additional complication brought about in the regime $\theta>C_{d}$ is the fact that the set of minimizers is now an interval containing $0$ (\Cref{lem:theta>cd}) which prevents the immediate deduction of the effective $\ell^{\infty}$ bound. 

Key to the analysis here is the following inequality proved in \cite{WEI99}. Recall the definition of $R_p$ from \eqref{eq:Rp1}. Let $\psi\in \ztwo$, such that $\norm{\psi}^{2}_{2}=a< R_{p}$, whence $I(a)=0$.  Then
\begin{align}\label{eq:weibound}
\norm{\nabla \psi}^{2}_{2}-\frac{2\nu^{(p-2)/2}}{p}\norm{\psi}_{p}^{p}\geq \left(1-\left(\frac{1-\frac{C_{d}}{\theta}}{\nu^{-1}R_{p}}\right)^{(p-2)/2}\right)\norm{\nabla \psi}_{2}^{2}.
\end{align}
This inequality can be derived by using the fact that $R_p$ can also be characterized as the optimal constant for the discrete version of the GNS inequality. We refer the interested reader to \cite[Section 4]{WEI99} for more details.

{In terms of $\cH_{N}$, we may rescale to obtain the same inequality for all $\bZ^{d}-$embeddable $\psi$ whose mass is bounded above by $N(1-C_{d}/\theta)$, albeit with $\nu$ replaced with $\nu_{N}$. Moreover, we know that for $a<R_{p}/\nu$, $I(\nu a)=0$, which further tells us that $\norm{\nabla \psi}_{2}^{2}\geq 2p^{-1}\nu_{N}\norm{\psi}_{p}^{p}$. Combining,  
\begin{align}\label{eq:weibound2}
\cH_{N}(\psi)\geq \frac{2\nu^{(p-2)/2}}{N^{(p-2)/2}p}\left(1-\left(\frac{1-\frac{C_{d}}{\theta}}{\nu^{-1}R_{p}}\right)^{(p-2)/2}\right)\norm{\psi}_{p}^{p}.
\end{align}
}
\begin{lemma}\label{lem:linfty2}
Let $\eps_N$ be as in \Cref{prop:U}.
Let $\theta>C_{d}$ and $\nu<\nu_{c}(\theta)$. Then 
\[
\bP\left(\norm{\Psi}_{\infty}\geq N^{1/2}\eps_{N}^{\frac1{4p}}\right)\leq e^{-N\sqrt{\eps_{N}}/2}.
\]
\end{lemma}
\begin{proof}
We know from \Cref{lem:theta>cd} that $\mathscr{M}(\theta,\nu)=[0,1-C_{d}/\theta]$ where $1-C_{d}/\theta<R_{p}/\nu$. Also recall the notion of the separating set $U=U(\Psi_{N},\eps_{N})$. By definition of the separating set, the  $\ell^{\infty}$ bound on $U^{c}$ is at most $\sqrt{\eps_N N}$, so we need to show is that $\norm{\Psi|_{U}}_{\infty}<\eps_{N}^{\frac1{4p}}\sqrt{N}$ holds with probability approaching $1$ as $N\to \infty$, appropriately fast. Also recall the set $\cG(\eps_{N})$ introduced in Proposition \ref{prop:U}, which we recall here for convenience
$$
    \cG(\eps_{N}):=\{\psi:\frac{1}{N}\norm{\psi|_{U(\psi,\eps_{N})}}_{2}^{2}\notin \mathscr{M}(\eps_{N}) \}. 
    $$
where $\mathscr{M}(\eps_N)  := \{x \in [0,1), F(x)-F_{\min} < \sqrt{\eps_N} $ and $F$ is the free energy functional defined in \eqref{eq:F}.
It was shown there in \Cref{prop:U} that there exists $\gd_{N}\to 0$ as $N\to \infty$ such that 
\[
\Psi\in \cG(\eps_{N})^{c}\Rightarrow |N^{-1}\norm{\Psi|_{U}}_{2}^{2}-a|<\gd_{N} \text{ for some }a\in \mathscr{M}(\theta,\nu). 
\]
for $N$ sufficiently large. We define
\begin{align*}
\cA&:=\left\{\norm{\psi|_{U}}_{2}^{2}\in \left(N\eps_{N}^{\frac{1}{2p}},(1-\frac{C_{d}}{\theta})N\right)\right\}\cap \{\norm{\psi |_{U}}_{\infty }\geq \sqrt{N}\eps_{N}^{\frac{1}{4p}}\} 
\end{align*}
Note that $\cG(\eps_{N})^{c}\cap \{\norm{\Psi}_{\infty}\geq \sqrt{N}\eps_{N}^{\frac1{4p}}\} \subseteq \cA$ since obviously $\norm{\psi |_{U}}^2_{\infty } \le \norm{\psi|_{U}}_{2}^{2}$. Thus, 
\begin{align*}
\bP(\norm{\Psi_{N}}_{\infty}\geq \sqrt{N}\eps_{N}^{\frac1{4p}}) &\leq \bP(\Psi_{N}\in \cG(\eps_{N})) +\bP(\Psi_{N} \in \cA) \\ &\leq e^{-N\sqrt{\eps_{N}}/2}+\bP(\Psi_{N}\in \cA) ,
\end{align*}
It suffices now to verify the bound on $\bP\left(\Psi_{N}\in \cA\right)$.
We bring in the same decomposition introduced in Proposition \ref{prop:U}, where $\cC(A,\eps_{N})$ denotes the collection of $\psi$ such that $A=U(\psi,\eps_{N})$. We then have the following upper bound
\begin{align*}
\bP(\Psi\in \cA) &\leq \frac{1}{\cZ_{N}}\sum_{A:|A|\leq \eps_{N}^{-d-2}}\int_{\cC(A,\eps_{N})\cap \cG(\eps_{N})^{c}\cap \{\norm{\psi}_{\infty}\geq \sqrt{N}\eps_{N}^{\frac{1}{4p}}\}} e^{-\theta\cH_{N}(\psi)}\vd \psi \\
&\leq \frac{1}{\cZ_{N}} \sum_{A}\sum_{\stackrel{i+j\leq N^{1-\xi}}{ iN^{\xi-1} \in (N\eps_{N}^{\frac{1}{2p}},a^{*}N)}}\int_{\cB_{i,A}\cap\{\norm{\psi}_{\infty}\geq \sqrt{N}\eps_{N}^{\frac{1}{4p}}\}}e^{-\theta \cH_{N}(\psi|_{A})}\vd\psi|_{A}\int_{\tilde{\cB}_{j,A^{c}}} e^{-\theta \cH_{N}(\psi|_{A^{c}})}\vd {\psi|_{A^{c}}}
\end{align*}
where $\cB_{i,A}$ and $\tilde \cB_{j,A^c}$ are as in \Cref{prop:U} and $a^{*}=(1-C_d/\theta)$. Let $\norm{\psi|_{A}}_{\infty}\geq \sqrt{N} \eps_{N}^{3/4p}$, and $\norm{\psi}_{2}^{2}\in (N\eps_{N}^{1/2p},(1-C_{d}/\theta)N)$. Then we know by Lemma \ref{lem:lp_smallest} that the smallest possible $\ell^{p}$ norm of a function contained in this set is achieved by those having one site attain the $\ell^{\infty}$ norm, and have rest of the mass equally distributed among the remainder of the lattice sites. 
Note, $|U|\leq \eps_{N}^{-d-2}$. Thus, if for $\mvx\in A$ such that $|\psi_{\mvx}|=\sqrt{N}\eps^{1/p}$, the function values on the other lattice sites is given by \[(N\eps_{N}^{1/2p}-N\eps_{N}^{6/4p})^{1/2}\cdot \eps_{N}^{(d+2)/2}\geq \frac{1}{2}N\eps_{N}^{\frac{1}{2p}+\frac{d+2}{2}},\]
 where the inequality holds for sufficiently large $N$.  The lower bound for $\norm{\psi}_{p}^{p}$ is thus given by 
\[
N^{p/2}\cdot \eps_{N}^{3/4}+\eps^{-d-2}\cdot N^{p/2}\cdot \eps_{N}^{\frac{1}{4}+\frac{p}{2}\cdot \frac{d+2}{2}}\geq N^{p/2}\eps^{3/4}. 
\]
Thus, 
\[
N^{-(p-2)/2}\cdot \norm{\psi}_{p}^{p}\geq N \cdot \eps_{N}^{3/4}.  
\]
By \eqref{eq:weibound}, this then tells us for all such $\psi$, 
\[
\cH_{N}(\psi)\geq \frac{2\nu^{(p-2)/2}}{N^{(p-2)/2}p}\left(1-\left(\frac{1-\frac{C_{d}}{\theta}}{\nu^{-1}R_{p}}\right)^{(p-2)/2}\right)\norm{\psi}_{p}^{p}\geq C(\theta, \nu,p)N\eps_{N}^{3/4}.
\]
We may use this estimate to obtain the upper bound 
\[
\int_{\cB_{i,A}\cap \{\norm{\psi}_{\infty}\geq \eps_{N}^{\frac{1}{p}}\sqrt{N}\}}e^{-\theta\cH_{N}(\psi|_{A})}\vd \psi|_{A}\leq e^{-\theta C N\eps_{N}^{3/4}}\text{Vol}(\cB_{i,A})\leq e^{-\theta C N\eps_{N}^{3/4}+o(N\eps_{N})}. 
\]
Thus, 
\[
\bP(\Psi\in \cA)\leq \frac{1}{\cZ_{N}} \cdot\left( \frac{\pi}{\theta}\right)^{N} \cdot e^{-C\theta N \eps_{N}^{3/4}}.N^{1-\xi}e^{-NW(\theta)+o(\eps_{N}N)}
\]
The proof is completed by bringing in the partition function lower bound from \cite{DKK}
\[
\cZ_{N}\geq e^{O(N\eps_{N})}\cdot \left(\frac{\pi}{\theta}\right)\cdot e^{-NW(\theta)},
\]
which then implies that 
\[
\bP(\Psi_{N}\in \cA)\leq e^{-O(N\eps^{3/4}_{N})}.
\]
\end{proof}
Now we upgrade the $o(\sqrt{N})$ bound to $O(N^{\beta})$ for $\gb$ as per Corollary \ref{cor:betachoice}. 
\begin{lemma}\label{lem:linfty3}
Let $\nu<\nu_{c}(\theta)$ for any $\theta>0$, and let $\gb$ be as per Corollary \ref{cor:betachoice}. Then 
\[
\bP(\|{\Psi^{\theta,\nu}_{N}}\|_{\infty}\geq N^{\gb})=O(\log N\cdot e^{-N^{\gb}}).
\]
as $N\to \infty$. 
\end{lemma}
\begin{proof}
We will refer to the proof technique here as the \textbf{dyadic argument}, it will make another prominent appearance in the proof of \Cref{thm:supercritical} which addresses the supercritcial case $\nu>\nu_c$. We know that there exists a $\gd_{N}$ such that $$\bP(\|{\Psi^{\theta,\nu}_{N}}\|_{\infty}\geq \sqrt{\gd_{N}N}) \leq e^{-N\sqrt{\eps_{N}}/2}$$ by \Cref{lem:sub_theta_nu,prop:U,lem:linfty2} which address the $\theta\leq C_{d}$ and $\theta> C_{d}$ cases respectively. Now we define the following sets
\[
\cP_{j}:=\{\psi\in \ttwo:2^{-j-1}\sqrt{\gd_{N}N}\leq \norm{\psi}_{\infty}\leq 2^{-j}\sqrt{\gd_{N}N} \}.
\]
Define $j_{\max}:=\lceil(\frac{1}{2}-\gb)\frac{\log {N} }{\log2 } \rceil$.
Note that
\begin{align}\label{eq:largeprobbound}
\bP(\|\Psi^{\theta, \nu}_N\|_\infty\geq N^{\gb})\leq \sum_{j=1}^{j_{\max}} \bP (\Psi_N^{\theta,\nu}\in \cP_{j})+\bP(\|\Psi_N^{\theta,\nu}\|_\infty\geq \sqrt{\gd_{N}N} )\leq \sum_{j=1}^{j_{\max}}\bP (\Psi_N^{\theta,\nu}\in \cP_{j})+e^{-N\sqrt{\eps_{N}}/2}.
\end{align}
Now, for a given $\psi\in \cP_{j}$, since $\norm{\psi}_{2}^{2}\leq N$, by Lemma \ref{lem:lp_largest}, we know 
\begin{align}
\norm{\psi}_{p}^{p}\leq 2^{-pj}\cdot N^{\frac{p}{2}}\gd_{N}^{\frac{p}{2}}.
\end{align}
We may rewrite the expectations in \eqref{eq:largeprobbound} in terms of the spherical law by using \eqref{eq:NLS_sph_relation}.
Namely,
\begin{align}\label{eq:largeprobbound2}
\bP(\Psi\geq N^{\gb})=\frac{\cZ_{\text{sph},N}(\theta)}{\cZ_{N}(\theta, \nu)}\sum_{j}\bE\left( \exp(\theta \nu_N\norm{\Psisph}_p^p) 1_{\Psisph\in \cP_{j}}\right) +e^{-N\sqrt{\eps_{N}}/2}. 
\end{align}
By Lemmas \ref{lem:spherical_tail} and \ref{lem:sphericaltail2}, we have a sub-Gaussian tail bound for the spherical model for the cases $\theta<C_{d}$ and $\theta\geq C_{d}$ respectively. The case of Lemma \ref{lem:spherical_tail} under consideration is the case with no boundary, that is $U=\emptyset$. For the application of \Cref{lem:sphericaltail2}, we only apply it with $b>N^\beta>\log N$ below, where $b$ is as in that lemma. Combining with \eqref{eq:largeprobbound2} and the definition of $\nu_N$ \eqref{eq:nuN},  
\begin{align*}
 \bE\left( \exp(\theta \nu_N\norm{\Psisph}_p^p) 1_{\Psisph\in \cP_{j}}\right) \\ &\leq \exp(\frac{2^{-pj+1}}{p}N\theta\nu^{\frac{p-2}{2}}\gd_{N}^{\frac{p}{2}}-2^{-2j}C\gd_{N}N) \\ &=\exp(2^{-2j}N\gd_{N}(\frac{2^{-(p-2)j+1}}{p}\theta(\nu\gd_{N})^{\frac{p-2}{2}}-C)).
\end{align*}
Regardless of what $\nu$ and $C$ are, since they are fixed constants the expression inside the exponential is negative for sufficiently large $N$ when $j=1$, and thus for all $j\leq j_{\max}$. We may bound the sum by 
\begin{align*}
\sum_{j=1}^{j}\exp(2^{-2j}N\gd_{N}(\frac{2^{-(p-2)j+1}}{p}\theta(\nu\gd_{N})^{\frac{p-2}{2}}-C)) &\leq j_{\max}\exp(N2^{-2j_{\max}}(\theta\nu^{\frac{(p-2)}{2}} \gd_{N}^{p/2}-C\gd_{N})) \\ &\leq j_{\max}\cdot e^{-CN^{2\gb}}.
\end{align*}
The proof is completed by noting that $\cZ_{\text{sph},N}(\theta)/\cZ_{N}(\theta,\nu)\leq 1$. Indeed, this can be easily seen by taking $\ff $ to be identically equal to 1 in \eqref{eq:NLS_sph_relation}.
\end{proof}
The following corollary is almost immediate:
\begin{cor}\label{cor:pfn1}
For any $\theta>0$ and $\nu<\nu_{c}(\theta)$,
\[
\frac{\cZ_{N}(\theta,\nu)}{\cZ_{\text{sph},N}(\theta)}\to 1
\]
as $N\to \infty$. 
\end{cor}
\begin{proof}
We know that   
\[
\frac{\cZ_{N}(\theta,\nu)}{\cZ_{\text{sph},N}(\theta)}=\bE \exp(\theta \nu_{N}\norm{\Psisph}_{p}^{p}).
\]
Clearly $\cZ_{N}(\theta,\nu)/\cZ_{\text{sph},N}(\theta)\geq 1$, all we need to do is verify the upper bound.By Lemma \ref{lem:linfty3}
\begin{align*}
\frac{\cZ_{N}(\theta,\nu)}{\cZ_{\text{sph},N}(\theta)}&=\bE\exp(\theta \nu_{N}\norm{\Psisph}_{p}^{p})1_{\norm{\Psisph}_{\infty}\leq N^{\gb}}+\bE\exp(\theta \nu_{N}\norm{\Psisph}_{p}^{p})1_{\norm{\Psisph}_{\infty}>N^{\gb}} \\ 
&= \bE\left(\exp(\theta \nu_{N}\norm{\Psisph}_{p}^{p})1_{\|\Psisph\|_\infty\leq N^{\gb}}\right) +\frac{\cZ_{\text{sph},N}(\theta)}{\cZ_{N}(\theta,\nu)}\bP(\norm{\Psi_{N}}_{\infty}\geq N^{\gb}).
\end{align*}
By Lemma \ref{lem:lp_largest} and Corollary \ref{cor:betachoice} we have more bounds on the exponential which tell us that 
\begin{align*}
\frac{\cZ_{N}(\theta,\nu)}{\cZ_{\text{sph},N}(\theta)}&=(1+O(N^{(p-2)\gb+(1-\frac{p-2}{2})})) \bE (1_{\|{\Psisph}\|_{\infty}\leq N^{\gb}}) +O(e^{-CN^{2\gb}}).
\end{align*}
By choice of $\gb$, we know that $(p-2)\gb+(1-\frac{p-2}{2})<0$. Finally, the subgaussian tail bound through Lemmas \ref{lem:spherical_tail} and \ref{lem:sphericaltail2} allow us to conclude the proof, since 
\[
\bP(\norm{\Psisph}_{\infty}\leq N^{\gb})=1-O(e^{-N^{2\gb}}).
\]\end{proof}
We are now ready to characterize the local limit of $\Psi_{N}^{\theta,\nu}$ in the subcritical phase. In particular, we now show that the local limit is the same as that of the spherical model.  
\begin{proof}[Proof of Theorem \ref{thm:main_summary}, Cases a and b] 
Let $A\subset \bT^{d}_{n}$ be a fixed finite collection of lattice sites, and let $\ff:\bC^{A}\to \bR$ be a bounded continuous function. Now, note that 
\begin{align*}
\bE (\ff(\Psi^{\theta, \nu}_{N}))&=\bE(\ff(\Psi^{\theta,\nu}_{N})1_{\norm{\Psi_{N}}_{\infty}\leq N^{\gb}})+O(e^{-CN^{2\gb}}) \\ 
&=\frac{\cZ_{N}(\theta,\nu)}{\cZ_{\text{sph},N}(\theta)}\cdot \bE(\exp(\theta \nu_{N}\norm{\psi}_{p}^{p})\cdot \ff(\Psisph)1_{\|{\Psisph}\|_{\infty}\leq N^{\gb}})+O(e^{-CN^{2\gb}})
\end{align*}
By Corollary \ref{cor:betachoice}, we know that 
\[
\bE (\exp(\theta \nu_{N}\norm{\psi}_{p}^{p})\ff(\Psisph)1_{\norm{\Psisph}\leq N^{\gb}})=(1+O(N^{(p-2)\gb+(1-\frac{p-2}{2})}))\bE \ff(\Psisph)1_{\norm{\Psisph}\leq N^{\gb}}.
\]
Next by Lemmas \ref{lem:spherical_tail} and \ref{lem:sphericaltail2}
\begin{align*}
\bE\ff(\Psisph)=\bE \ff(\Psisph)1_{\|{\Psisph}\|_{\infty}<N^{\gb}}+O(e^{-CN^{2\gb}}).
\end{align*}
Combining bounds and applying Corollary \ref{cor:pfn1}, we find that 
\begin{align*}
\bE \ff(\Psi_{N})=(1+O(N^{(p-2)\gb+(1-\frac{p-2}{2})})) \bE \ff(\Psisph)+O(e^{-N^{2\gb}}). 
\end{align*}
Finally by Theorem \ref{thm:main_sph},
\[
\bE \ff(\Psisph)=\begin{cases}(1+o(1))\cdot \bE\ff(\theta^{-\frac{1}{2}}\cdot \Phi^{\bZ^{d}}_{m}) \text{ when }\theta<C_{d}, \\ (1+o(1))\cdot \bE \ff(\theta^{-\frac{1}{2}}\cdot \Phi^{\bZ^{d}}+U\mv1) \text{ when }\theta\geq C_{d},\end{cases}
\]
where $U$ is uniformly distributed on the disk of radius $\sqrt{1-\frac{C_{d}}{\theta}}$ in $\bC$, and  $m$ and $\theta$ are related as per \eqref{eq:mass}.
\end{proof}
\subsection{Local limit of the supercritical  phase $\nu > \nu_c$}\label{sec:supercritical}
In this section, we prove the following theorem, which is a detailed version of \Cref{thm:main_summary}, item c.
\begin{thm}\label{thm:supercritical}
Fix $p>4$ and $\nu>\nu_c$. The sequence of focusing NLS invariant fields $(\Psi^{\nu,\theta}_N)_{N \ge 1}$ admits subsequential limits in the local weak sense. Furthermore, each subsequential limit is a massive Gaussian free field with a random mass $\{\theta^{-\frac{1}{2}}\cdot \Phi^{\bZ^d}_{M(\Xi)}(\mvx)\}_{\mvx\in A}$ where $M=M(\Xi)$ solves 
\begin{equation}
\int_{[0,1]^{d}}\frac{d\mv\gk}{4\sum_{i=1}^{d}\sin^{2}(\pi \gk_{i})+M} = \theta(1-\Xi).\label{eq:mass_a}
\end{equation}
where  $\Xi \in \mathscr{M}(\theta, \nu)$ and $M(\Xi)>0$ almost surely. Furthermore if $\theta<C_d$, then $M(\Xi)<m$ almost surely where $m$ and $\theta$ are related as in \eqref{eq:mass}.
\end{thm}
As mentioned in \Cref{sec:outline}, the idea is to condition the field in the separating set and then prove a local limit theorem for the NLS field with boundary conditions. By \Cref{prop:U},  a positive proportion of the mass is concentrated on the separating set. The issue of course is that the separating set only gives us a boundary condition $o(\sqrt{N})$ on its boundary, which needs to be pushed down to $N^{\beta}$ for small $\beta>0$ in order to invoke \Cref{cor:betachoice}.  This will be done using a dyadic argument similar to the subcritical case, but the boundary conditions give rise to additional complications. We plan to consider a set $U_1 \supset U$, still not too big, which will ensure the $\ell^\infty$ drop off in mass via the dyadic argument. However, we need to be careful as extra mass might be lost if $U_1$ is too big.

Given the above strategy, it is no surprise that we now need the following version of the random NLS field $\Psi^{\nu, \theta, f, U}_N$ with boundary conditions.  Let $U \subset \bT_n^d$ and let $f:U \to \bC$ be a function. For any bounded, continuous function $\ff:\bC^{\bar U^c} \to \bR$ (recall $\bar U^c = U^c \cup \partial U$).

\begin{multline}
\bE(\ff(\Psi^{\nu, \theta,\gamma, f, U}_N))\\=\frac{1}{\ Z_{N}(\nu,\theta, \gamma,f, U)}\int_{\bC^{\bar U^c}} \ff(\psi)\exp\left[\theta \left(\nu_N\norm{\psi}_{p}^{p}-\norm{\nabla \psi}_{2}^{2} \right)\right]\mv1_{\norm{\psi}_{2}^{2}\leq \gamma N}\mv1_{\psi|_{\partial U} \equiv f|_{\partial U}  } \vd \psi.\label{eq:psi_boundary_cond}
\end{multline}
where the gradients are taken in the graph $G_{U^c}$, the subgraph induced by $\bar U^c$.

For some $A \subset \bT_n^d$ and $f:A \to \bC$, recall that $\cU = \cU(A,f, \eps_N)$ denotes the following event:
\begin{itemize}
    \item $U(\Psi_N^{\nu, \theta}, \eps_N)=A$
    \item $\Psi_N^{\nu, \theta}|_A \equiv f$.
\end{itemize}
Let $\alpha_N = \frac1N\|f|_A\|_2^2$. Recall that by the definition of separating set, we must have $\|f|_{\partial A}\|_2^2 \le \sqrt{\eps_NN}/10 $.
The conditional law of $\Psi^{\nu, \theta}_N$ in $A^c$  is given by $\Psi_N^{\nu, \theta, (1-\alpha_N), f,A^c}$ (NLS field with boundary condition $f$ on $A$, defined in \eqref{eq:psi_boundary_cond}) conditioned on the maximum absolute value of $\Psi_N^{\nu, \theta, (1-\alpha_N), f,A^c}$ in $A^c$ being at most $\sqrt{\eps_N N}$.  In other words, the conditional density is given by
\begin{equation}
    \frac1{\cZ_N'}\exp\left[\theta \left(\nu_N\norm{\psi}_{p}^{p}-\norm{\nabla \psi}_{2}^{2} \right)\right]\mv1_{\norm{\psi|_{A^c}}_{2}^{2}\leq (1-\alpha_N) N}\mv1_{\|\psi|_{A^c}\|_{\infty} \le \sqrt{\eps_N N}}\mv1_{\psi|_{A} \equiv f|_{A}  }, \qquad \psi \in \bC^{\bar A^c} \label{eq:cond_density_linfty}
\end{equation}
for an appropriate partition function $\cZ_N'= \cZ_N'(\nu,\theta, 1-\alpha_N,f,U, \eps_N)  $ (see \Cref{lem:conditional_density_separating_set}). Let $\Psi_{N, \eps_N}^{\theta, 1-\alpha_N}$ denote this conditional random field, where we dropped the $\nu, f,A^c$ from the superscript to lighten notation.

Similarly we introduce $\Psispheps^{\theta, 1-\alpha_N}$ to denote the spherical law  on $A^c$ with boundary condition $f$ and mass cutoff $1-\alpha_N$, and conditioned to take maximum absolute value $\sqrt{\eps_NN}$. That is, the density of $\Psispheps^{\theta, 1-\alpha_N}$ is given by
\begin{equation*}
    \frac1{\cZ'_{\text{sph}, N}(\theta, 1-\alpha_N, \eps_N)}\exp\left(-\theta \norm{\nabla \psi}_{2}^{2} \right)\mv1_{\norm{\psi|_{A^c}}_{2}^{2}\leq (1-\alpha_N) N}\mv1_{\|\psi|_{A^c}\|_{\infty} \le \sqrt{\eps_N N}}\mv1_{\psi|_{A} \equiv f|_{A}  }, \qquad \psi \in \bC^{\bar A^c}.
\end{equation*}
where $\cZ'_{\text{sph}, N} =\cZ'_{\text{sph}, N}  (\theta, 1-\alpha_N, \eps_N)$ is the appropriate partition function.
Also, for any bounded continuous function $\ff: \bC^
{\bar A^c} \to \bR$, 
\begin{equation}
    \bE(\ff(\Psi^{\theta, 1-\alpha_N}_{N,\eps_N})))  = \frac{\cZ'_{\text{sph}, N}}{\cZ'_N}\bE\left(\ff(\Psispheps^{\theta, 1-\alpha_N})\exp\left[\theta \left(\nu_N\norm{\Psispheps^{\theta, 1-\alpha_N}}_{p}^{p}\right)\right]\right)\label{eq:NLS_to_spherical}
\end{equation}
Note that 
\begin{equation}
  \frac{\cZ'_{\text{sph}, N}}{\cZ'_N} \le 1.\label{eq:ratio_NLS_sph_partition_upper_bd}  
\end{equation}
Now let us compare the law of $\Psispheps^{\theta, 1-\alpha_N}$ to the same law without the $\ell^{\infty}$ conditioning. The latter is already defined in \eqref{eq:density_spherical_boundary_general_mass} and let  $\Psisph^{\theta, 1-\alpha_N}$ denote the random field with this law. Let $\cZ_{\text{sph}, N}(f,\theta, 1-\alpha_N, A^c) = \cZ_{\text{sph}, N}(\theta, 1-\alpha_N)$ denote its partition function (simplifying the notation). Our next lemma states that these two partition functions are very close to each other if the mass of $f$ is chosen correctly.
 
 \begin{lemma}\label{lem:partition_condition_close}
    Suppose $\nu>\nu_c$ and $f,A, \alpha_N, \eps_N$ be as above and assume $\alpha_N=a+o_N(1)$ for some $a \in \mathscr{M}(\nu, \theta)$. We have
     $$
     \frac{\cZ'_{{\sf \text{sph}}, N}(\theta, 1-\alpha_N, \eps_N)}{\cZ_{\text{sph}, N}(\theta, 1-\alpha_N) }  = 1+o(1).
     $$
 \end{lemma}
 \begin{proof}
     Note that 
     \begin{equation*}
    \frac{\cZ_{\text{sph}, N}(\theta, 1-\alpha_N, \eps_N)}{\cZ_{\text{sph}, N}(\theta, 1-\alpha_N) }  = \bP(\|\Psisph^{\theta, 1-\alpha_N}|_{A^c}\|_\infty\le\sqrt{\eps_N N})  = \bP(\|\Psisph^{\theta(1-\alpha_N)}|_{A^c}\|_\infty \le \sqrt{\eps_N N} )
     \end{equation*}
     where we used the scaling relation \eqref{eq:scaling_spherical} for the second equality. To estimate this probability, we invoke the Gaussian tail bound in \Cref{lem:spherical_tail} and an union bound. Recall that by the definition of the separating set, the boundary values of $f$ on $\partial U$ is at most $\sqrt{\eps_N N}/10$. Hence the massive harmonic extension is also at most $\sqrt{\eps_N N}/10$. Furthermore, note $\theta(1-\alpha_N)\to \theta(1-a)$ for some $a \in \mathscr{M}(\theta, \nu)$, and \Cref{lem:super_to_sub} ensures that $\theta(1-a)<C_d$ (this is where we crucially use the fact that the mass lost in $U$ pushes the regime to the subcritical phase in the rest.) Thus $\theta(1-\alpha_N)<C_d$ for all large enough $N$. Thus \Cref{lem:spherical_tail} is applicable to all vertices in $A$, and hence 
     \begin{equation*}
         \bP(\|\Psisph^{\theta(1-\alpha_N)}|_{A^c}\|_\infty \le \sqrt{\eps_N N} ) = 1+o(1).
     \end{equation*}
     This completes the proof.
 \end{proof}
 The following is the key proposition of this section, which significantly strengthens \cref{prop:U}.  Given $U \subset \bT_n^d$, {let 
\begin{equation}
    U_1=U_1(C):= \{\mvx \in \bT_n^d: d(\mvx,U) \le C\log^2 N\}.\label{eq:tildeU}
\end{equation} 
}
\begin{proposition}\label{prop:tilde_U}
There exists a constant $,C>0$ such that for all $\nu>\nu_c$ and $\beta>0$ the following holds. Let $\eps_N=(\log N)^{-1}, \tilde \eps_N = (\log N)^{-10d}$ and let $\delta_N,\tilde \delta_N$ be as in \Cref{prop:U} for $\eps_N$ and $\tilde \eps_N$ respectively. Let $U(\Psi^{\nu, \theta}_N, \eps_N)=U$.
Then the following holds with probability at least $1-\exp(-c N^{2(\beta\wedge \frac12)})$:
     \begin{enumerate}[a.]
        \item $\frac1N\|\Psi^{\nu, \theta}_N|_{U_1}\|^2_2|\in (a - \tilde \delta_N, a+\tilde \delta_N)$  for some $a \in \mathscr{M}(\theta, \nu)$,
        \item $\|\Psi^{\nu, \theta}_N|_{U_1^c}\|_\infty < N^\beta$. 
        \item $\|\Psi^{\nu, \theta}_N|_{\partial U_1}\|_\infty < N^\beta/2$.
    \end{enumerate}
    where $U_1 = U_1(C)$ is as in \eqref{eq:tildeU} if $\beta<1/2$ and $U_1=U$ is $\beta \ge 1/2$.
\end{proposition}
\begin{proof}
If $\beta \ge 1/2$ we are done by using \Cref{prop:U}. So
assume $\beta<1/2$.
Let $$\cB = \cB(\eps_N, \delta_N):= \{\frac1N\|\Psi^{\nu,\theta}_N |_U \|^2_2 \in (a - \delta_N, a+\delta_N) \text{ for some }a \in \mathscr{M}(\theta, \nu)\}.$$ Note $\eps_N  = (\log N)^{-1}>e^{-\sqrt{\log N}}$ is a valid choice for \Cref{prop:U} to be applicable. Applying \Cref{prop:U} for our choice of $\delta_N$, we conclude
\begin{equation}
  \bP(\cB) \ge 1-\exp(-N\sqrt{\eps_N}/2) . \label{eq:cB_bound}
\end{equation}

Condition on the event $\cU(A,f, \eps_N)$ defined just before \eqref{eq:cond_density_linfty}  and assume $\frac1N\|f\|_2^2 = \alpha_N=a+o(1)$ for some $a \in \mathscr{M}(\nu, \theta)$ (i.e., $\cB$ occurs). Recall the notation $\Psi^{\theta, 1-\alpha_N}_{N, \eps_N}$ which denotes the conditional law of the NLS field on $A^c$ conditioned on $\cU(A,f,\eps_N)$.
Let $m_N$ be related to $\theta(1-\alpha_N)$ as in \eqref{eq:mass}. Since $\alpha_N$ is close to some $a \in \mathscr{M}(\nu, \theta)$, by \Cref{lem:super_to_sub}, $\theta(1-\alpha_N)$ is bounded above and away from $C_d$,  and hence we conclude $m_N\ge c_0>0$ for all  large enough $N$.

The idea now is to apply the dyadic argument as in the the subcritical case, but since the field can be actually large near $U$, we need to inductively do it at each dyadic scale to counter the tilt. We do so by looking at neighbourhoods of logarithmic radius around $U$ and  prove that for each such neighbourhood, the field goes down by a dyadic scale with high probability.

This leads to the following definitions. Let $$A_{j}:=\{\mvx \in \bT^{d}_{n}:d(\mvx,A_{j-1})\leq C\log N\}, \text{ for $j\geq 1$},$$ where $A_{0}=A$.  The constant $C$ will be specified at the appropriate juncture in the argument.  We define the following events:
\[
\cA_{j}:=\{\psi:\norm{\psi|_{\partial A_{j}}}_{\infty}\leq 2^{-j-1}{\sqrt{\eps_{N} N}\text{ and }\|{\psi|_{A_{j}^{c}}}\|_{\infty}\leq 2^{-j}\sqrt{\eps_{N}N}}\}.
\]
Note that $\bP(\Psi_{N,\eps_N}^{\nu, \theta} \in \cA_0)=1$ (see \Cref{def:separating_set}).
Let $j_{\max}$ denote the smallest integer $j$ such that $2^{-j}\sqrt{\eps_{N}N}\leq N^{\gb}$. Note that $A_{j_{\max}} = U_1(\tilde C)$ where $C j_{\max}/\log N = \tilde C$. Note $j_{\max} = O(\log N)$ hence $\tilde C = O(1)$. 

The goal for now is to prove that 
\[
\bP(\Psi^{\theta, 1-\alpha_N}_{N, \eps_N} \in \cA_{j_{\max}})  \ge 1-e^{-CN^{2\gb}}\label{eq:l_infty_bound_tildeU} 
\]
To this end, define
\begin{align}
\cC_{k}=\bigcap_{j=0}^{k}\cA_{j}\text{  and }\cD_{k}:=\cA_{k}^{c}\cap \cC_{k-1}. 
\end{align}
and note that $$\bP(\Psi^{\theta, 1-\alpha_N}_{N, \eps_N} \in \cA_{j_{\max}}) \ge \cP(\Psi^{\theta, 1-\alpha_N}_{N, \eps_N} \in \cC_{j_{\max}}) = 1-\sum_{j=1}^{j_{\max}} \bP(\Psi^{\theta, 1-\alpha_N}_{N, \eps_N} \in \cA_{j_{\max}} \in \cD_j).$$
We now concentrate on upper bounding $\bP(\Psi^{\theta, 1-\alpha_N}_{N, \eps_N} \in \cD_j)$.
Pick function in $f_j \in \cA_j$ which is an extension of $f$ on $A_0$. Condition on the event that $\Psi^{\theta, 1-\alpha_N}_{N, \eps_N} \equiv f_j$ on $A_j$ satisfying $\cC_{j}$ and let $\Psi^j_N$ be the conditional law of $\Psi^{\theta, 1-\alpha_N}_{N, \eps_N}$ on $A_j^c$ (dropping the subscripts and superscripts for clarity). By the same logic as in \Cref{lem:conditional_density_separating_set}, this conditional law is given by an NLS field in $A_j^c$ with boundary values on $\partial A_j$ at most $2^{-j-1}\sqrt{\eps_N N}$ in absolute value, the maximum absolute value on $A_j^c$ being at most $2^{-j}\sqrt{\eps_N N}$ and mass cutoff $1-\alpha_{N,j}$ where $\alpha_{N,j} = \frac1N \|f_j\|_2^2$. We now write
\begin{equation}
    \bP(\Psi^{\theta, 1-\alpha_N}_{N, \eps_N} \in \cD_{j+1}) = \bE(\bP(\Psi_N^j \in \cA^c_{j+1}))\label{eq:cDj}
\end{equation}
where the expectation is over all possible $f_j$ satisfying the constraints  above. Let $\Psi_{\textsf{sph},N}^j$ be the spherical law on $A_j^c$ with boundary condition equalling $f_j$ on $\partial A_j$, $\ell^\infty$-norm bounded by $2^{-j}\sqrt{\eps_N N}$ on $A_j^c$ and mass cutoff $1-\alpha_{N,j}$. Using analogues of \Cref{eq:NLS_to_spherical} and \eqref{eq:ratio_NLS_sph_partition_upper_bd} for our setup, we have 

\begin{equation}
    \bP(\Psi^j_{N}\in \cA_{j+1}^{c})\leq \bE \left(\mv1_{\Psi_{\textsf{sph},N}^{j} \in \cA_{j+1}^{c}}\cdot \exp(\frac2p\theta \nu_{N}\norm{\Psisph^{j}}_{p}^{p})\right)
\end{equation}
Applying Lemma \ref{lem:lp_largest} utilizing the $\ell^\infty$ upper bound on $\Psi_{\textsf{sph},N}^{j}$, we find that 
\[
\bE (\mv1_{\Psi_{\textsf{sph},N}^{j} \in \cA_{j+1}^{c}}\cdot \exp(\theta \nu_{N}\norm{\Psisph^{j}}_{p}^{p}))\leq \exp\left(\frac2pN\theta \nu^{\frac{p-2}{2}}\eps_{N}^{\frac{p-2}{2}}2^{-j(p-2)/2} \right)\cdot \bP(\Psisph^{j}\in \cA_{j+1}^c).
\]
Recall that $\Psisph^{\theta, 1-\alpha_{N,j}}$ denotes be the spherical law on $A_j^c$ with boundary condition $f_j$ on $A_j$ and mass cutoff $1-\alpha_{N,j}$ (i.e. $\Psi_{\textsf{sph},N}^{j}$ without any $\ell^\infty$-conditioning on $A_j^c$). Thus we can write
\begin{equation}
    \bP(\Psisph^{j}\in \cA_{j+1}^c) \le \frac{\bP(\Psisph^{\theta, 1-\alpha_{N,j}} \in \cA_{j+1}^c)}{\bP(\|\Psisph^{\theta, 1-\alpha_{N,j}}\|_\infty < 2^{-j}\sqrt{\eps_N N})}.\label{eq:spherical_conditional_ratio}
\end{equation}
Since $\alpha_{N,j} \ge \alpha_N$, we have $\theta(1-\alpha_{N,j})<C_d$. Hence $m_{N,j} \ge c_0$ where $m_{N,j}$ is the mass parameter related to $\alpha_{N,j}$ through \eqref{eq:mass}. Thus we can now apply \Cref{lem:spherical_tail} to each term in the ratio \eqref{eq:spherical_conditional_ratio}. First note that because of the boundary condition, the harmonic extension (with mass given by $m_N$) is at most $2^{-j-1}\sqrt{\eps_N N}$. Therefore, an application of the spherical tail and union bound yields
$$
\bP(\|\Psisph^{\theta, 1-\alpha_{N,j}}\|_\infty < 2^{-j}\sqrt{\eps_N N}) = 1+o(1).
$$
This is the reason why the boundary condition was chosen to be slightly lower than the $\ell^\infty$-norm condition. Furthermore, 
we can apply \Cref{lem:spherical_tail} for the tail bound of $\Psisph^{\theta, 1-\alpha_{N,j}}$ away from the boundary. It is at this juncture that we make our choice of $C$ that define the $A_{j}$'s. Let $h$ be the massive harmonic extension of $f_{j}$ onto $A_{j}^c$ with mass $m_N$.
Since $m_N\ge c_0>0$, we can choose $C$ large enough such that  for every $\mvy \in A_{j}^c$  which is at distance at least $C\log N$, $|h(\mvy)|<N^{-10}$. In particular, for $\mvy\in A_{j+1}^{c}$, we know that \[
\bP(|\Psisph^{\theta, 1-\alpha_{N,j}}(\mvy)|\geq 2^{-j-1}\sqrt{\eps_{N}N})\leq \exp(-C N\eps_{N} 2^{-2j-2}) .
\]
Overall,  by the application of a union bound over all $\mvy \in A_{j+1}^{c}$, 
$$
\bP(\Psisph^{j}\in \cA_{j+1}^c) \le \exp(-C N\eps_{N} 2^{-2j-2}) .
$$
for potentially a different constant $C$. Thus

\begin{align*}
\bP(\Psisph^{j}\in \cA_{j+1}^c)&\leq N\cdot \exp\left(\frac2pN\theta \nu^{\frac{p-2}{2}}\eps_{N}^{\frac{p-2}{2}}2^{-j(p-2)/2} -C N\eps_{N} 2^{-2j-2}\right) \\ &\leq N \exp \left(-N2^{-2j}\eps_{N}(\frac{C}{2}-2^{-\frac{p-4}{2}}\theta \frac{\nu^{\frac{p-2}{2}}}{p}\eps_{N}^{\frac{p-4}{2}})\right )\\& \leq N \exp(-C'2^{-2j_{\max}}N\eps_{N}) \\
&\leq N \exp(-C'N^{2\gb})
\end{align*}
where the second last inequality holds for $N$ sufficiently large as $\eps_N \to 0$ (this is the reason why we needed the initial estimate of \Cref{prop:U} to bring the $\ell^\infty$-norm down to $o(\sqrt{N})$). Note that $j_{\max}$ is $O(\log N)$, and thus plugging this estimate back into \eqref{eq:cDj}, 
\[
\bP(\Psi^{\theta, 1-\alpha_N}_{N, \eps_N} \in \cA_{j_{\max}})\geq 1-\sum_{j=0}^{j_{\max}}\bP(\Psi^{\theta, 1-\alpha_N}_{N, \eps_N} \in  \cD_{j})\geq 1-O(N\log N)e^{-CN^{\gb}}\geq 1-e^{-CN^{2\gb}},
\]
where again the last inequality in the chain holds for $N$ sufficiently large. Thus we have proved \eqref{eq:l_infty_bound_tildeU}. Recall that  since $\cA_{j_{\max}} =U_1(\tilde C)$, combining \eqref{eq:cB_bound} with \eqref{eq:l_infty_bound_tildeU}, we obtain 
\begin{multline}
\bP(\|\Psi^{\nu, \theta}_{N} |_{U_1^c}\|_\infty < N^\beta)  \le  \bP(\|\Psi^{\nu, \theta}_{N} |_{U_1^c}\|_\infty < N^\beta, \cB) + \exp(-N\sqrt{\eps_N}/2) \le \\\exp(-CN^{2\beta}) + \exp(-N\sqrt{\eps_N}/2)
\end{multline}

We are almost done, except there is a gap between $U(\Psi_N, \eps_N)$ and $U_1$: we only know that the $\ell^{2}$-mass of $\Psi_N^{\nu, \theta}$ on $U(\Psi_N^{\nu, \theta}, \eps_N)$ satisfies part $a.$ while part $b.$ is only satisfied outside $U_1$ with high probability. We get around this simply by making a smaller choice of $\eps_N$ as follows.
First note that by our choice of $\eps_N$,
the volume of $U_1 \setminus U(\Psi_N^{\nu, \theta}, \eps_N)$ is at most 
$$|U(\Psi_N^{\nu, \theta}, \eps_N)|O(\log^{2} N)^{d} =O((\log N)^{4d+1}).$$
where the constant above depends on the choice of $\tilde C$ and $d$.
Now choose $\tilde \eps_N = (\log N)^{-10d}$ and choose $\tilde \delta_N$ as in \Cref{prop:U}. Observe that $U_2:=U(\Psi_N^{\nu, \theta}, \tilde \eps_N) \supseteq U(\Psi_N^{\nu, \theta}, \eps_N)$. On the event 
\begin{equation}
    \cB(\tilde \eps_N, \tilde \delta_N) \cap \{\|\Psi_N^{\nu, \theta}|_{U_1^c}\|_\infty  \le N^\beta\}\label{event}
\end{equation}
we see that the $\|\Psi_N^{\nu, \theta}|_{U_1}\|_2^2$ satisfies item b. with 
$$\delta_N= \tilde \delta_N +\frac{O((\log N)^{4d+1})}{(\log N)^{10d}} .$$
Indeed, on this event, the mass of $\Psi_N^{\nu, \theta}$ inside $U_2$ is in $((a-\tilde \delta_N)N, (a+\tilde \delta_N)N)$ for some $a \in \mathscr{M}(\theta, \nu)$ and the mass in $U_1 \setminus U_2 $ is at most $$\tilde \eps_N N|U_1 \setminus U_2| \le \tilde \eps_N N|U_1 \setminus U| \le  \frac{O(\log N)^{4d+1}}{(\log N)^{10d}}N .$$
This completes the proof since the probability of the event $\cB(\tilde \eps_N, \tilde \delta_N)$ is at least $1-\exp(-c_2\sqrt{\tilde \eps_N }N/2)$ which is bigger than $1-\exp(-c N^{2\beta})$ since $\tilde \eps_N$ decays polylogarithmically and $\beta<1/2$.
\end{proof}

We now finish with the proof of \Cref{thm:supercritical}, which of course supersedes \Cref{thm:main_summary} part c.
\begin{proof}[Proof of \Cref{thm:supercritical}]
Choose $1/100>\beta>0$ as in \Cref{cor:betachoice}.
Let $\eps_N, \tilde \eps_N,\delta,\tilde \delta_N,c,C,U_1$ be as in \Cref{prop:tilde_U}. Let $\cB' = \cB'(C,\eps_N, \tilde \delta_N, \beta)$ be the event in \Cref{prop:tilde_U}, and hence 
$$
\bP(\cB' ) > 1-\exp(-cN^{\beta}).
$$
 Let $$\cB'(f,A):=\cB' \cap \{U_1(C)=A,\Psi_N^{\nu, \theta}|_{A} \equiv f\},$$
 and we only consider those $f,A$ which makes the above event non-trivial, that is we restrict to the set $$\cJ:=\{(f,A): \cB'(f,A) \neq \emptyset \}.$$
Fix such $(f,A) \in \cJ$ and let $\alpha_N = \frac1N \|f\|_2^2$ and note that since $\cB'$ occurs, $\alpha_N = a+o(1)$ for some $a \in \mathscr{M}(\nu, \theta)$. As before, the conditional law of $\Psi_N^{\nu, \theta}$ given $\cB'(f,A)$ is  $\Psi_{N,N^{\beta-1}}^{ \theta, (1-\alpha_N)} $. Let $\Psi_{\mathsf{sph},N,N^{\beta-1}}^{\theta, (1-\alpha_N)}$  be the spherical law on $A^c$ with boundary condition $f$ on $A$, mass cutoff $1-\alpha_N$ and $\ell^{\infty}$-norm bounded by $N^\beta$.

Now let $\ff$ be a function on a bounded set of vertices containing $\mv0$. Assume further that $d(A, \mvo) >N^{0.1}$ as required in \Cref{lem:llimit1}.  Then first note,
\begin{align}
    \bE(\ff(\Psi_{N,N^{\beta-1}}^{ \theta, (1-\alpha_N)}) & = \frac{\bE\left(\Psi_{\mathsf{sph},N,N^{\beta-1}}^{\theta, (1-\alpha_N)}\exp(\theta \nu_N\|\Psi_{\mathsf{sph},N,N^{\beta-1}}^{\theta, (1-\alpha_N)}\|_p^p)\right)}{\bE\left(\exp(\theta \nu_N\|\Psi_{\mathsf{sph},N,N^{\beta-1}}^{\theta, (1-\alpha_N)}\|_p^p)\right)}\nonumber\\
    & = \bE\left(\ff(\Psi_{\mathsf{sph},N,N^{\beta-1}}^{\theta (1-\alpha_N)}\right)(1+o(1))\nonumber\\
    & = \frac{\bE\left(\ff(\Psi_{\mathsf{sph},N}^{\theta (1-\alpha_N)})1_{\|\Psi_{\mathsf{sph},N}^{\theta (1-\alpha_N)}|_{A^c}\|_\infty<N^{\beta}}\right)}{\bP(\|\Psi_{\mathsf{sph},N}^{\theta (1-\alpha_N)}|_{A^c}\|_\infty<N^{\beta})}(1+o(1))\nonumber\\
    &=\bE(\ff(\Psi_{\mathsf{sph},N}^{\theta (1-\alpha_N)}))+o(1)\label{eq:reduce_to_phi}
\end{align}
Here, the second equality follows from \Cref{cor:betachoice} (which kills the tilt because of the choice of $\beta$) and the scaling relation for spherical law \eqref{eq:scaling_spherical}. For the third equality, note that the absolute boundary values of the field on $\partial A$ is at most $N^{\beta}/2$ (by item c. of \Cref{prop:tilde_U}), and $\theta(1-\alpha_N)$ is upper bounded away from $C_d$ by \Cref{lem:super_to_sub}. Letting $m_N$ be the mass related to $\theta(1-\alpha_N)$ via \eqref{eq:mass}, we conclude $m_N $ is lower bounded away from 0 for all large enough $N$. Therefore the harmonic extension of the boundary condition with mass $m_N$ has $o(N)$ $\ell^{2}$-norm by \Cref{lem:l_p_h}. Consequently, we can readily apply 
apply \Cref{lem:spherical_tail} and an union bound to conclude 
\begin{equation*}
    \bP(\|\Psi_{\mathsf{sph},N}^{\theta (1-\alpha_N)}|_{A^c}\|_\infty > N^\beta) <\exp(-cN^{2\beta}).
\end{equation*}
This, along with the fact that $\ff$ is bounded justifies the fourth equality. 

Now we average over $\cB'(f,A)$ First of all let $\cB''$ be the union over $\cB'(f,A)$ such that $d(\mv0, A) > N^{0.1}$. By translation invariance of the set $U_\tau$, $\bP(\cB'') = \bP(\cB')+o(1)$ as $U_1$ is of polylog order (deterministically). Therefore, we can write
\begin{align*}
    \bE(\ff(\Psi_N^{\nu, \theta})) &= \bE(\ff(\Psi_N^{\nu, \theta})1_{\cB''}) + o(1)\\
    &=\bE(\bE(\ff(\Psi_N^{\nu, \theta})| (U_\tau, \Psi_N^{\nu, \theta}|_{U_1}))1_{\cB''})+o(1)
    \end{align*}
    Note $(\Xi_N)_{N\ge 1}:=(\frac1{N}\|\Psi_N^{\nu, \theta}|_{U_1}\|_2^2)_{N \ge 1} $ is a tight sequence of random variables, and has a subsequential limit. Take one such subsequence $(N_\ell)_{\ell} $ and call the limit along the sequence $\Xi$. Note that since $\bP(\cB'')=1+o(1)$, $\Xi \in \mathscr{M}(\nu, \theta)$ almost surely by the first item in \Cref{prop:tilde_U}.
Finally, on $\cB''$, our choice of $A$ is at least $N^{0.1}$ away from $\mv0$ as required by \Cref{lem:llimit1}. Hence, by \eqref{eq:reduce_to_phi} and \Cref{lem:llimit1},
\begin{equation*}
    \bE(\bE(\ff(\Psi_N^{\nu, \theta})| (U_1, \Psi_N^{\nu, \theta}|_{U_1}))1_{\cB''}) = \bE(\ff(\theta^{-1/2}\Phi_{M(\Xi_N)}^{\bZ^d}))+o(1)
\end{equation*}
where $M(\Xi_N)$ solves \eqref{eq:mass_a} with $\theta(1-\Xi_N)$ on its right hand side.
Since as $a_n \to a$, $\Phi^{\bZ^d}_{a_n} \to \Phi_a^{\bZ^d}$ in law,
by dominated convergence theroem, we conclude that along the subsequence $(N_\ell)$, $$\Psi_{N_\ell}^{\nu, \theta}\xrightarrow[\ell \to \infty]{(d)}\theta^{-1/2}\Phi_{\Xi}^{\bZ^d} $$
where the above convergence is in law in the local weak sense.  Further properties of $\Xi$ as asserted in the theorem are clear from \Cref{lem:super_to_sub} and \eqref{eq:mass}.
\end{proof}

\section{Proof of \Cref{thm:main_sph} and results in  \Cref{sec:spherical_details}.}\label{sec:spherical_proofs}
In the following two subsections, we prove \Cref{thm:main_sph}, so it is useful to recall the definition of the spherical law  $\Psisph^\theta$ whose density is given by \eqref{eq:spherical}. Also recall $C_d$ is the Green's function in $\bZ^d$.

\subsection{The non massive case $\theta \ge C_d$.}\label{sec:masslessll}
In this section, we prove the second item of \Cref{thm:main_sph}. The proof proceeds through Fourier analysis. 
Let $(\phi_{\mvk})_{\mvk \in \bT_n^d}$ and $(\lambda_{\mvk})_{\mvk \in \bT_n^d}$ be the eigenvectors and eigenvalues of $-\Delta$ as in \eqref{eq:evectors} and \eqref{eq:evalues}. Recall that $\lambda_{\mv0} = 0$. Now we do the following change of variables (Fourier transform):
$$
Y_{\mvk} = \la \Psisph^\theta, \phi_{\mvk} \ra \text{ for }\mvk \in \bT_n^d.
$$
Since this is a linear transformation and $(\phi_{\mvk})_{\mvk \in \bT_n^d}$ are orthonormal, the Jacobian of this transformation is 1. Let $\fg$ be a complex valued function. Let $\hat \fg_{\mvk} := \la \fg, \phi_{\mvk}\ra$ for each $\mvk \in \bT_n^d$ be its Fourier transforms. Therefore
$$
\la \fg,\Psisph^\theta \ra = \la \hat \fg, Y\ra
$$
Now we wish to use the method of moments and prove that the moments of $\la \fg,\Psisph^\theta \ra$ converges to the moments of the desired Gaussian free field. However, we need to remember that our functions are complex valued, so we look at a more involved (but still elementary) expression involving the joint moments of the real and the imaginary part.
$$
(\fR(\la \hat{\mathfrak{g}},Y\ra))^ q (\fI(\la \hat{\mathfrak{g}},Y\ra))^{q'}  = \left(\frac{\la \hat \fg,Y \ra + \overline{\la \hat \fg,Y \ra}}{2}\right)^q \left(\frac{\la \hat \fg,Y \ra - \overline{\la \hat \fg,Y \ra}}{2i}\right)^{q'}
$$
for some $q,q' \in \bN$. The above product can be written as a sum over terms of the form
\begin{equation}
    \alpha( Y,\iota):= \frac1{2^q (2i)^{q'}}a_{i_1}a_{i_2}\dots a_{i_q}b_{j_1}\dots b_{j_{q'}}\label{eq:alphaYi}
\end{equation}

where each $a_{\ell}$ is of the form  $\hat \fg_\ell Y_{\ell}$ or $\overline{\hat \fg_\ell Y_{\ell}}$  and each $b_{\ell}$ is of the form $\hat \fg_\ell Y_{\ell}$ or $-\overline{\hat \fg_\ell Y_{\ell}}$. Each $\alpha$ comes with a collection of indices $(i_1,\dots,i_q)$ and $(j_1,\dots, j_q)$, all in $\bT_n^d$,  which we denote by $\iota$. Let $\Xi$ be the set of all possible such indices, and we emphasize that repetitions are allowed.


Overall, we obtain 
\begin{align}
    \bE( (\fR(\la \hat{\mathfrak{g}},Y\ra))^ q (\fI(\la \hat{\mathfrak{g}},Y\ra))^{q'})
    &=\frac{\sum_{\iota \in \Xi}\int_{\bC^N} \alpha(y, \iota) \exp(-\theta \sum_{\mvk \in \bT_n^d \setminus \{\mv0\}} \lambda_{\mvk }|y_{\mvk}|^2 )\mv1_{|y_{\mv0}|^2+|y'|_2^2 \le N}  d\mvy }{\int_{\bC^N}  \exp(-\theta \sum_{\mvk \in \bT_n^d \setminus \{\mv0\}} \lambda_{\mvk }|y_{\mvk} |^2) \mv1_{|y_{\mv0}|^2+|y'|_2^2 \le N}d\mvy} \label{eq:ratio_super}
    \end{align}
where $y'$ is the vector $y$ but excluding the coordinate $y_0$. Now  we concentrate on one fixed $\alpha$ in the sum. First observe that if there is a term of the form $\hat \fg_\ell y_\ell$ which is not 
paired with its conjugate, $\ga$ integrates to 0 because of the integral over the angular part of $y_\ell$. Therefore, we only need to consider the indices in $\Xi' \subset \Xi$ where each term in the product comes with its conjugate pair (therefore, $q+q'$ must necessarily be even for the integral to be non-zero).  Thus we only need to consider product of terms of the form $|\hat \fg_\ell y_\ell|^{2}$ where there are $\fs:=(q+q')/2$ many terms in the product (with repititions allowed).

Now we write \eqref{eq:ratio_super} in terms of expectations of random variables. Let $Z':=(Z_{\mvk} )_{\mvk \in \bT_n^d \setminus \{\mv0\}}$ be independent complex Gaussians, each with mean 0 and $Z_{\mvk}$ having variance $\frac1{\theta \lambda_{\mvk}}$ and let $Z_{\mv0} $ be distributed uniformly on the unit disc, independent of everything else. Note that the $y_{\mv0}$ term is missing in the exponent in \eqref{eq:ratio_super} (since $\lambda_{\mv0} = 0$), from which it is easy to see that \eqref{eq:ratio_super} can be written as 
\begin{equation}
    \frac{\sum_{\iota \in \Xi'} \bE(\alpha((\sqrt{N}Z_{\mv0}, Z'), \iota)  1_{N|Z_0|^2+\|Z'\|_2^2 \le N} ) }{\bE(1_{N|Z_0|^2+\|Z'\|_2^2 \le N})}\label{eq:alpha_indices}
\end{equation}
We take a moment to point out that with $Z'$ so defined,  
\begin{align}
\sum_{\mvk\neq 0}Z_{\mvk}\phi_{\mvk}(\mvx)\equald \theta^{-\frac{1}{2}}\Phizero(\mvx). 
\end{align}
where recall that $\Phizero$ is the zero average GFF defined in \Cref{sec:zero_avg_gff}. The following immediate consequence will be very useful later. 
\begin{align}\label{eq:ZGFFrel}
\sum_{(\mv0\notin \iota )\in \Xi'} \E \ga(Z',\iota)=\E (\fR(\la \hat{\mathfrak{g}},\theta^{-\frac{1}{2}}\Phizero\ra))^ q (\fI(\la \hat{\mathfrak{g}},\theta^{-\frac{1}{2}}\Phizero\ra))^{q'}
\end{align}
We now need a few estimates on the Fourier transforms. First of all it is easy to see from the form the eigenvectors \eqref{eq:evectors} and the fact that $A$ is a bounded set that
\begin{equation}
   |\hat \fg_\ell| \leq \frac {C}{\sqrt{N}}\label{eq:hatg_bound}
\end{equation}

for a constant $C$ depending only on $\fg $ and the size of $A$. Let us consider a typical $\alpha((\sqrt{N}Z_{\mv0}, Z'), \Xi')$. It is of the form 
$$
|\hat g_{\mv0}\sqrt{N}Z_{\mv0}|^{2m_{\mv0}}\prod_{\mvk \in \bT_n^d\setminus \{\mv0\}}|\hat g_{\mvk}Z_{\mvk}|^{2m_{\mvk}}
$$
where $\sum_{\mvk}2m_{\mvk}  = q+q'$. Note that we separated out the $Z_{\mv0}$ as it is special: it has uniform distribution scaled by $\sqrt{N}$ while the rest are complex Gaussians. Observe from \eqref{eq:hatg_bound} that 
\begin{equation}
    \bE(|\hat g_{\mv0}\sqrt{N}Z_{\mv0}|^{2m_{\mv0}}) \le C^{2m_{\mv0}}\bE(|Z_{\mv0}|^{2m_{\mv0}}) = C_1(m_{\mv0})\label{eq:zero_moment_bound}
\end{equation}
Also for any $\mvk \neq \mv0$, a standard calculation if moments of complex Normals and \eqref{eq:hatg_bound} shows that
\begin{equation}
    \bE(|\hat g_{\mvk}Z_{\mvk}|^{2m_{\mvk}}) \le \frac{C}{(\theta \lambda_{\mvk})^{m_{\mvk}}N^{m_{\mvk}}}. \label{eq:generic_moment_bound}
\end{equation}
Since $(\lambda_{\mvk})^{-1}$ diverges near the corners of the torus, we need to separate them out, and bank on the fact that the number of terms in the corners is small. This motivates the following definitions. 
Now fix a small $\eta>0$. Define an $\eta$ corner to be the set of all $\mvk = (k_1,\dots, k_d) \in \bT_n^d$ such that  $k_i \in (\eta n) \cup ((1-\eta)n)$ for all $1\le i \le d$. Let $\Xi_{\text{corn}}$ be the set of elements in $\Xi'$ with at least one non-zero index in the $\eta$-corner. 
\begin{lemma}\label{lem:corner_estimate}
There exists a constant $C$ depending only on $q,q'$ and the dimension $d$ such that  
$$
\sum_{\iota \in \Xi_{\text{corn}}} \bE(\alpha((\sqrt{N}Z_{\mv0}, Z'), \iota) ) \le \begin{cases} C\eta \text{ when }d\geq 4 \\ C\eta|\log \eta| \text{ when }d=3.\end{cases}
$$ 

\end{lemma}
\begin{proof}
    From \eqref{eq:generic_moment_bound} and \eqref{eq:zero_moment_bound}, we see that for every $\iota \in \Xi'$,
    \begin{equation}
        \bE(\alpha((\sqrt{N}Z_{\mv0}, Z'), \iota) \le C^{\fs}\prod_{\mvk \in \iota, \mvk \neq 0}\frac{1}{\theta\lambda_kN}
    \end{equation}
   where recall $\fs = (q+q')/2$. Thus we can write
   \begin{align}
       \sum_{\iota \in \Xi_{\text{corn}}}\bE(\alpha((\sqrt{N}Z_{\mv0}, Z'), \iota) &\le C^{\fs}\sum_{\iota \in \Xi_{\text{corn}}} 
       \prod_{\mvk \in \iota, \mvk \neq 0}\frac{1}{\theta\lambda_kN}\\ \nonumber
       &=C^{\fs} \sum_{\iota \in \Xi'}
       \prod_{\mvk \in \iota, \mvk \neq 0}\frac{1}{\theta\lambda_kN} - C^{\fs} \sum_{\iota \in \Xi'\setminus \Xi_{\text{corn}}}
       \prod_{\mvk \in \iota, \mvk \neq 0}\frac{1}{\theta\lambda_kN}\\ \nonumber
       & = C^{\fs} (1+\sum_{\mvk \in \bT_n^d \setminus \{\mv0\}}\frac1{\theta\lambda_{\mvk}N})^{\fs} - C^{\fs}(1+\sum_{\stackrel{\mvk \in \bT_n^d \setminus \{\mv0\}}{ \mvk  \not \in \eta\text{-corner}}}\frac1{\theta\lambda_{\mvk}N})^{\fs}\\ \nonumber      
   \end{align}
   By the Riemann sum formula and Lemma \ref{lem:gfconv}, the last term above is
   \begin{equation}
   C^{\fs}\left[(1+\int_{[0,1]^d}\frac1{4\sum_{i=1}^{d}\sin^2{\pi x_{i}}}\vd\mvx)^{\fs} - (1+\int_{[\eta,1-\eta]^d}\frac1{4\sum_{i=1}^{d}\sin^2{\pi x_{i}}}\vd\mvx)^{\fs}+O(N^{\frac{2-d}{d}}) \right] 
   \end{equation}
  The smallest $\eta$ can ever be is $N^{-\frac{1}{d}}$, thus the $O(N^{\frac{2-d}{d}})$ term will never corresponding to the leading order asymptotics. The order of this difference is the same as the difference in the integrals
  \begin{equation*}
      \int_{[0,1]^d}\frac1{4\sum_{i=1}^{d}\sin^2{\pi x_{i}}}\vd\mvx - \int_{[\eta,1-\eta]^d}\frac1{4\sum_{i=1}^{d}\sin^2{\pi x_{i}}}\vd\mvx.
  \end{equation*}
  For $d \ge 4$, the asymptotics now simply follows by observing that 
  $$
  \int_{0}^\eta \int_{[0,1]^{d-1}} \frac1{4\sum_{i=1}^{d}\sin^2{\pi x_{i}}}\vd\mvx \le \int_0^\eta \int_{[0,1]^{d-1}} \frac1{\sum_{i=2}^{d}\sin^2{\pi x_{i}}}\vd\mvx = O(\eta) 
  $$
  as the integrand over $(x_2,\ldots, x_d)$ is integrable for $d \ge 4$. For $d=3$, we can calculate this by hand:
  \begin{align*}
  \int_{0}^{\eta}\int_{[0,1]^{2}}\frac{1}{4\sum_{i=1}^{3}\sin^{2}(\pi x_{i})}dx_{2}dx_{3}dx_{1}&=4\int_{0}^{\eta}\int_{[0,\frac{1}{2}]^{2}}\frac{1}{4\sum_{i=1}^{3}\sin^{2}(\pi x_{i})}dx_{2}dx_{3}dx_{1}\\ &\leq \frac{\pi^{2}}{4}\int_{0}^{\eta}\int_{[0,\frac{1}{2}]^{2}}\frac{1}{x_{1}^{2}+x_{2}^{2}+x_{3}^{2}}dx_{2}dx_{3}dx_{1}.
  \end{align*}
  The last inequality follows from the fact that $\sin(x)\geq 2x/\pi$ for all $x\in [0,\frac{\pi}{2}]$ by convexity. Swapping the order of integration, and integrating over $x_{1}$ first, we obtain 
  \[
  \frac{\pi^{2}}{4}\int_{[0,\frac{1}{2}]^{2}}\frac{1}{\sqrt{x_{2}^{2}+x_{3}^{2}}}\arctan(\frac{\eta}{\sqrt{x_{2}^{2}+x_{3}^{2}}})dx_{2}dx_{3}\leq \frac{\pi^{3}}{2}\int_{0}^{\sqrt{2}}\arctan(\frac{\eta}{r})dr.
  \]
  Integrating by parts, we find that 
  \begin{align*}
  \int_{0}^{\sqrt{2}}\arctan(\frac{\eta}{r})dr&=\sqrt{2}\arctan(\frac{\eta}{\sqrt{2}})+\int_{0}^{\sqrt{2}}\frac{r}{r^{2}+\eta^{2}}dr
  \\ &=\sqrt{2}\arctan(\frac{\eta}{2})+\frac{\eta}{2}(\log(2+\eta^{2}) -\log(\eta^{2}))
  \\ &=O(\eta |\log \eta|),
  \end{align*}
  which completes the proof.
\end{proof}

Take $\beta \in \Xi' \setminus \Xi_{\text{corn}}$. Let  $Z^\beta$ be the 
vector $(Z_{\mvk})_{\mvk \in  \beta \setminus \{\mv0\} }$ and
let $Z^{\sbeta}$ denote the vector $(Z_{\mvk})_{\mvk \in \bT_n^d \setminus ( \beta \cup \{0\})  }$. 
In this notation, the constraint is of the form
$$
\|Z^{\sbeta}\|_2^2 + \|Z^\beta\|_2^2+N|Z_{\mv0}|^2 \le N.
$$
Now we claim
\begin{claim}\label{claim:Z_conv_prob}
For every sequence of $q+q'$ index choices $\beta = \beta(N) \in \Xi'$,
   $\frac1N\|Z^{\sbeta}\|_2^2 \to C_d/\theta$ and $\frac1N\|Z^\beta\|_2^2 \to 0$ in probability.
\end{claim}
\begin{proof}
The convergence in probability of $\frac{1}{N}\norm{Z'}_{2}^{2}\to C_{d}/\theta$ can be verified by showing that the variance of $N^{-1}\norm{Z'}^{2}_{2}$ converges to $0$ as $N\to \infty$. We already know the limit by Lemma \ref{lem:gfconv}. Note, 
\begin{align*}
\var (N^{-1}\norm{Z'}_{2}^{2})&=\frac{1}{N^{2}}\sum_{\mvk\neq 0}\var(|Z_{\mvk}|^{2})=\frac{1}{N}\cdot \frac{1}{N}\sum_{\mvk\neq 0}\frac{1}{\gl_{\mvk}^{2}}\\ 
&=\frac{1}{N}\biggl( O(N^{1/3})1_{d=3}+O(\log N)1_{d=4}+O(1)1_{d\geq 5}\biggr).
\end{align*}
The last equality follows from known asymptotics concerning traces of negative powers of the Laplacian, and can be verified by examining the corresponding Riemann sum. This is addressed in more detail in the proof of Lemma \ref{lem:critplb}, and we also refer the reader to Lemma 2.1 of \cite{PDae}. Since $\norm{Z'}_{2}^{2}=\norm{Z^{\sbeta}}_{2}^{2}+\norm{Z^{\gb}}_{2}^{2}$, it suffices to now prove that for any $\gb$, $N^{-1}{\norm{Z^{\gb}}_{2}^{2}}$ converges to $0$ in probability as $N\to \infty$. Note that $\norm{Z^{\gb}}_{2}^{2}$ is the sum at most $q+q'$ many exponential random variables with mean bounded above by $N^{1/d}$. In particular,
\[
\E N^{-1}\norm{Z^{\gb}}_{2}^{2}\leq N^{-1}\cdot O(N^{1/d})=O(N^{\frac{-(d-1)}{d}}). 
\]
Since $\norm{Z^{\gb}}_{2}^{2}$ is positive, this completes the proof. 
\end{proof}
The above convergence can be made almost sure by using less crude bounds. However, we do not require this. 
\begin{proof}[Proof of Theorem \ref{thm:main_sph}, item b., $\theta>C_d$ case] 
First assume $C_d<\theta$. Then we have using \Cref{claim:Z_conv_prob}, for any $\beta \in \Xi' \setminus \Xi_{\text{corn}}$
\begin{equation}
    \bP(\frac1N\|Z^{\sbeta}\|_2^2 + \frac1N\|Z^\beta\|_2^2+|Z_{\mv0}|^2\le 1| Z_{\mv0}, Z^{\beta}) \to 1_{|Z_{\mv0}|^2 <1-\frac{C_d}{\theta}}\text{ a.s.}\label{eq:a.s.conditional}
\end{equation}
To see this, observe that by \Cref{claim:Z_conv_prob}, for any $\delta>0$, the above sequence of conditional probabilities converge to $1$ if $|Z_0|^2<1-C_d/\theta-\delta$ and to $0$ if $|Z_0|^2>1-C_d/\theta+\delta$. Since $\delta>0$ is arbitrary, the above convergence follows.

Using \eqref{eq:a.s.conditional}, we immediately conclude that the denominator in \eqref{eq:alpha_indices} is $$\bP(|Z_0|^2 <1-\frac{C_d}{\theta})+o(1) = (1-\frac{C_d}{\theta})+o(1),$$ 
since $|Z_{\mv0}|^2 \sim $Unif$(0,1)$. Furthermore, employing the worst case of \Cref{lem:corner_estimate} (twice), the numerator in \eqref{eq:alpha_indices} is 
\begin{align*}
    \sum_{\iota \in \Xi'\setminus \Xi_{\text{corn}}} \bE(\alpha((\sqrt{N}Z_{\mv0}, Z'), \iota)  1_{|Z_0|^2<1-C_d/\theta} ) &+O(\eta|\log \eta|) \\&=  \sum_{\iota \in \Xi'} \bE(\alpha((\sqrt{N}Z_{\mv0}, Z'), \iota)  1_{|Z_0|^2<1-C_d/\theta} ) +O(\eta|\log \eta|). 
\end{align*}
We have absorbed the $o(1)$ term coming from \Cref{eq:a.s.conditional} into the $O(\eta|\log \eta|)$ term. Overall, the ratio \eqref{eq:alpha_indices} is given by 
\begin{align*}
\frac{\sum_{\iota \in \Xi'} \bE(\alpha((\sqrt{N}Z_{\mv0}, Z'), \iota)  1_{|Z_0|^2<1-C_d/\theta} ) +O(\eta|\log \eta|)}{(1-\frac{C_d}{\theta})+o(1)}\\ = \bE((\fR(\la {\mathfrak{g}},\theta^{-\frac{1}{2}}\cdot\Phizero+U\mv1\ra))^ q (\fI(\la {\mathfrak{g}},& \theta^{-\frac{1}{2}}\cdot\Phizero+U\mv1\ra))^{q'}) +O(\eta|\log \eta|),
\end{align*}
where $U$ is uniform over the disk of radius $\sqrt{1-\frac{C_{d}}{\theta}}$ in the complex plane and independent of $\Phizero$. Note, we have used \eqref{eq:ZGFFrel} here.  
Since $\eta$ is arbitrary the proof is complete by using Lemma \ref{lem:gfconv}, the local convergence of $\Phizero$ to $\Phi^{\bZ^{d}}$.
\end{proof}

We now turn to the $\theta = C_d$ case, but this needs some preparation. First, we separate the indices, let $\beta$ denote the collection of indices that are present in the monomial $\ga(Z',\iota)$. The ratio \eqref{eq:alpha_indices} takes the following form:

\begin{equation}\label{eq:thetacrit0}
\frac{\sum_{\iota \in \Xi'} \bE(\alpha((\sqrt{N}Z_{\mv0}, Z'), \iota)1_{\|Z^{\sbeta}\|_2^2 + \|Z^\beta\|_2^2+N|Z_{\mv0}|^2\le N}}{\bP(\|Z'\|_2^2+N|Z_{\mv0}|^2\le N)}
\end{equation}
The problem now is that the denominator tends to zero, so we need to show that this divergence is canceled out by the numerator. A fundamental step here is to  actually control the rate of convergence of the denominator to $0$, which involves two ingredients. The first is an estimate of the rate of convergence of the diagonal of the zero average Green's function to its $\bZ^{d}$ counterpart, which was established in Lemma \ref{lem:gfconv}. For the second ingredient, we need separate the cases of $d\geq 4$ and $d=3$ as they require different tools. For $d\geq 4$, we will essentially prove a  Central Limit Theorem for $\norm{Z'}_{2}^{2}$, by establishing the Lyapunov criterion. 
\begin{lemma}\label{lem:lyapunov}
For $d\ge 4$,
$$
\frac{\|Z'\|_2^2 - \bE(\|Z'\|^2)}{\sqrt{\var(\|Z'\|_2^2)}} \xrightarrow[N \to \infty]{(d)} N(0,1)
$$   
\end{lemma}
\begin{proof}
    We check the Lyapunov condition. Note that $(Z_{\mvk} )_{\mvk \neq \mv0}$ is equal in distribution to $\frac{X_{\mvk
    }}{\theta\lambda_{\mvk}}$ where $(X_{\mvk })_{\mvk \neq \mv0} $ are i.i.d.\ Exponential$(1)$. Thus to estimate the moments, we are left with computing the asymptotics of sums of the type $\sum_{\mvk \neq \mv0}\lambda_{\mvk}^{-q}$ for integers $q$.  One can check Lemma 2.1 of \cite{PDae}
    \begin{equation}
        \sum_{\mvk \neq \mv0}\frac1{\lambda_{\mvk}^q} \asymp
        \begin{cases}
            N^{\frac{2q}{d}} \text{ if }2q > d\\
            N\log N \text{ if }2q=d\\
            N \text{ if }2q<d
        \end{cases}
\end{equation}
  Using these estimates, it is easy to verify the Lyapunov criterion for $d\ge 4$:
  \begin{equation}
      \lim_{N \to \infty} \frac{\sum_{\mvk \neq \mv0} \bE((Z_{\mvk} - \bE(Z_{\mvk}))^4)
      }{(\var(\|Z'\|_2^2)^2} \to 0
  \end{equation}
   This immediately gives us the required Central limit theorem.
\end{proof}

\begin{lemma}\label{lem:critplb}
Let $d\geq 4$ and $\theta=C_{d}$. There exists $C>0$ such that 
\[\bP(\norm{Z'}_{2}^{2}+N|Z_{\mv0}|^{2}\leq N)\geq \begin{cases} C\frac{\sqrt{\log N}}{N} \text{ when }d=4\\ \frac{C}{\sqrt{N}}\text{ when }d\geq 5.\end{cases}\]

\end{lemma}
\begin{proof}
By the Green's function estimate of Lemma \ref{lem:gfconv},
we know that 
\[
|\E \norm{Z'}_{2}^{2} -N|\leq CN^{\frac{2}{d}}. 
\]
Thus, the fluctuations of $\norm{Z'}_{2}^{2}$, as characterized in Lemma \ref{lem:lyapunov} overwhelm the deterministic deviation of the expectation from $N$ for all $d\geq 4$. In particular, this allows us to conclude that for $s>0$ and $d=4$, by Lemma \ref{lem:lyapunov}
\begin{align*}
\bP(|Z_{\mv0}|^{2}\leq \frac{s\sqrt{\log N}}{\sqrt{N}},\norm{Z'}_{2}^{2}\leq N-s\sqrt{N \log N}) &= \frac{s\sqrt{\log N}}{\sqrt{N}}\cdot\bP(\norm{Z'}^{2}_{2} \leq N-s\sqrt{N \log N})\\ 
&\geq \frac{s\sqrt{\log N}}{\sqrt{N}}\cdot (F_{W}(-s)+o(1)),
\end{align*}
and for $d\geq 5$, 
\begin{align*}
\bP(|Z_{\mv0}|^{2}\leq \frac{s}{\sqrt{N}},\norm{Z'}_{2}^{2}\leq N-s\sqrt{N}) &= \frac{s}{\sqrt{N}} \bP(\norm{Z'}_{2}^{2}\leq N-s\sqrt{N})  \\ &\geq \frac{s}{\sqrt{N}}\cdot (F_{W}(-s)+o(1)).
\end{align*}
Where $F_{W}$ denotes the CDF of a standard real valued Gaussian random variable. Of course, both of these are bounded above by $\bP(\norm{Z'}_{2}^{2}+N|Z_{\mv0}|^{2}\leq N )$.
\end{proof}
The CLT method fails in $d=3$, because the mass contributed by the low frequency modes is significant. We will use a more hands on approach here. 
\begin{lemma}\label{lem:critplb3}
Let $d=3$ and $\theta =C_{d}$. Then there exists $C>0$ such that 
\[
\bP(\norm{Z'}^{2}_{2}+N|Z_{\mv0}|^{2}\leq N)\geq CN^{-\frac{1}{3}}(\log \log N)^{\frac{1}{3}}(\log N)^{-1}.
\]

\end{lemma}
\begin{proof}
Recall that $n$ denotes the side length of our discrete torus $\bT^{3}_{n}$, and $N=n^{3}$. Also recall the notion of $\eta$-corner introduced in \Cref{sec:masslessll}. We will denote by $\cS_{N}$ the $\eta-$corner, with the specific choice of $\eta=\frac{(c\log\log N)^{1/3}}{n}$ for some $c>0$. Now, we may separate 
\[
\norm{Z'}^{2}_{2}=\sum_{\mvk\in \cS_{N}}|Z_{\mvk}|^{2}+\sum_{\mvk\notin \cS_{N}}|Z_{\mvk}|^{2}. 
\]
We also introduce the notiation
\[
\norm{Z^{\cS_{N}}}_{2}^{2}:=\sum_{\mvk\in \cS_{N}}|Z_{\mvk}|^{2} \text{ and } \norm{Z^{\cS_{N}^{c}}}_{2}^{2}:=\sum_{\mvk\in \cS_{N}^{c} \setminus \{\mv0\}}|Z_{\mvk}|^{2}.
\]
Every $|Z_{\mvk}|^{2}$ is equal in distribution to $(\theta \gl_{\mvk})^{-1}X_{\mvk}$ where $\{X_{\mvk}\}$ is an i.i.d. family of exponential random variables with mean $1$. Consider the event
\[
\cV_{N}:=\{X_{\mvk}\leq \frac{1}{2} \text{ for all }\mvk \in \cS_{N}\}.
\]
There are $2^{3}\cdot c\log \log N$ vertices in $\cS_{N}$, and thus the probability of this event is $(1-e^{-1/2})^{c2^{3}\log \log N}$, which decays polylogaritmically as $N\to \infty$ We choose $c$ such that this probability is $(\log N)^{-1}$. On $\cV_{N}$,  
\[ \bE\norm{Z^{\cS_{N}}}_{2}^{2}-\norm{Z^{\cS_{N}}}_{2}^{2}1_{\cV_N} \geq \frac{1}{2}\E \norm{Z^{\cS_{N}}}_{2}^{2}.\] Using the explicit expression of eigenvalues $\gl_{\mvk}$, we find that 
\begin{align*}
\E \norm{Z^{\cS_{N}}}_{2}^{2}=\sum_{\mvk\in \cS_{N}}\frac{1}{\theta\gl_{\mvk}}&=\frac{N}{\theta}\cdot \frac{1}{N}\sum_{\mvk\in \cS_{N}}\frac{1}{4\sum_{i=1}^{3}\sin^{2} (\frac{\pi k_{i}}{n})} \\ &\geq \frac{N}{\theta }\cdot \frac{1}{n^{3}}\sum_{\mvk \in \cS_{N}}\frac{1}{4\pi^{2}\norm{\mvk/n}^{2}}.
\end{align*}
The last inequality follows from the upper bound $\sin(x)\leq x$ for all $x\geq 0$. We may directly compare this resulting sum to the integral of $\norm{x}^{-2}$ where $x\in [0,\frac{(c\log \log N)^{\frac{1}{3}}}{n}]^{3}$, excluding the cube with corner $0$ and side length $n^{-1}$. Clearly, the above sum is the Riemann sum approximation to this integral. In particular, there exists a constant $C>0$ such that 
\begin{align*}
\frac{1}{N}\sum_{\mvk\in \cS_{N}}\frac{1}{\norm{\mvk/n}^{2}} &\geq C\int_{\frac{1}{n}\leq \norm{x}\leq \frac{(c\log \log N)^{1/3}}{n} }\norm{x}^{-2}
dx\geq C (\frac{{c\log \log N}}{N} )^{\frac{1}{3}}, 
\end{align*}
where we use the fact that $\norm{x}^{-2}$ is integrable in a bounded neighborhood of $0$ when $d=3$.  Using the bounds $\sin(x)\geq \frac{2}{\pi}x$ for $x\in [0,\frac{\pi}{2}]$ and $\sin(x)\geq \frac{2}{\pi}(\pi-x)$ for $x\in[\frac{\pi}{2},\pi]$ (both bounds follow from concavity) yields a corresponding upper bound as well, and tells us that in fact
\begin{align}\label{eq:eigasymp}
\frac{1}{N}\sum_{\mvk \in \cS_{N}}\frac{1}{\theta \gl_{\mvk}}\asymp (\frac{\log \log N}{N})^{\frac{1}{3}}. 
\end{align}
for some strictly positive $C>0$. Next, we consider the contribution of $\cS_{N}^{c}$. Here, we need to bound two things. First, the gap between $\E \norm{Z^{\cS_{N}^{c}}}_{2}^{2}$ and $N$. The Green's function estimate in \Cref{lem:gfconv} tells us that $$|\bE \norm{Z'}_{2}^{2}-N|\leq N^{2/3}.$$ Since $\E \norm{Z'}=\E \norm{Z^{\cS_{N}}}_{2}^{2}+\E\norm{Z^{\cS_{N}^{c}}}$, we know that 
\[
\E\norm{Z^{\cS_{N}^{c}}}=N+O(N^{2/3})-\E\norm{Z^{\cS_{N}}}_{2}^{2},
\]
and by \eqref{eq:eigasymp},
\[
\E \norm{Z^{\cS_{N}}}_{2}^{2}-O(N^{2/3})=\E \norm{Z^{\cS_{N}}}_{2}^{2}(1+o(1)).
\]
Next, the variance of $\norm{Z^{\cS_{N}^{c}}}_{2}^{2}$. By the same bounds on the $\sin$ function, 
\[
\var \norm{Z^{\cS_{N}^{c}}}=\sum_{\mvk \notin \cS_{N}^{c}}\frac{\var(X_{\mvk})}{\theta^{2}\gl_{\mvk}^{2}}\leq \frac{N}{\theta^{2}}\cdot \frac{1}{N}\sum_{\mvk\notin \cS_{N}} \frac{1}{\norm{\mvk/n}^{4} }.
\]
Again, using the comparison to an integral, we find that 
\[
\var \norm{Z^{\cS_{N}^{c}}}_{2}^{2}\leq CN\int_{n^{-1}(c\log \log N)^{1/3}}^{1}r^{-2}dr\leq CN(\frac{N}{{c\log\log N}})^{1/3}. 
\]
Note that $\sqrt{\var\norm{Z^{\cS_{N}^{c}}}_{2}^{2}}$ is of smaller order than $N^{2/3}$, and in particular is of smaller order than $\E \norm{Z^{\cS_{N}}}_{2}^{2}$. By Chebyshev's inequality, 
\[
\bP(\bigl|\norm{Z^{\cS_{N}^{c}}}_{2}^{2}-\E\norm{Z^{\cS_{N}^{c}}}_{2}^{2}\bigr|\geq CN^{2/3})\leq \frac{C}{(\log \log N)^{1/3}}.
\]
Now, consider the event
\[
\tilde{\cV}_{N}:=\{|Z_{\mv0}|^{2}\leq \frac{1}{3N}\E\norm{Z^{\cS_{N}}}_{2}^{2}\}\cap \cV_{N}\cap \{\bigl|\norm{Z^{\cS_{N}^{c}}}_{2}^{2}-\E\norm{Z^{\cS_{N}^{c}}}_{2}^{2}\bigr|\leq CN^{\frac{2}{3}}\}.
\]
On $\tilde{\cV}_{N}$, by construction 
\[
\biggl(\norm{Z'}_{2}^{2}+N|Z_{\mv0}|^{2}\biggr)1_{\tilde{\cV}_{N}}\leq N-\frac{1}{6}\E\norm{Z^{\cS_{N}}}_{2}^{2}(1+o(1))\leq N,
\]
where the last inequality is clearly true for $N$ large enough. For large enough $N$, $\tilde{\cV}_{N}\subset \{\norm{Z'}_{2}^{2}+N|Z_{\mv0}|^{2}\leq N\}$. Using independence and the fact that $|Z_{\mv0}|^{2}$ is uniform on $[0,1]$, we may calculate the probability of $\tilde{\cV}_{N}$ to obtain 
\[
\frac{1}{3N}\E\norm{Z^{\cS_{N}}}_{2}^{2}\cdot (\log N)^{-1}\cdot \bigl(1-\frac{C}{(\log \log N)^{1/3}} \bigr)\leq \bP(N|Z_{\mv0}|^{2}+\norm{Z'}_{2}^{2}\leq N).
\]
Applying \eqref{eq:eigasymp} completes the proof. 
\end{proof}

Now, we finally verify the convergence of the relevant moments.
\begin{proof}[Proof of \Cref{thm:main_sph}, item b. ($\theta=C_{d}$ case)] We will need to consider two separate subcases, the first corresponding to indices $\iota$ that exclude an $\eta-$corner, and the second the indices that are in the corner. The exact dependence of $\eta$ on $N$ will be specified at the appropriate juncture, and will be $d-$dependent. The goal is to show that the relevant moments converge to the corresponding GFF moments. For the corner, the goal is to show that the moments  do not contribute in the limit. This includes $\mv0$, and there is no global shift like the $\theta >C_d$ case. We begin with the first case, and focus on a single index $\iota$, so in this case $\alpha((\sqrt{N}Z_{\mv0}, Z'), \iota) = \alpha(Z', \iota)$. Let $\gb$ denote the collection of $\mvk$ that appear in the monomial $\ga(Z',\iota)$. The term in question is 
\begin{align}\label{eq:thetacrit1}
\frac{\bE\left(\ga(Z',\iota)1_{\norm{Z^{\sbeta}}_{2}^{2}+\norm{Z^{\gb}}_{2}^{2}+N|Z_{\mv0}|^{2}\leq N}\right)}{\bP(\norm{Z'}^{2}+N|Z_{\mv0}|^{2}\leq N)}.
\end{align}
Now, recall that 
\[
\norm{Z^{\gb}}_{2}^{2}=\sum_{\mvk\in \gb}\frac{X_{\mvk}}{\gl_{\mvk}},
\]
where $X_{\mvk}$ are i.i.d. exponentials with mean $1$.
Define  \[\cA_{N,\gb}:=\{\norm{X^{\gb}}_{\infty}\leq K\log N\}.\] We choose $K\log N$ such that the corresponding probability of the complement vanishes polynomially, and the explicit choice of $K$ will be made at the appropriate juncture.  Note $\gb$ contains at most $q+q'$  indices not from the $\eta-$corner. 
Now, we may rewrite \eqref{eq:thetacrit1} as 
\begin{align}\label{eq:nonzeromoment}
\frac{\bE\left(\ga(Z',\iota)1_{\cA_{N,\gb}}\cdot 1_{\norm{Z^{\sbeta}}_{2}^{2}+\norm{Z^{\gb}}_{2}^{2}+N|Z_{\mv0}|^{2}\leq N}\right)}{\bP(\norm{Z'}^{2}+N|Z_{\mv0}|^{2}\leq N)}+\frac{\bE\left(\ga(Z',\iota)1_{\cA_{N,\gb}^{c}}1_{{\norm{Z'}_{2}^{2}}+N|Z_{\mv0}|^{2} \le N}\right)}{\bP({\norm{Z'}_{2}^{2}}+N|Z_{\mv0}|^{2}\leq N)}
\end{align}
The second term is the error and an upper bound suffices. Let $m$ be a positive integer.  Applying the H\"older inequality, 
\begin{align*}
\frac{\bE(\ga(Z',\iota)1_{\cA_{N,\gb}^{c}}1_{{\norm{Z'}_{2}^{2}}+N|Z_{\mv0}|^{2}\leq N})}{\bP({\norm{Z'}_{2}^{2}}+N|Z_{\mv0}|^{2}\leq N)} &\leq \frac{\bP( \cA^{c}_{N,\gb})^{1-\frac{1}{2m}}}{\bP(\norm{Z'}^{2}_{2}+N|Z_{\mv0}|^{2}\leq N)^{\frac{1}{2m}}}({\bE \ga(Z',\iota)^{2m}})^{\frac{1}{2m}} 
\end{align*}

Now we consider the first term. Since we are excluding points in the $\eta-$corner, the best bound we can get for $\|Z^{\gb}\|_2^2$ on $\cA_{N,\gb}$ is $O(\eta^{-2}\log N)$. We can rewrite the expectation by conditioning with respect to $Z^{\gb}$. However, before we do this, observe that for any any $z>0$
\begin{align*}
&\bP({\norm{Z^{\sbeta}}^{2}}+N|Z_{\mv0}|^{2}\leq N) -\bP({\norm{Z^{\sbeta}}^{2}}+N|Z_{\mv0}|^{2}\leq N-Nz)\\ =&\bP(1-z\leq\frac{\norm{Z^{\sbeta}}_{2}^{2}}{N}+|Z_{\mv0}|^{2} \leq 1) =z.
\end{align*}
where the last equality follows from the fact that $Z^{\sbeta}$ and $Z_{\mv0}$ are independent, and $|Z_0|^2\sim $Unif$[0,1]$.
Recall that $\ga(Z',\iota)$ depends only on the indices corresponding to $\gb$ by definition. We express the numerator in terms of conditioning with respect to $\norm{Z^{\gb}}^{2}_{2}$. Let $p_{\gb}(w)$ denote the density of $Z^{\gb}$. Applying the above estimate, we obtain 
\begin{align*}
\E (\ga(Z',\iota)1_{\cA_{N,\gb}}&\E(1_{\norm{Z^{\gb}}+\norm{Z^{\sbeta}}+N|Z_{\mv0}|^{2}\leq N}|Z^{\gb}) )\\&=\int_{\cA_{N,\gb}} \ga(w)\cdot \bP(\norm{Z^{\sbeta}}^{2}_{2}+N|Z_{\mv0}|^{2}+|w|^{2}\leq N)\cdot p_{\gb}(w)dw\\ &=\bigl(\bP(\norm{Z^{\sbeta}}+N|Z_{\mv0}|^{2}\leq N)+O(\frac{\eta^{-2}\log N}{N})\bigr)\E \ga(Z',\iota)1_{\cA_{N,\gb}}.
\end{align*}
Thus, the numerator is 
\[
\bigl(\bP(\norm{Z^{\sbeta}}+|Z_{\mv0}|^{2}\leq N)+O(\frac{\eta^{-2}\log N}{N})\bigr)\E\ga (Z',\iota)1_{\cA_{N,\gb}}
\]
For the denominator, the same conditioning argument as the numerator tells us that 
\begin{align*}
\bP(\frac{\norm{Z'}_{2}^{2}}{N}+|Z_{\mv0}|^{2}\leq 1)&=\bP(\{\frac{\norm{Z'}_{2}^{2}}{N}+|Z_{\mv0}|^{2}\leq 1\}\cap \cA_{N,\gb})+\bP(\{\frac{\norm{Z'}_{2}^{2}}{N}+|Z_{\mv0}|^{2}\leq 1\}\cap\cA_{N,\gb}^{c}) \\
&= \bP({\norm{Z^{\sbeta}}_{2}^{2}}+N|Z_{\mv0}|^{2}\leq N) +O(\eta^{-2}\frac{\log N}{N}) +\gc_{N}, 
\end{align*}
where 
\[
0\leq \gc_{N}\leq e^{-K\log N}\bP(\norm{Z^{\sbeta}}_{2}^{2}+N|Z_{\mv0}|^{2}\leq N).
\]
Clearly, $\gc_{N}$ goes to zero strictly faster than $\bP(\norm{Z^{\sbeta}}_{2}^{2}+N|Z_{\mv0}|^{2}\leq N)$.
Thus, \eqref{eq:thetacrit1} may be recast as  
\begin{align*}
&\frac{\bigl(\bP(\norm{Z^{\sbeta}}_{2}^{2}+N|Z_{\mv0}|^{2}\leq N)+O(\frac{\eta^{-2}\log N}{N})\bigr)\E\ga (Z',\iota)}{\bP({\norm{Z^{\sbeta}}_{2}^{2}}+N|Z_{\mv0}|^{2}\leq N) +O(\frac{\eta^{-2}\log N}{N}) +\gc_{N}} \\  =&\frac{(1+\frac{O(\frac{\eta^{-2}\log N}{N})}{\bP(\norm{Z^{\sbeta}}_{2}^{2}+N|Z_{\mv0}|^{2}\leq N)})}{(1+\frac{O(\frac{\eta^{-2}\log N}{N})+\gc_{N}}{\bP(\norm{Z^{\sbeta}}_{2}^{2}+N|Z_{\mv0}|^{2}\leq N)})}\E \ga(Z',\iota)1_{\cA_{N,\gb}} .   
\end{align*}
It is at this juncture that we introduce the bounds on $\bP(\norm{Z'}_{2}^{2}+N|Z_{\mv0}|^{2}\leq N)$ calculated in Lemma \ref{lem:critplb}. Recall, 
\[
\bP(\norm{Z^{\sbeta}}_{2}^{2}+N|Z_{\mv0}|^{2}\leq N)\geq \begin{cases} C\frac{(\log \log N)^{1/3}}{N^{1/3}\log N} \text{ when }d=3\\C\frac{\sqrt{\log N}}{\sqrt{N}}\text{ when }d=4 \\ \frac{C}{\sqrt{N}}\text{ when }d\geq 5.\end{cases}
\]
When $d\geq 5$, notice that with the smallest possible choice for $\eta$, namely $\eta=N^{-\frac{1}{d}}$, we have
\[
\frac{O(\frac{N^{\frac{2}{d}}\log N}{N})}{\bP(\norm{Z^{\sbeta}}_{2}^{2}+N|Z_{\mv0}|^{2}\leq N)}=O(\frac{\log N}{N^{\frac{d-4}{2d}}}),
\]
which of course decays to $0$. When $d=4$, the worst choice of $\eta$ is no longer adequate. Let $\eta=(\log N)^{2}\cdot N^{-\frac{1}{d}}$. Note that in this case, we have that 
\[
\frac{O(\frac{N^{\frac{2}{d}}\log N}{N\log ^{4}N})}{\bP(\norm{Z^{\sbeta}}_{2}^{2}+N|Z_{\mv0}|^{2}\leq N)}=O(\frac{1}{(\log N)^{\frac{5}{2}}}). 
\]
Finally, for $d=3$, we choose $\eta=(\log N)^{2}\cdot N^{-1/d}$, and  
\[
\frac{O(\frac{N^{\frac{2}{d}}\log N}{N (\log N)^{4}})}{\bP(\norm{Z^{\sbeta}}_{2}^{2}+N|Z_{\mv0}|^{2}\leq N)}=O(\frac{(\log \log N)^{\frac{1}{3}}}{(\log N)^{2}}).
\]
In all three cases, 
\[
\frac{(1+\frac{O(\frac{\eta^{-2}\log N}{N})}{\bP(\norm{Z^{\sbeta}}_{2}^{2}+N|Z_{\mv0}|^{2}\leq N)})}{(1+\frac{O(\frac{\eta^{-2}\log N}{N})+\gc_{N}}{\bP(\norm{Z^{\sbeta}}_{2}^{2}+N|Z_{\mv0}|^{2}\leq N)})}\E \ga(Z',\iota)1_{\cA_{N,\gb}}=(1+o(1))\E \ga(Z',\iota)1_{\cA_{N,\gb}}
\]
An application of the Cauchy Schwarz inequality allows us to discard the indicator on $\cA_{N,\gb}$, since
\[
\E \ga(Z',\iota)1_{\cA_{N,\gb}}=\E \ga(Z',\iota)-\E \ga(Z',\iota)1_{\cA_{N, \gb}^{c}},
\]
and 
\[
\E \ga(Z',\iota)1_{\cA_{N,\gb}^{c}}\leq \sqrt{\E \ga(Z',\iota)^{2}}\cdot e^{-K\log N}. 
\]
Combining all our bounds, we see that for indices excluding $0$, 
\begin{align*}
 \frac{ \bE(\alpha((\sqrt{N}Z_{\mv0}, Z'), \iota)  1_{N|Z_0|^2+\|Z'\|_2^2 \le N} ) }{\bE(1_{N|Z_0|^2+\|Z'\|_2^2 \le N})}=(1+o(1))&\bE\ga (Z',\iota)  +e^{-\gO(\log N)}\sqrt{\bE \ga(Z',\iota)^{2}} \\ &+\frac{e^{-\gO(\log N)}(\E\ga(Z',\iota)^{2m} )^{\frac{1}{2m}}}{\bP(\norm{Z'}^{2}_{2}+N|Z_{\mv0}|^{2}\leq N)^{\frac{1}{2m}}}
 \end{align*}

We would like to point out to the reader now that since $\ga$ is a product of even powers of independent Gaussians, where each of the powers is bounded above by $q+q'$, there exists a constant $C(q,m)$ such that 
 \begin{align}\label{eq:mombound}
 (\bE {\ga(Z',\iota)^{2m}})^{\frac{1}{2m}}\leq C(q,m) \bE\ga (Z',\iota). 
 \end{align}
We utilize this to further obtain 
\begin{align*}
\frac{ \bE(\alpha((\sqrt{N}Z_{\mv0}, Z'), \iota)  1_{N|Z_0|^2+\|Z'\|_2^2 \le N} ) }{\bE(1_{N|Z_0|^2+\|Z'\|_2^2 \le N})}&=(1+o(1))\E \ga(Z' ,\iota )\\&+\left(e^{-\gO(\log N)}+\frac{e^{-\gO(\log N)}}{\bP(\norm{Z'}_{2}^{2}+N|Z_{\mv0}|^{2}\leq N)^{\frac{1}{2m}}}\right)\E \ga(Z',\iota).
\end{align*}
Finally, we may choose $m$ sufficiently large, so that 
\[
\frac{e^{-\gO(\log N)}}{\bP(\norm{Z'}_{2}^{2}+N|Z_{\mv0}|^{2}\leq N)^{\frac{1}{2m}}}=o(1). 
\]
Also remember that by construction, 
\[
0<\sum_{\iota \in \Xi'}\ga(Z',\iota)+\tilde{\ga}(Z',\iota)=O(1).
\]
We now need to deal with the $\eta-$corner. We will repeat the trick with H\"older's inequality. Note that for an integer $m\geq 0$ to be chosen,
\[
\frac{\sum_{\iota \in \Xi'} \bE(\alpha((\sqrt{N}Z_{\mv0}, Z'), \iota)  1_{N|Z_0|^2+\|Z'\|_2^2 \le N} ) }{\bE(1_{N|Z_0|^2+\|Z'\|_2^2 \le N})}\leq (\E \ga(\sqrt{N}Z_{\mv0},Z',\iota)^{2m})^{\frac{1}{2m}}\cdot \bP(\norm{Z'}_{2}^{2}+N|Z_{\mv0}|^{2}\leq N)^{-\frac{1}{2m}}. 
\]
Using \eqref{eq:mombound}, this is bounded above by 
\[
C(m,q)\E \ga(\sqrt{N}Z_{\mv0},Z',\iota)\cdot \bP(\norm{Z'}_{2}^{2}+N|Z_{\mv0}|^{2}\leq N)^{-\frac{1}{2m}} 
\]
We note that \eqref{eq:mombound} still holds despite the presence of $Z_{\mv0}$, since $\sqrt{N}\hat{\fg}_{\mvo}|Z_{\mv0}|^{2}$ is bounded above and positive. We now sum over all $\iota $ in the corner and apply the worst case of Lemma \ref{lem:corner_estimate} to obtain the upper bound 
\[
\frac{\sum_{\iota \in \Xi_{corn}} \bE(\alpha((\sqrt{N}Z_{\mv0}, Z'), \iota)  1_{N|Z_0|^2+\|Z'\|_2^2 \le N} ) }{\bE(1_{N|Z_0|^2+\|Z'\|_2^2 \le N})}\leq  O(\eta|\log \eta|)\cdot \bP(\norm{Z'}_{2}^{2}+N|Z_{\mv0}|^{2}\leq N)^{-\frac{1}{2m}}.
\]
We can choose $m$ large enough such that for all three cases of $\eta$ and lower bounds of $\bP(\norm{Z'}_{2}^{2}+N|Z_{\mv0}|^{2}\leq N)$, the above expression is $o(1)$. Finally, summing over all $\iota \in \Xi'$ and using \eqref{eq:ZGFFrel}, 
 \begin{align*}=(1+o(1))\bE (\fR \la \theta^{-\frac{1}{2}}\Phizero^{N},\fg \ra)^{q}(\fI \la\theta^{-\frac{1}{2}}\Phizero^{N},\fg \ra )^{q'})+o(1).
 \end{align*}
 This completes the proof.  
 \end{proof}
 Recall that an important ingredient in the proof of Theorem \ref{thm:main_summary} in the subcritical case was the tail bound in the spherical law. For the case $\theta\geq C_{d}$ this was Lemma \ref{lem:sphericaltail2}. We are now in a position to prove this.  
\begin{proof}[Proof of \Cref{lem:sphericaltail2}]
Fix a lattice site $\mvx \in \bT^{d}_{n}$, and note using \eqref{eq:ZGFFrel} that 
\begin{align*}
\bP(|\Psi_{{\sf sph},N}^{\theta}(\mvx)|\geq b) &= \frac{\bE \left(1_{|\theta^{-\frac{1}{2}}\Phizero^{N}(\mvx)+Z_{\mv0}|>b}\cdot  1_{\norm{\theta^{-\frac{1}{2}}\Phizero^{N}}_{2}^{2}+N|Z_{\mv0}|^{2}\leq N}\right)}{\bP(\norm{\theta^{-\frac{1}{2}}\Phizero^{N}}_{2}^{2}+N|Z_{\mv0}|^{2}\leq N)} \\ 
&\leq e^{O(\log N)} \cdot \bP(|\theta^{-\frac{1}{2}}\Phizero^{N}(\mvx)+Z_{\mv0}|>b)\\
&\leq e^{O(\log N)}\cdot \bP(|\theta^{-\frac{1}{2}}\Phizero^{N}(\mvx)|>b-1)\\
&\leq e^{O(\log N)}\cdot e^{-\frac{(b-1)^{2}}{2}}\\
&\leq e^{-(b-1)^{2}/2 +O(\log N)}\leq e^{-b^{2}/4}, 
\end{align*}
where the first inequality follows from the bound in Lemma \ref{lem:critplb}, and the last inequality in the chain holds for all $N$ large enough since $b$ was chosen to be larger than $\log N$.
\end{proof}
 \subsection{The massive case $\theta<C_d$}
We know turn to the characterization of the local limit in the massive regime $(\theta<C_{d})$. We will also need to characterize the local limit in this regime with boundary conditions, which is Theorem \ref{lem:llimit1}.  Note that this is a substantial generalization from part a of \Cref{thm:main_sph}, which is for the case $U=\emptyset$. The proof of \Cref{lem:llimit1}, much like that of the massless case, relies on representing the relevant moments of the spherical law in terms of expectations involving an independent family of random variables with nice densities.  Let $|U^{c}| = k$ and let $\phi_1, \ldots, \phi_k$ be an orthonormal basis of $\bC^{U^{c}}$ consisting of eigenvectors of $-\Delta_{U^{c}}$ (recall $-\Delta_{U^{c}}$ is the restriction of the negative Laplacian to $U^{c}$, defined precisely at the beginning of \Cref{sec:gff}). Let $\lambda_1, \ldots, \lambda_k$ be the respective eigenvalues, which are all non-negative. We will still refer to the operation of changing basis to $\phi_{1},\phi_{2},\ldots\phi_{k}$ as the Fourier transform. 

\begin{lemma}\label{lem:fourier_law}
Let $U \subset \bT_n^d$, $| U^{c}| = k$ and fix any $m>0$. Write 
$$
\Psisph^{f, \theta, U^{c}} = \Psi_1+ h
$$
where $h$ is the massive harmonic extension of $f$ with mass $m$. Let $Y = (Y_i)_{1 \le i \le k}$ where  $Y_i = \la \Psi_1, \phi_i\ra  =  (\hat \Psi_1)_i$ be the Fourier transform of $\Psi_1$. Let $\hat h_i = \la h, \phi_i\ra$. Then for any bounded continuous function $\ff:\bC^{k} \to \bR$
\begin{equation*}
    \bE(\ff(Y)) = \frac{\bE\left(\ff(Z)\exp(\theta m \|Z+\hat h\|_2^2)1_{\|Z+\hat h\|_2^2 \le  N }\right)}{\bE\left(\exp(\theta m \|Z+\hat h\|_2^2)1_{\|Z+\hat h\|_2^2 \le  N }\right)}
\end{equation*}
where $Z=(Z_1,\ldots, Z_n)$ are independent and $Z_i $ is distributed as a complex Gaussian with variance $(2\theta(\lambda_i+m))^{-1}$. 
\end{lemma}
\begin{proof}
First of all, we are going to write the density of $\Psisph^{f, \theta, U^{c}}$, first expressed \Cref{eq:density_spherical_boundary_general_mass} for $\gamma=1$, as below

\begin{equation*}
   \exp\left(\theta m \|\psi\|_2^2-\theta \norm{\nabla \psi}_{2}^{2} - \theta m \|\psi\|_2^2\right)\mv1_{\norm{\psi|_{U^c}}_{2}^{2}\leq  N}1_{\psi| \partial U\equiv f |\partial U},  \qquad \psi \in \bC^{\bar U^c}
\end{equation*}
which will amount to tilting the law of the fourier transforms as we shall see.
Now we decompose any $\psi$ with boundary values $f$ on $\partial U$  as $$\psi = \psi_1 + h$$
where $h$ is the  massive harmonic extension of $f$ with mass $m$ onto $U^{c}$. Note that $\psi_{1}$ is identically 0 on $\partial U$. In the graph $G_{U^{c}}$ (recall $G_{U^{c}}$ is the subgraph induced by $U^{c}$)
$$
\norm{\nabla \psi}_{2}^{2} +  m\|\psi\|_2^2 = \la (\psi_1 + h, (-\Delta_{U^c}+m) (\psi_1+h) \ra  = \la \psi_1 , (-\Delta_{U^{c}}+m) \psi_1 \ra + \la h, (-\Delta_{U^c}+m) h\ra 
$$
Indeed, $$\la \psi_1, (-\Delta_{U^c} +m) h\ra = \la h, (-\Delta_{U^c}+m)\psi_1 \ra = 0 $$
since $\psi_1$ is 0 on $U$ and $h$ is massive harmonic with mass $m$ on $U^{c}$.  Since $h$ does not depend on $\psi_1$,  we can write for any bounded continuous function ${\f f}:\bC^{U^{c}} \to \bR$,
\begin{equation*}
    \bE({\ff}(\Psisph^{f, \theta, U^{c}})) = \frac{1}{\mathcal Z}\int_{\bC^{U}}  {\ff}(\psi_1+h)\exp\left(\theta m\|\psi_1+h\|_2^2+\theta\la \psi_1 , (\Delta_{U^{c}}-m) \psi_1 \ra\right)\mv1_{\norm{\psi_1 + h_U}_{2}^{2}\leq  N}\vd \psi_1.
\end{equation*}

where 
$$
\cZ = \int_{\bC^{U^{c}}}  \exp\left(\theta m\|\psi_1+h\|_2^2+\theta\la \psi_1 , (\Delta_{U^{c}}-m) \psi_1 \ra\right)\mv1_{\norm{\psi_1 + h}_{2}^{2}\leq  N}\mv\vd \psi_1.
$$
 Define $y = (y_i)_{1 \le i \le k}$ and $\hat h = (\hat h_i)_{1 \le i \le k}$ to be the Fourier transforms of $\psi_1$ and $h$, i.e., let 
 \begin{equation}
   y_i = \la \psi_1, \phi_i\ra=:{(\hat\psi_1)_i}; \qquad \hat h_i = \la h,\phi_i\ra = \la h, \phi_i\ra.   \label{eq:yhath}
 \end{equation}
  We can now write
\begin{equation*}
 \|\psi_1+h\|_2^2 = \|y+\hat h\|_2^2; \qquad   \la \psi_1 , (-\Delta_{U^{c}}+m) \psi_1 \ra  = \sum_{j=1}^k (\lambda_i+M)|y_i|^2.
\end{equation*}
Since the transformation $\psi_1 \mapsto y  $ is orthogonal, the Jacobian of this transformation has absolute value $1$. Now we define the random vector $Y$ as the Fourier transform of $\Psi_1$ where $\Psi_1 := (\Psisph^{f, \theta, U^{c}}-h)$. For any $\ff:\bC^k \to \bR$, we can write $Y$
\begin{equation}
    \bE(\ff(Y)) = \frac1{\cZ}\int_{\bC^k} \ff(y)\exp\left(\theta m\|y+\hat h\|_2^2-\theta\sum_{i=1}^k (\lambda_i+m)|y_i|^2 \right)\mv1_{ \| y + \hat h\|^2\leq  N}\vd y
\end{equation}
where 
$$
\cZ = \int_{\bC^k} \exp\left(\theta m\|y+\hat h\|_2^2-\theta\sum_{i=1}^k (\lambda_i+m)|y_i|^2 \right)\mv1_{ \| y + \hat h\|^2\leq  N}\vd y
$$
Since 
$$
\frac{\prod_{i=1}^k (\theta(\lambda_i+m))}{\pi^{k}} \exp(-\theta\sum_{i=1}^k (\lambda_i+m)|y_i|^2)
$$
is the density of $(Z_1,\ldots,Z_k)$ as described in the lemma, we are done.

\end{proof}
At this juncture, just as in Section \ref{sec:masslessll}, we would like to point out for the family of random variables $(Z_{i})_{i=1}^{k}$ defined in the proof of Lemma \ref{lem:fourier_law}, 
\begin{align}\label{eq:ZMGFFrel}
\sum_{i=1}^{k}Z_{i}\phi_{i}(\mvx)\equald \theta^{-\frac{1}{2}}\cdot \Phi_{m}^{U^{c}}(\mvx). 
\end{align}
This allows us to relate the moments of $Z$ and $\Phi^{U^c}_m$. 
\begin{proof}[Proof of \Cref{lem:llimit1}]
Since $h|_A \to 0$ because of exponential decay (first item of \Cref{lem:l_p_h}), it is enough to prove convergence of $\Psi_1$ to the correct massive GFF limit using the method of moments. Let $\mathfrak g$ be a complex valued function supported on $A$. Note for any $\psi_1:\bC^U \to \bR$
\begin{equation*}
    \la \mathfrak g, \psi_1\ra =  \la y,\hat \fg\ra
\end{equation*}
where $\hat \fg = (\hat \fg_i)_{1 \le i \le k}$ are the Fourier transforms of $\fg$ and $y$ is as in \eqref{eq:yhath}.
Using \Cref{lem:fourier_law} and the same decomposition as in \eqref{eq:alpha_indices}, for any integer $q \ge 1, q' \ge 1$, and any $m>0$, 
\begin{equation}
    \bE( (\fR(\la \mathfrak{g}, \Psi_1\ra))^ q\fI(\la \mathfrak{g}, \Psi_1\ra))^ {q'} ) =\sum_{\iota \in \Xi'} \bE(\alpha(Y, \iota)) =\frac{\sum_{\iota \in \Xi'}\bE\left(\alpha(Z, \iota)\exp(\theta m \|Z+\hat h\|_2^21_{\|Z+\hat h\|_2^2 \le  N }\right)}{\bE\left(\exp(\theta m \|Z+\hat h\|_2^21_{\|Z+\hat h\|_2^2 \le  N }\right)} \label{eq:moment_expression} 
\end{equation}
where $\alpha(Z,\iota)$ is as in \eqref{eq:alphaYi} and $\Xi'$ is the set of indices where every term is paired with its conjugate in \eqref{eq:alphaYi} (as defined after \eqref{eq:ratio_super}).
It is now enough to prove the following lemma.
\begin{lemma}\label{lem:enough}
Suppose $Z, \fg, \hat \fg, \hat h$ be as above. Then 
\begin{align}
    \frac{\bE\left(\alpha( Z, \iota)\exp(\theta m \|Z+\hat h\|_2^21_{\|Z+\hat h\|_2^2 \le  N })\right)}{\bE\left(\exp(\theta m \|Z+\hat h\|_2^21_{\|Z+\hat h\|_2^2 \le  N }\right)}\nonumber= \bE(\alpha(Z,\iota))(1+o(1)).\label{eq:enough}
\end{align}
\end{lemma}
Indeed, given \Cref{lem:enough}, we see that summing over all possible $\iota \in \Xi'$
\begin{multline*}
    \sum_{\iota \in \Xi'} \bE(\alpha(Y, \iota)) =  \sum \bE(\alpha(Z,\iota))(1+o(1)) \\= \bE(\fR(\la \hat \fg, Z\ra)^q\fI(\la \hat \fg, Z\ra)^{q'}) (1+o(1)) = \bE(\fR(\la \fg, \theta^{-\frac{1}{2}}\Phi_m\ra)^q\fI(\la \fg, \theta^{-\frac{1}{2}}\Phi_m\ra)^{q'}) (1+o(1)).
\end{multline*}
Since $\la\fg, \theta^{-\frac{1}{2}}\Phi_m\ra$ is complex Gaussian, its distribution is determined by the joint  moments of its real and imaginary parts. Hence by method of moments,  $\la  \fg, \Psi_1\ra \to \la  \fg, \theta^{-\frac{1}{2}}\Phi_M\ra$ in law.  Consequently $\Psi_1 $ converges locally to $\Phi_m$ as desired. The proof of the theorem is now complete, modulo the proof of  \Cref{lem:enough} which we prove below.
\end{proof}

\begin{proof}[Proof of \Cref{lem:enough}]
Take $\beta  \in \Xi'$.
For a vector $y$, let $y^{\sbeta}$ denote the portion of the vector $y$ without the indices in $\beta$ and $y^\beta$ denote the portion of $y$ with indices in $\beta$.  We similarly define the corresponding quantities for $\hat h,Z$:  $\hat h^\beta$ and $\hat h^{\sbeta} $ and $Z^\beta$ and $Z^{\sbeta}$. Then we break up \eqref{eq:moment_expression} as follows. Write $\bE_{\beta}$ to be the conditional expectation $\bE(\cdot | Z^\beta)$. Then we can write the left hand side as
\begin{equation}
   \frac{ \bE\left(\alpha( Z, \iota)\exp(\theta M\zeta(Z^\beta))\bE_{\beta}\left(\exp(\theta m \|Z^\sbeta+\hat h^\sbeta\|_2^2)1_{\|Z^\sbeta+\hat h^\sbeta\|_2^2 \le  N -\zeta(Z^\beta)}\right)\right)}{\bE\left(\exp(\theta m \|Z+\hat h\|_2^2)1_{\|Z+\hat h\|_2^2 \le  N }\right)}\label{eq:ratio_conditional}
\end{equation}
 where 
 $$
 \zeta(Z^\beta) = \|Z^\beta + h^\beta\|_2^2.  
 $$

Let us concentrate on the conditional expectation of the numerator, the denominator can be estimated by similar logic. We will apply 
a sharpened Lindeberg-Feller theorem (\cref{lem:lclt}) to the the independent random variables $(|Z_i  + \hat h_i|^2)_{i \in \beta}$, we postpone the exact statement of that lemma for now. The expectation and variance are easily computed to be 
\begin{align*}
    \bE(|Z_i  + \hat h_i|^2) = \frac1{\theta(\lambda_i+m)} + |\hat h_i|^2\\
    \text{Var}(|Z_i  + \hat h_i|^2) = \frac1{\theta^2(\lambda_i+m)^2} + \frac{2|\hat h_i|^2}{\theta(\lambda_i+m)}
\end{align*}

Now define
\begin{align*}
    A_{N,\sbeta} &:= \sum_{i \not \in \beta} \bE(|Z_i  + \hat h_i|^2); \qquad A_{N} := \sum_{i =1}^k \bE(|Z_i  + \hat h_i|^2) \\
    B^2_{N,\sbeta} &:= \sum_{i \not \in \beta}\text{Var}(|Z_i  + \hat h_i|^2)\qquad B^2_{N} := \sum_{i =1}^k\text{Var}(|Z_i  + \hat h_i|^2).
\end{align*}
Let $p_N(u) $ denote the density of $\|Z^\sbeta+\hat h^\sbeta\|_2^2$. Conditioned on $Z^\beta = y^\beta$ (for some fixed vector $y^\beta$ independent of $N$), the conditional expectation can be written as
\begin{equation}
    \bE_{\beta}\left(\exp(\theta m \|Z^\sbeta+\hat h^\sbeta\|_2^21_{\|Z^\sbeta+\hat h^\sbeta\|_2^2 \le  N - \zeta(y^\beta))}\right) = \int_0^{ N-\zeta(y^\beta)} \exp(\theta mu)p_N(u)du\label{eq:num_local}
\end{equation}
 At this stage, the choice of mass parameter $m_{N}$ becomes important. We choose $m_{N}$ such that 
\begin{equation}
    A_{N}= \sum_{i=1}^k \frac1{\theta(\lambda_i+m_{N})} + \|\hat h\|_2^2 =  N.\label{eq:mass_finite_def1}
\end{equation}
Note that 
$$
\|\hat h\|_2^2 \le   \| h\|_2^2 =o(N)
$$
by second part of \Cref{lem:l_p_h} and  our hypothesis on the boundary condition.
Combining this fact with \eqref{lem:mass_torus} and \Cref{mass_torus_holed} yields $m(N) \to m(\theta)>0$ as $N \to \infty$. In particular, $m_{N}$ is bounded away from  0 and $\infty$ uniformly in $N$.
Furthermore for this choice of $m_N$, and also from the fact that $\norm{h}_{2}^{2}$ is $o(N)$,
\begin{equation}
   B^2_{N} := \sum_{i=1}^k\frac1{(\lambda_i+m_{N})^2} + \sum_{i=1}^k\frac{2|\hat h_i|^2}{\lambda_i+m_{N}} {= \sum_{i=1}^k \frac1{(\lambda_i+m_{N})^2}(1+o(1))}. \label{eq:var_bound}
\end{equation}
Indeed \eqref{eq:var_bound} follows from \Cref{lem:torus_holed_fourth} (which tells us that $\sum_{i=1}^k(\lambda_i+m_{N})^{-2}$ is $\Omega(N)$) and 
\[
(4d+m_{N})^{-1}\cdot \norm{h|_U}_{2}^{2}\leq \sum_{i=1}^{k}\frac{|\hat{h}_{i}|^{2}}{\gl_{i}+m_{N}}\leq m_{N}^{-1}\norm{h|_U}_{2}^{2} \label{eq:BN_bound}
\]
Perform a change variables $w := B_{N,\sbeta}^{-1} (u-A_{N,\sbeta})$, so that the integral \eqref{eq:num_local} becomes
\begin{equation*}
    \int_{-B_{N,\sbeta}^{-1} A_{N,\sbeta}}^{B_{N,\sbeta}^{-1}(A_N - A_{N,\sbeta}-\zeta(y^\beta))} \exp(\theta m_{N} (B_{N,\sbeta}w+A_{N,\sbeta}))p_N(B_{N,\sbeta}w+A_{N,\sbeta})B_{N,\sbeta}dw
\end{equation*}
Using \Cref{eq:var_bound} and \Cref{lem:torus_holed_fourth}, it is easy to conclude that 
\begin{equation}
   B_{N,\sbeta}^{-1}(A_N - A_{N,\sbeta} -\zeta(y^\beta)) = o(1); \qquad  \frac{B_{N,\sbeta}}{B_N} = 1+o(1)  \label{eq:BNS_estimate}
\end{equation}

We now claim 
\begin{claim}\label{claim:int_approx}
\begin{multline*}
      \int_{-B_{N,\sbeta}^{-1} A_{N,\sbeta}}^{B_{N,\sbeta}^{-1}(A_N - A_{N,\sbeta}-\zeta(y^\beta))} \exp(\theta m_{N} (B_{N,\sbeta}w+A_{N,\sbeta}))p_N(B_{N,\sbeta}w+A_{N,\sbeta})B_{N,\sbeta}dw \\= \frac1{B_{N,\sbeta}\theta m_N}\exp(\theta m_N (A_N - \zeta(y^\beta))) (1+o(1)) 
     \end{multline*}
\end{claim}
The main ingredient of the proof of \Cref{claim:int_approx}, which we postpone for clarity is a local version of Lindeberg CLT for triangular arrays. Assuming \Cref{claim:int_approx} for now note that exactly same analysis can be done for the denominator in \Cref{eq:ratio_conditional} by assuming $\beta = \emptyset$. Thus $A_{N,\sbeta}$ in \eqref{claim:int_approx} can be replaced by $A_N$, there is no $\zeta(y^\beta)$ term, and $B_{N,\sbeta}$ can be replaced by $B_N$. This yields the asymptotics of the denominator as 
\begin{equation}
     \frac1{B_{N}\theta m_{N}} \exp(\theta m_{N} A_{N})\left(1+o(1)\right)\label{eq:denominator}
\end{equation}
Plugging these estimates back in \eqref{eq:ratio_conditional}, we get
\begin{multline*}
   \frac{ \bE\left(\alpha( Z, \iota)\exp(\theta M\zeta(Z^\beta))\bE_{\beta}\left(\exp(\theta m \|Z^\sbeta+\hat h^\sbeta\|_2^2)1_{\|Z^\sbeta+\hat h^\sbeta\|_2^2 \le  N -\zeta(Z^\beta)}\right)\right)}{\bE\left(\exp(\theta m \|Z+\hat h\|_2^2)1_{\|Z+\hat h\|_2^2 \le  N }\right)}\\
   =  \bE(\alpha( Z, \iota)\exp(\theta m_{N}\zeta(Z^\beta))\exp(-\theta m_{N} \zeta (Z^\beta))(1+o(1))=\bE(\alpha(Z, \iota))(1+o(1))
\end{multline*}
In the last line, we used $\bE(\alpha(Z, \iota))$ is bounded since $\fg$ is bounded and $Z$ is complex Gaussian. This completes the proof of \Cref{lem:enough} modulo \Cref{claim:int_approx}.
\end{proof}
To prove \Cref{claim:int_approx}, we need the following refined version of the Lindeberg-Feller Theorem which in addition yields uniform convergence of densities. We refer the reader to Appendix A4 of \cite{ARU} for a proof, although the result is attributed to Kolmogorov and Gnedenko \cite{KG}.

\begin{thm}\label{lem:lclt}
Let $\{Z_{i,N}\}_{i=1}^{N}$ be a triangular array of centered, independent $\bR^n$ valued random variables with density given by $\{q_{i,N}(x)\}$. Assume that the following conditions hold.
\begin{enumerate}
    \item For all $i$, $q_{i,N}$ are in $L^{r}(\bR^{n})$ where $r\in (1,2]$ is independent of $i,N$. Moreover, $\sup_{i,N}\norm{q_{i,N}}_{L^{r}(\bR)}\leq M$ for some $M>0$. 
    \item For all $i\leq N$, $\gs_{i,N}^{2}=\var Z_{i,N}$ are uniformly bounded away from $0$ for all $i,N$.  
    \item Let $B^{2}_{N}:=\sum_{i\leq N}\gs_{i,N}^{2}$.  Lindeberg's condition holds, namely for all $\eps>0$
    \[
    \frac{\sum_{i\leq N}\bE(Z_{i,N})^{2}1_{|Z_{i,N}|\geq \eps B_{N}}}{B_{N}^{2}}\to 0
    \]

\end{enumerate}
Then, if $p_{N}$ denotes the distribution of the sum $\sum_{i\leq N} Z_{i,N}$, we have that 
\[
B_{N}\cdot p_{N}(B_{N}x+A_{N})=\frac{1}{\sqrt{2\pi}}\exp\left(-\frac{x^{2}}{2}\right)+v_{N},
\]
Where $v_{N}$ decays to $0$ as $N\to \infty$ uniformly over $x\in \bR$.
\end{thm}

\begin{proof}[Proof of \Cref{claim:int_approx}]
We will apply \Cref{lem:lclt} to $\|Z^\sbeta+\hat h^\sbeta\|_2^2$ whose density is $p_{N}$. Since $Z_i$s are exponential with variance $(\lambda_i+m)^{-1} \le m^{-1}$ and $m>0$, it is straightforward to apply check that all the conditions in \Cref{lem:lclt} is verified.
Pick a large number $T$, and we break up the integral into an integral from $(-T,B_{N,\sbeta}^{-1}(A_N - A_{N,\sbeta} - \zeta(y^\beta))$ and the rest. Let $p$ denote the density of standard (real) Normal random variable.  Using \Cref{lem:lclt}, the former can be written as 
\begin{equation*}
    \exp(\theta m_N A_{N,\sbeta})\int_{-T}^{B_{N,\sbeta}^{-1}(A_N - A_{N,\sbeta}-\zeta(y^\beta))} \exp(\theta m_N B_{N,\sbeta}w)p(w)(1+o(1))dw
\end{equation*}
Observe the replacement of $B_{N}p_{N}(B_{N}x+A_{N})$ by $p(w)(1+o(1))$, does not follow from Lemma \ref{lem:lclt} for all $w$. However, note that for all $|x|\leq T$, $p(x)\geq p(|T|)>0$. Thus, on this interval we know that $B_{N}p_{N}(B_{N}x+A_{N})=p(x)(1+\frac{v_{N}}{p(|T|)})$, which justifies the above approximation.

Next, by a change of variable,
\begin{equation*}
    \int_{-T}^{B_{N,\sbeta}^{-1}(A_N - A_{N,\sbeta}-\zeta(y^\beta))} \exp(\theta m_N B_{N,\sbeta}w)p(w)dw = \int_{-TB_{N,\sbeta}}^{A_N-A_{N,\sbeta} - \zeta(y^\beta)}\exp(\theta m_N w_1) p(\frac{w_1}{B_{N,\sbeta}})\frac1{B_{N,\sbeta}}dw_1
\end{equation*}
 By dominated convergence theorem, 
\begin{equation*}
    \int_{-TB_{N,\sbeta}}^{A_N-A_{N,\sbeta} - \zeta(y^\beta)}\exp(\theta m_N w_1) p(\frac{w_1}{B_{N,\sbeta}})dw_1 = \frac1{\theta m_N}\exp(\theta m_N (A_N - A_{N,\sbeta} - \zeta(y^\beta)))(1+o(1)).
\end{equation*}
Finally, for the integral from $-B_{N,\sbeta}^{-1} A_{N,\sbeta}$ to $-T$ we need the fact 
$$
p_N(B_{N,\sbeta}w+A_{N,\sbeta})B_{N,\sbeta} \le C
$$
for some universal constant $C$ independent of everything else, which holds due to the uniform convergence of $B_{N, \sbeta} p_{N}(B_{N, \sbeta}x+A_{N,\sbeta})$. Then, since $B_{N,\sbeta}^{-1}A_{N,\sbeta} \to \infty$ (as $A_{N, \sbeta} = \Omega(N)$ and $B_{N,\sbeta} = \Omega(\sqrt{N})$) 
\begin{multline*}
   \int_{-B_{N,\sbeta}^{-1} A_{N,\sbeta}}^{-T} \exp(\theta m_N (B_{N,\sbeta}w+A_{N,\sbeta}))p_N(B_{N,\sbeta}w+A_{N,\sbeta})B_{N,\sbeta}dw \le C\frac{\exp(\theta m_N A_{N,\sbeta}-\theta m_N T B_{N,\sbeta}) }{\theta m B_{N,\sbeta}}.
\end{multline*}
So we get an upper bound
\begin{equation*}
    C\exp(\theta m A_{N,\sbeta})\frac1{\theta m_NB_{N,\sbeta}} \exp(-\theta m_N T B_{N,\sbeta})
\end{equation*}
Combining with the first estimate of \eqref{eq:BNS_estimate}, we see that the integral \eqref{eq:num_local} is
\begin{multline}
   \frac1{B_{N,\sbeta}\theta m_N} \exp(\theta m_N A_{N,\sbeta})\Big[\left(\exp(\theta m (A_N - A_{N,\sbeta} - \zeta(y^\beta)))\right)(1+o(1))\Big] \\=\frac1{B_{N,\sbeta}\theta m_N}\exp(\theta m_N (A_N - \zeta(y^\beta))) (1+o(1)) \label{eq:num_local2}
\end{multline}
as desired.
\end{proof}

Much as in the massless case, we conclude this section with the proof of the tail bound, which for $\theta<C_{d}$ is \Cref{lem:spherical_tail}. While the lemma we needed was extremely straightforward for $\theta\geq C_{d}$, the situation is more complicated here. The presence of boundary conditions and  that we require the tail bound away from the harmonic extension near the boundary requires more care. 
\begin{proof}[Proof of \Cref{lem:spherical_tail}]
For this proof, we recall the notation $\cZ_{\text{sph},N} (f, \theta,\gamma, U^{c})$ for the partition function of the spherical law with general boundary condition $f$ and general mass cutoff $\gamma N$, defined in \Cref{eq:density_spherical_boundary_general_mass}, which we recall here for convenience:
\begin{equation}
\bE(\ff(\Psisph^{f, \theta,\gamma, U^{c}})):=\frac{1}{\cZ_{\text{sph},N} (f, \theta,\gamma, U^{c})}\int_{\bC^{U^{c}}} \ff(\psi) \exp\left(-\theta \norm{\nabla \psi|_{U^{c}}}_{2}^{2}\right)\mv1_{\norm{\psi|_{U^{c}}}_{2}^{2}\leq \gamma N}1_{\psi|_{ \partial U}\equiv f_{|\partial U}}\vd \psi.\label{eq:density_spherical_boundary_general_mass2}
\end{equation}
Let $f_w$ be a function defined on $U \cup \{\mvy\}$ which equals $f$ on $U$ and equals $w \in \bC$ on $\mvy\in U^{c}$. Then we can write
$$
\bP(|\Psisph^{f, \theta, U^{c}} (\mvy) - h(\mvy)| > b) = \int_{|z| > b} \frac{\cZ_{\text{sph},N} (f_{h(\mvy)+z},\theta, 1-\frac{|h(\mvy)+z|^2}{N}, U^{c} \setminus \{\mvy\})}{\cZ_{\text{sph},N} (f, \theta, U^{c})}dz
$$
where the numerator of the integrand is taken to be 0 in case $|h(\mvy)+z|^2/N \ge 1$.  We now compute the asymptotics of the  integrand, and let us consider the numerator first. We will do an analysis similar to the one done in  \Cref{lem:llimit1}, but we need to be careful with the cancellations of the harmonic function terms as they are different in the numerator and the denominator. Choose $m_N=m$ as in \eqref{eq:mass_finite_def1} (the dependence on $N$ is suppressed here for clarity). All the Laplacians are computed in the graph induced by $U^c \cup \partial U$. Breaking up any $\psi:   U ^{c}\setminus \{\mvy\} \to \bC$ with boundary condition given by $f_w$ into $\psi_1 + h_w$ where $h_w$ is the harmonic extension of $f_w$ with mass $m$ onto $U^{c}\setminus \{\mvy\}$, we can rewrite the numerator, with $w = z+h(\mvy)$ as
\begin{multline}
    C(h_w)\int_{\bC^{U^{c} \setminus \{\mvy\}}}  \exp\left(\theta m\|\psi_1+h_w\|_2^2+\theta\la \psi_1 , (\Delta_{U^{c}\setminus \{\mvy\}}-m) \psi_1 \ra\right)\mv1_{\norm{\psi_1 + h_w}_{2}^{2}\leq  N - |w|^2}\vd \psi_1.\label{eq:tail_num1}
\end{multline}
where
\begin{equation}
    C(h_w):= \exp(\theta \sum_{\mvx \in \partial U \cup \{\mvy\}} h_w(\mvx)\overline{\Delta h_w(\mvx)})\label{eq:c1}
\end{equation}
In the above, the expression for $h_w$ simplified because of the cancellation
\begin{equation}
    \theta \la h_w, (\Delta-m)h_w\ra +  \theta m \sum_{\mvx \in \partial U \cup \{\mvy\}} |h_w(\mvx)|^2 = \theta \sum_{\mvx \in \partial U \cup \{\mvy\}} h_w(\mvx)\overline{\Delta h_w(\mvx)}.\label{eq:harmonic_cancellation}
\end{equation}
\Cref{eq:harmonic_cancellation} is easily justified since $h_w$ is massive harmonic in $U^{c} \setminus \{\mvy\}$ with mass $m$ and the fact that the integral in \eqref{eq:c1} is over $U^c \setminus \{\mvy\}$ which justifies the second term in the left hand side in \eqref{eq:harmonic_cancellation}. Let $|U^{c}|=k$ and let $(\lambda_i)_{i=1}^k,(\lambda_i')_{i=1}^{k-1}$ be the eigenvalues of the negative Dirichlet Laplacian on $U^{c}$ and $U^{c} \setminus \{\mv y\}$ respectively. 
In the same vein as \Cref{lem:fourier_law}, we can write
\eqref{eq:tail_num1} as 
\begin{equation*}
    C_1(h_w) \bE(\exp(\theta m \|Z'+\hat h_w\|_2^2 1_{\|Z'+\hat h_w\|_2^2 \le  N - |w|^2})).
\end{equation*}
where $(Z'_1,\ldots, Z'_{k-1})$ are independent and $Z'_i$ is a complex Gaussian with variance $(2\theta(\lambda'_i+m))^{-1}$ and 
$$
C_1(h_w) = \prod_{i=1}^{k-1}\frac{2\pi}{\theta(\lambda_i'+m)} \exp(\theta\sum_{\mvx \in \partial U \cup \{\mvy\}} h_w(\mvx)\overline{\Delta h_w(\mvx)})
$$
Similarly the denominator can also be written in terms of complex Gaussians, and in fact we have already analyzed its asymptotics in \eqref{eq:denominator}, which are given by 
\begin{equation}
    \cZ_{\text{sph},N} (f, \theta, U) = C_1(h)\frac1{B_{N}\theta m} \exp(\theta m  N)\left(1+o(1)\right) 
\end{equation}
where $B_N$ is as in the proof of \Cref{lem:llimit1} and 
$$
C_1(h)=\prod_{i=1}^{k}\frac{2\pi}{\theta(\lambda_i+m)} \exp(\theta\sum_{\mvx \in \partial U } h(\mvx)\overline{\Delta h(\mvx)}).
$$
Let us now compute the asymptotics of the numerator and compare. Let
$$
A_N' = \bE(\|Z'+\hat h_w\|_2^2);\qquad (B_N')^2 = \text{Var}(\|Z'+\hat h_w\|_2^2)
$$
Then, letting $p_N'$ to be the density of  $\|Z'+\hat h_w\|_2^2$,
\begin{align}
     &\bE(\exp(\theta m \|Z'+\hat h_w\|_2^2 1_{\|Z'+\hat h_w\|_2^2 \le  N - |w|^2}))\\
     = &\int_0^{ N-|w|^2}e^{\theta m u}p'_N(u)du\nonumber\\
     =&\int_{-A_N'/B_{N}'}^{(B_N')^{-1}( N-|w|^2-A_N')}\exp(\theta m (B_N'w+A_N'))p'_N(B_N'w+A_N')B_N'dw\nonumber\\
     \le & C\frac{\exp(\theta m( N - |w|^2))}{B_N'\theta m}\label{eq:uniform_upper}
\end{align}
where we used again the uniform bound on the density $p_N$.
Now we need to take care of the harmonic function terms.  We can break up $h_w  = h + g$ where $g$ is the massive harmonic extension of the function taking value $w - h(\mvy)=z$ in $\mvy$ and 0 on $\partial U$. Then it is straightforward computation to see
\begin{equation}
    \sum_{\mvx \in \partial U \cup \{\mvy\}} h_w(\mvx)\overline{\Delta h_w(\mvx)} -\sum_{\mvx \in \partial U } h(\mvx)\overline{\Delta h(\mvx)}  = \sum_{\mvx \in \partial U} f(\mvx)\overline{\Delta g(\mvx)} + h_w(\mvy)\overline{\Delta h_w(\mvy)}.\nonumber
\end{equation}
We deal with the two terms above separately. Let $(X^{\mvu}_t)_{t \ge 0}$ be a simple random walk on $G_U$ started at $\mvu\in U^{c}$ which is killed at a vertex $\mvu$ with probability $m/(\deg(\mvu)+m)$ and   $\tau$ is the first time when $(X^{\mvx'}_t)_{t \ge 0}$ either gets killed or hits $\partial U \cup \{\mvy\}$. Using the killed random walk description of the massive harmonic function,
Now we can write
\begin{equation}
   \sum_{\mvx \in \partial U} f(\mvx)\overline{\Delta g(\mvx)} = \sum_{\mvx \in \partial U} f(\mvx) \sum_{\mvx' \sim \mvx}\overline{g(\mvx')} = \bar z\sum_{\mvx \in \partial U} f(\mvx) \sum_{\mvx' \sim \mvx}\bP(X_{\tau}^{\mvx'} = \mvy)
\end{equation}
By the reversibility of the killed walk, we see that 
    $$
    \sum_{\mvx' \sim \mvx}\bP(X_{\tau}^{\mvx'} = \mvy) =  \sum_{\mvy' \sim \mvy}\bP(X^{\mvy'}_{\tau} = \mvx) .
$$
Therefore
\begin{equation*}
     \sum_{\mvx \in \partial U} f(\mvx)\overline{\Delta g(\mvx)}=\bar z\sum_{\mvx \in \partial U} f(\mvx) \sum_{\mvy' \sim \mvy}\bP(X^{\mvy'}_{\tau} = \mvx) 
\end{equation*}
Since $h_w$ is massive harmonic in $U^c$, 
\begin{equation*}
    \sum_{\mvx \in \partial U} f(\mvx) \bP(X^{\mvy'}_{\tau} = \mvx) = h_{w}(\mvy') - w\bP(X^{\mvy'}_\tau=w).
\end{equation*}
Therefore, we simplify to
\begin{align*}
    \sum_{x \in \partial U} f(\mvx)\overline{\Delta g(x)}  &= \bar z\sum_{\mvy' \sim \mvy}( h_{w}(\mvy') - w\bP(X^{\mvy'}_\tau=\mvy))\\
    & = \bar z\left[ \Delta h_w(\mvy)+\deg(\mvy)h_w(\mvy) -  (\Delta g(\mvy) + \deg(\mvy)g(\mvy)) - h(\mvy)\sum_{\mvy' \sim \mvy}\bP(X^{\mvy'}_\tau=w) \right]\\
    &=\bar z \left[ m h(\mvy)+\deg(\mvy)h(\mvy)  - h(\mvy)\sum_{\mvy' \sim \mvy}\bP(X^{\mvy'}_\tau=w)\right].\\
    &=m\bar zh(\mvy)-h(\mvy)\overline{\Delta g(\mvy)}
\end{align*}
Now for the other term,
\begin{align*}
    h_w(\mvy)\overline{\Delta h_w(\mvy)}&=w(\overline{\Delta h(\mvy)}+\overline{\Delta g(\mvy)})\\
    &=w(m\overline{h(\mvy)} + \overline{\Delta g(\mvy)})\\
    &=mz\overline{h(\mvy)}  + m|h(\mvy)|^2 + w\overline{\Delta g(\mvy)}
\end{align*}
Combining, we get
\begin{align}
    &\sum_{\mvx \in \partial U \cup \{\mvy\}} h_w(\mvx)\overline{\Delta h_w(\mvx)} -\sum_{\mvx \in \partial U } h(\mvx)\overline{\Delta h(\mvx)} \\
    &=m\bar zh(\mvy)-h(\mvy)\overline{\Delta g(\mvy)} + mz\overline{h(\mvy)}  + m|h(\mvy)|^2 + w\overline{\Delta g(\mvy)}\\
    &=2m\fR(\bar z h(\mvy)) + m|h(\mvy)|^2+z\overline{\Delta g(\mvy)}
\end{align}
To combine this with \eqref{eq:uniform_upper}, we need to subtract $m|w|^2$ from the last expression above. This simplifies to
\begin{align*}
   2m\fR(\bar z h(\mvy)) + m|h(\mvy)|^2+z\overline{\Delta g(\mvy)}-m|w|^2
   =&z\overline{\Delta g(\mvy)} - m|z|^2\\
   =&|z|^2 (\sum_{\mvy' \sim \mvy}\bP(X^{\mvy'}_\tau = \mvy)-\deg(\mvy)-m)\\
   =&-|z|^2(\deg(\mvy)+m) (1-\bP(X^{\mvy}_{\tau_+})) \\
   =&-|z|^2(\deg(\mvy)+m)({\sf G}^{m, U}(\mvy, \mvy))^{-1}\\
   :=&-c_1|z|^2
\end{align*}
where ${\sf G}^{m, U^{c}}(\mvy, \mvy)$ is the Green's function for the killed random walk.
Overall, we get 
\begin{align*}
  &{\cZ_{\text{sph},N} \biggl(f_{h(\mvy)+z},\theta, 1-\frac{|h(\mvy)+z|^2}{N}, U^{c} \setminus \{\mvy\}\biggr)}/{\cZ_{\text{sph},N} (f, \theta, U^{c})} \\ &\le \frac{C \theta}{2\pi }\frac{B_N}{B_N'}\frac{ C\exp(\theta m N }{ \exp(\theta m  N)\left(1+o(1)\right)}\exp(-c_1|z|^2)  \frac{\prod_{i=1}^{k}(\lambda_i+m)}{\prod_{i=1}^{k-1}(\lambda'_i+m)}
\end{align*}
We finally claim 
$$
\frac{B_N'}{B_N} = O(1) \text{ and }\frac{\prod_{i=1}^{k}(\lambda_i+m)}{\prod_{i=1}^{k-1}(\lambda'_i+m)} = O(1).
$$
For this, observe that $\gD_{U^{c}\setminus \{\mvy\}}$ is a principal sub matrix of $\gD_{U^{c}}$ of dimension $|U^{c}|-1$, and thus we have the following interlacing relation for the eigenvalues $\forall i\in \{1,2,\ldots |U|-1\} $:
\[\gl_{i}\leq \gl_{i}'\leq\gl_{i+1}.\]
For any bounded monotonic function $\fr:[0,4d]\to \bR$, note that 
\begin{align*}
|\sum_{i=1}\fr(\gl_{i})-\sum_{i=1}\fr(\gl_{i}')|&=|\fr(\gl_{|U|})+\sum_{i}\fr(\gl_{i})-\fr(\gl'_{i})|\\
&\leq |\fr(\gl_{|U|})|+|\sum_{i}(\fr(\gl_{i}')-\fr(\gl_{i}))|\leq 2\|{\fr}\|_{\infty}.
\end{align*}
We apply this to the functions $\log(x+m)$ and $(x+m)^{-2}$. Thus, we have that 
\begin{align}
|\sum_{i}\log (\gl_{i}+m)-\sum_{i}\log (\gl_{i}+m)|\leq 2\log (4d+m),
\end{align}
and 
\begin{align}
|\sum_{k}\frac{1}{(\gl_{i}+m)^{2}}-\sum_{k}\frac{1}{(\gl_{i}'+m)^{2}}|\leq \frac{2}{m^{2}}. 
\end{align}
What this yields is that 
\[
|\log \frac{\prod_{i=1}^{k}(\lambda_i+m)}{\prod_{i=1}^{k-1}(\lambda'_i+m)}| \leq 2 \log(m+4d), 
\]
and in particular, 
\[
(m+4d)^{-2}\leq \frac{\prod_{i=1}^{k}(\lambda_i+m)}{\prod_{i=1}^{k-1}(\lambda'_i+m)}\leq (m+4d)^{2}. 
\]
As for the ratio $B_{N}/B_{N}'$, note that 
\[
\frac{B_{N}}{B_{N}'}=\frac{\sum\frac{1}{(\gl_{k}+m)^{2}} +\sum\frac{|\hat{h}|^{2}}{(\gl_{k}+m)}}{\sum\frac{1}{(\gl_{k}'+m)^{2}} +\sum\frac{|\hat{h}'|^{2}}{(\gl_{k}'+m)}} \leq \frac{\sum_{k}(\gl_{k}+m)^{-2}+m^{-1}\norm{h}_{2}^{2}}{\sum_{k}(\gl_{k}+m)^{-2}-2m^{-2}+(4d+m)^{-1}\norm{h^{w}}_{2}^{2}}\leq 2.
\]

The last inequality holds for $N$ sufficiently large. Combined, these yield the Gaussian upper bound (independent of $h$)
$$
 \frac{\cZ_{\text{sph},N} (f_{h(\mvy)+z},\theta, 1-\frac{|h(\mvy)+z|^2}{N}, U^{c} \setminus \{\mvy\})}{\cZ_{\text{sph},N} (f, \theta, U^{c})}  \le Ce^{-c_1|z|^2}
$$
for some $C,c>0$ independent of everything else. Integrating from $b$ to $\infty$, we get the required tail bound.
\end{proof}
\section{Future work and open questions}\label{sec:future}
Over the course of working on this article, we came across several interesting questions which are yet to be answered. 
\begin{enumerate}
\item \textbf{Non mixture in supercritical:} If $\nu > \nu_c$, we anticipate that when $\mathscr{M}(\theta,\nu)=\{a_{*}\}$ for some $a_{*}>0$, and that in particular, $\Xi=a_{*}$ almost surely. The local limit should not be a non trivial mixture of Gaussians in the super critical regime. This requires more information on the function $I$, in particular properties of higher derivatives.

\item \textbf{Mixture in Critical:}  For $\theta\leq C_{d}$, there is discontinuity in the phase transition in the following sense. \begin{lemma}
Let $\theta\leq C_{d}$, and $\nu=\nu_{c}$. Then $\mathscr{M}(\theta, \nu)=\{0\}\sqcup \mathscr{M}'$ where $\mathscr{M}'$ is non empty and bounded away from zero.  
\end{lemma}
\begin{proof}
A key point to the argument here is the fact that the function $I(\cdot )$ is identically $0$ on the interval $[0,R_{p}]$, this will be used repeatedly. If $0\notin \mathscr{M}'$, then there exists $x\in [0,a_{M}]$ such that $F_{\theta,\nu}(x)<F_{\theta,\nu}(0)$. Since $F(x)=W(\theta(1-x))+\theta\nu^{-1}I(\nu x)$, by continuity, there exists $\nu'<\nu$ such that $F_{\theta,\nu'}(x)>F_{\theta,\nu}(x)=F_{\theta,\nu'}(x)$. Thus, this tells us that $\nu>\nu_{c}$, a contradiction. Now, suppose $\mathscr{M}(\theta,\nu)=\{0\}$. Then let $F_{**}(\nu)=\min_{x\geq R_{p}/2\nu_{c}} F(x)$. $F_{**}(\nu)$ is continuous in $\nu$ by definition. Note, $F_{**}(\nu)>F(0)$. In particular, there exists a sufficiently small perturbation $\nu''>\nu$ such that $F_{**}>F(0)$. Since changing $\nu$ does not alter the behavior of $F$ on the interval $[0,R_{p}/2\nu_{c}]$, $0$ remains the sole minimizer of $F_{\theta,\nu''}$, implying that $\nu<\nu_{c}$, a contradiction. Finally, $0$ is isolated since $W(\theta(1-x))$ is strictly increasing for all $x>0$.
\end{proof}

We expect $\mathscr{M}'$ to be a singleton, however we currently lack certain properties of the function $I$ required to prove this. In addition, when $\theta<C_{d}$, we have $F'(0)>0$ while $F'(x)=0$ for all other minimizers. This has an important consequence, and the intuition provided by the saddle point method suggest that the local limit should be the same as that of the super critical case where mass is lost. It seems that a non trivial mixture can be observed when $\theta=C_{d}$, however this would require $\mathscr{M}'$  to be finite, and for the second derivatives of $F$ at all $x\in \mathscr{M}'$ to be positive. We currently lack this. In addition, we would require a refinement of the rate of error in the convergence of the free energy, the current rate is inadequate for the analysis by Laplace asymptotics. It is unclear whether the same discontinuity of phase transition holds for $\theta\geq C_{d}$.

\item \textbf{Fluctuations near $U$ for $\nu>\nu_{c}$: }In this article, we have proved that the fluctuations of $\Psi_{N}^{\theta,\nu}$ are $O(1)$ and asymptotically Gaussian far away from $U$. However, we anticipate that these properties should hold in the vicinity of $U$ as well. We reserve this as a direction of future research. 

\item \textbf{Same result for $p\leq 4$:} When $p\leq 4$, notice that the dyadic argument breaks down because there is no threshold beyond which the non linear part can be ignored. We can still prove the required $\ell^{\infty}$ estimate to start the argument, but cannot bootstrap it down enough to drown out the non linearity. New additions to our techniques are likely required. 
\end{enumerate}
\subsection{Acknowledgments:} We would like to thank Juhan Aru, Partha Dey, Kay Kirkpatrick and Aleksandra Korzhenkova for useful discussions. KK would like to thank Marek Biksup for an enlightening discussion on the relation of the model under consideration to the self attracting random walk, and the Salmon Coast Field Station for hosting him while a portion of this work was completed. 

\bibliographystyle{abbrv}
 \bibliography{nls}
\end{document}